\documentclass[sn-mathphys]{sn-jnl}%

\usepackage[labelfont=bf, justification=justified]{caption}

\usepackage{nicefrac}
\usepackage{mathtools}
\usepackage{graphicx}
\usepackage[caption=false]{subfig}
\usepackage{booktabs}
\usepackage{rotating}
\usepackage{microtype}
\usepackage{bm}
\DeclareMathOperator{\Span}{span}

\DeclareMathOperator{\diag}{diag}
\DeclareMathOperator{\Range}{range}

\jyear{2021}%

\theoremstyle{thmstyleone}%
\newtheorem{theorem}{Theorem}%
\newtheorem{proposition}[theorem]{Proposition}%
\newtheorem{lemma}{Lemma}
\theoremstyle{thmstylethree}%
\newtheorem{definition}{Definition}%
\newtheorem{remark}{Remark}%

\begin{document}

\title[Automatic coarsening in AMG with matching-based aggregations]{Automatic coarsening in Algebraic Multigrid utilizing quality measures for matching-based aggregations}

\author[1]{\fnm{Pasqua} \sur{D'Ambra}}\email{pasqua.dambra@cnr.it}
\equalcont{These authors contributed equally to this work.}

\author*[2,1]{\fnm{Fabio} \sur{Durastante}}\email{fabio.durastante@unipi.it}
\equalcont{These authors contributed equally to this work.}

\author[3,1]{\fnm{Salvatore} \sur{Filippone}}\email{salvatore.filippone@uniroma2.it}
\equalcont{These authors contributed equally to this work.}

\author[4]{\fnm{Ludmil} \sur{Zikatanov}}\email{ludmil.math@gmail.com}
\equalcont{These authors contributed equally to this work.}

\affil[1]{\orgdiv{Institute for Applied Computing ``Mauro Picone'' (IAC)}, \orgname{Consiglio Nazionale delle Ricerche}, \orgaddress{\street{Via Pietro Castellino 111}, \city{Naples}, \postcode{80131}, \state{NA}, \country{Italy}}}

\affil*[2]{\orgdiv{Dipartimento di Matematica}, \orgname{Università di Pisa}, \orgaddress{\street{Largo Bruno Pontecorvo, 5}, \city{Pisa}, \postcode{56127}, \state{PI}, \country{Italy}}}

\affil[3]{\orgdiv{Department of Civil and Computer Engineering}, \orgname{University of Rome ``Tor Vergata''}, \orgaddress{\street{Via Politecnico 1}, \city{Rome}, \postcode{00133}, \state{RM}, \country{Italy}}}

\affil[4]{\orgdiv{Department of Mathematics}, \orgname{The Pennsylvania State University}, \orgaddress{\street{McAllister Building, Pollock Rd}, \city{State College}, \postcode{16802}, \state{PA}, \country{USA}}}

\abstract{In this paper, we discuss the convergence of an Algebraic MultiGrid (AMG) method for general symmetric positive-definite matrices. The method relies on an aggregation algorithm, named {\em coarsening based on compatible weighted ma\-tch\-ing}, which exploits the interplay between the principle of compatible relaxation and the maximum product matching in undirected weighted graphs. The results are based on a general convergence analysis theory applied to the class of AMG methods employing unsmoothed aggregation and identifying a quality measure for the coarsening; similar quality measures were originally introduced and applied to other methods as tools to obtain good quality aggregates leading to optimal convergence for M-matrices. 
The analysis, as well as the coarsening procedure, is purely algebraic and, in our case, allows an \emph{a posteriori} evaluation of the quality of the aggregation procedure which we apply to analyze the impact of approximate algorithms for matching computation and the definition of graph edge weights. We also explore the connection between the choice of the aggregates and the compatible relaxation convergence, confirming the consistency between  theories for designing  coarsening procedures in purely algebraic multigrid methods and the effectiveness of the coarsening based on compatible weighted matching. We discuss various completely automatic algorithmic approaches to obtain aggregates for which good convergence properties are achieved on various test cases.}



\maketitle

\section{Introduction}
\label{sec:intro}

We assess here the convergence 
of a MultiGrid method (MG) for the solution of linear systems of the form 
\begin{equation}\label{eq:sys}
A \mathbf{u} = \mathbf{f},
\end{equation}
on the finite-dimensional linear vector space $V$ equipped with an inner product $( \cdot, \cdot)$, where $A: V \rightarrow V'$ is symmetric positive definite (SPD), $\mathbf{f} \in V'$ and $V'$ is the dual of $V$; by the Riesz representation theorem $V'$ can be identified with $V$. More specifically, we focus on a recently proposed method belonging to the class of Algebraic MultiGrid Methods (AMG) with \emph{unsmoothed aggregation} (UA-AMG) or {\em plain aggregation}~\cite{N2010,XuZikatanovReview}. These can be seen as  particular instances of a general stationary linear iterative method for solving~\eqref{eq:sys} 
\begin{equation}\label{eq:gitm}
\mathbf{u}^m=\mathbf{u}^{m-1}+B(\mathbf{u}-A\mathbf{u}^{m-1}), \quad m=1,2,\ldots; \quad \text{ given } \mathbf{u}^0 \in V,
\end{equation}
where $B: V' \rightarrow V$ is a linear operator which can be interpreted  as an approximate inverse of~$A$. An AMG method, or indeed any MG, is based on the recursive use of a two-grid scheme combining the action of a \emph{smoother}, i.e., a convergent iterative method, and a \emph{coarse-grid correction}, which corresponds to the solution of the residual equations on a coarser grid. In completely general terms, the guiding design principle of an AMG is the  optimization of the choice of coarse space for a given smoother. The most commonly used smoothers are the splitting-based methods, such as the Gauss--Seidel method and the (modified or scaled) Jacobi method. 

As usual in the MG context, the final objective of any analysis is to achieve uniform convergence with respect to the problem size (optimal convergence). Unfortunately, this is a property that can normally be established only for the two-level AMG (TL-AMG); it is  very rarely extended to the multilevel case when  no ``geometric'' information on the matrix $A$ is available. 
Our task is then to ensure the selection of an appropriate set of aggregates, i.e., the disjoint sets of fine grid unknowns to which the coarse grid unknowns are associated, to guarantee a \emph{fast} convergence at a reasonable cost per iteration. Of the many possible ways of achieving such a result, we narrow down our investigation to the case of UA-AMG; see~\cite{firstamg1,firstamg2} for the first works in this direction. Within this framework, we are going to exploit the unifying theory outlined in the review~\cite{XuZikatanovReview} to assess convergence and to investigate and characterize the quality of the coarse spaces generated by means of the aggregation procedure introduced in~\cite{DV2013,BootCMatch}. The latter is a technique based on the use of matching algorithms for edge-weighted graphs~\cite{suitor,preis,matchingduffhsl} that aims to achieve a purely algebraic and automatic approach for the solution of~\eqref{eq:sys}, with no further assumption on the SPD system matrix, and independently of any user defined parameter. Indeed, this approach fits within a trend of similar algebraic techniques, e.g., those based on path-covering algorithms~\cite{HuLinZikatanovPathCover}, or on the use of matching to generate multilevel hierarchies for graph Laplacians relative to coarse subspaces in finite elements applications~\cite{XuZikatanovAdaptiveAggregation}, striving for purely algebraic aggregation procedures that are adaptive in nature and allow for an \emph{a posteriori} analysis of the quality of the generated coarse spaces.

We observe that, as reported in~\cite[Section~8.5, Section~9.5]{XuZikatanovReview}, the general convergence theory we specialized in this paper for the aggregation based on matching in weighted graphs, was originally designed for the AGMG method in~\cite{N2010,NN2011} and extended in~\cite{NN2012,N2012}, to obtain AMG methods based on unsmoothed aggregation with a user-defined bound on the convergence rate. In~\cite{NN2012} the authors show that, for the class of nonsingular symmetric M-matrices with nonnegative row sum, if the aggregates can be built in such a way that a meaningful local bound is fulfilled, the resulting multilevel methods employing an appropriate AMLI cycle~\cite{VassilBook} shows an optimal convergence with a guaranteed convergence rate. The theory is extended to nonsymmetric M-matrices for a TL-AMG in~\cite{N2012}. In~\cite{XuZikatanovReview} the theory is again extended to more general SPD matrices and formalized as an abstract framework for the setup of coarsening methods. 

We finally note that the need to define local measures to assess the quality of a coarse space also led to the introduction of the notion of {\em compatible relaxation}.  
Compatible relaxation, first defined by Brandt in~\cite{BrandtCompRelax}, as {\em a modified relaxation scheme that keeps the coarse-level variables invariant}, was originally based on the idea to use a smoother to detect slowly converging components.
This principle has been largely applied to define a general procedure for
coarsening, both for selecting coarse variables and to adapt the prolongators in
adaptive AMG (see, e.g.,~\cite{L2004,FalgoutVassilevskyMeasure,BF2010,BBKL2011}).
It was a basic guideline for the formulation of our coarsening and
of its application in a bootstrap AMG based on composition of multiple AMG
hierarchies~\cite{DV2013,BootCMatch}. In our coarsening method, since we
explicitly define the complementary space to the coarse space, we can apply a
smoother to the only-fine variables and infer the quality of the coarse space by
an estimate of the corresponding convergence rate. Our experiments show the
coherency between the aggregation quality measure based on the general theory
in~\cite{XuZikatanovReview}, which has the advantage to be independent of the
smoother and only depends on the way we build aggregates,
and the quality measure based on the compatible relaxation.

	The main contributions of this paper can be summarized as follows. 
	\begin{itemize}
	    \item We prove that the automatic aggregation-based coarsening, relying on maximum weight matching in graphs equipped with a suitable choice of edge weights, fulfills all the conditions to have a bounded convergence rate of the corresponding TL-AMG for any SPD matrix. 
	    \item  We show how the resulting quality measure for the aggregation can be used to drive the choice of different (approximate) matching algorithms and of the edge weights in the adjacency graph of the system matrix, without resorting to heuristics and \emph{a priori} information on the near kernel of the matrix.
	    \item We emphasize the connection between the choice of the aggregates and the compatible relaxation principle for the new coarsening, confirming the consistency between the currently available theories for general coarsening in AMG.
	\end{itemize} 

The remainder of this paper is organized as follows: to begin with, in Section~\ref{sec:a-convergence-theory} we introduce a quality measure for a general UA-AMG in terms of the unifying theory from~\cite{XuZikatanovReview}. Then, in Section~\ref{sec:ourconvergence} we reintroduce the UA-AMG from~\cite{DV2013,BootCMatch} and specialize the convergence theory and the quality measure for the aggregates from the previous section to this case. Section~\ref{sec:numerical_experiments} is entirely  devoted to the application of the theory to some standard benchmarks; specifically, we investigate how the various matching algorithms applied for obtaining the aggregates influence their quality. Section~\ref{sec:compatiblerel} shows the coherency between the quality of aggregates and the convergence ratio of a convergent smoother applied to the effective smoother space, i.e., to the complementary space to the coarse space. Section~\ref{sec:concl}
summarizes conclusions.

\section{Convergence theory for TL-AMG algorithms and quality measure for aggregates}
\label{sec:a-convergence-theory}

The measure of the quality of the aggregates, and thus of the coarse space, for a given TL-AMG algorithm we are interested in depends both on the convergence ratio achieved by the resulting method and on the cost needed for defining and applying the multigrid hierarchy. To set the notation, and the context in which we are performing our analysis, let us  briefly recall the components of a TL-AMG method, i.e.:
\begin{itemize}
	\item a {\em convergent smoother}, $R: V' \rightarrow V$;  
	\item a {\em coarse space} $V_c$; this is either a subspace of $V$ or more generally a space with a smaller dimension than $V$. It is always linked to $V$ via a prolongation operator $P: V_c \rightarrow V$;
	\item a {\em coarse space solver}, $B_c: V'_c \rightarrow V_c$;
\end{itemize}
and how these components are related to its convergence properties. We follow the approach discussed in~\cite{XuZikatanovReview} that permits to analyze the convergence properties of a multigrid algorithm in a general way. To this end, we need to introduce the inner product 
\begin{equation*}
(\mathbf{u},\mathbf{v})_{\overline{R}^{-1}}=(\overline{T}^{-1}\mathbf{u},\mathbf{v})_A=(\overline{R}^{-1}\mathbf{u},\mathbf{v}), \;  \overline{T}=\overline{R}A, \text{ and }\overline{R}=R'+R-R'AR,
\end{equation*}
together with the accompanying norm $\| \cdot \|_{\overline{R}^{-1}}$, where $R'$ is the adjoint operator of $R$ and $\overline{R}$ is called the \emph{symmetrized} operator of $R$. We assume, moreover, that $\overline{R}$ is SPD, which implies that the smoother $R$ is always convergent and such that
\begin{equation*}
\|\mathbf{v}\|^2_A \leq \|\mathbf{v}\|^2_{\overline{R}^{-1}}.
\end{equation*}
The restriction of \eqref{eq:sys} to the coarse space is then expressed as
\begin{equation*}
A_c \mathbf{u}_c = \mathbf{f}_c
\end{equation*}
where 
\[
A_c=P'AP, \; \; \mathbf{f}_c=P'\mathbf{f}, \; \; \text{with} \; P' \; \; \text{adjoint operator of} \; P.
\]

For the sake of the analysis, the coarse space solver $B_c$ is often chosen to be the exact solver, namely $B_c=A_c^{-1}$, however, we should distinguish between an exact TL-AMG and an inexact TL-AMG  when $B_c$ is only an approximation of $A_c^{-1}$. Given $\mathbf{g} \in V'$, a TL-AMG operator $B$, defined by the above components is described in Algorithm \ref{tlamg-alg}. The corresponding error propagation operator $E=I-BA$ is $E=(I-RA)(I- \Pi_c)$, where $\Pi_c=PA_c^{-1}P'A$ is the orthogonal projection on $V_c$.

\begin{algorithm}[h]
	{\bf Data} $A$: matrix, $R$: convergent smoother, $P$: prolongator, $B_c$: coarse solver, $\mathbf{g}$: arbitrary vector in $V'$\\
	{\bf Result} {B$\mathbf{g}$: preconditioned vector}\\
	Coarse grid correction: $\mathbf{w}:=PB_cP'\mathbf{g}$\\
	Post-smoothing: $B\mathbf{g}:=\mathbf{w}+R(\mathbf{g}-A\mathbf{w})$\\
	\small{\caption{\emph{Two-level post-smoothed MG}\label{tlamg-alg}}}
\end{algorithm}

We can now explore the connection between the TL-AMG convergence rate and the selection of the coarse spaces. Let us consider the prolongation operator $P$, used in representing the operator $\Pi_c$; in our case,  $P$ will be  a piecewise constant prolongation, a very common choice. This means 
that the coarse grid correction computed on the residual equation will be transferred back to the fine grid by assigning the same value to all fine grid variables associated with a given coarse variable. 

A common alternative to this choice is to smooth out the prolongator $P$ by means of a number of smoothing iterations applied to a piecewise constant {\em tentative prolongator}; this choice  gives rise to the popular class of AMG algorithms with \emph{smoothed} aggregation~\cite{VMB1996,VassilBook,XuZikatanovReview}, but they are out of the scope of the present analysis.

We assume now that there exists a sequence of spaces $V_1, V_2, \ldots, V_J$, which are not necessarily subspaces of the vector space $V$, and that  each of them is related to the original space $V$ by a linear operator
\begin{equation}\label{eq:projector}
\Pi_j: V_j \rightarrow V.
\end{equation}
We are moreover assuming that $V$ can be written as a sum of subspaces as
\[
V= \sum_{j=1}^J \Pi_jV_j.
\]
Let $\underline{W} = V_1 \times V_2 \times \ldots \times V_J$, with the inner product
\[
(\underline{\mathbf{u}}, \underline{\mathbf{v}}) = \sum_{j=1}^J (\mathbf{u}_j, \mathbf{v}_j),
\]
with $\underline{\mathbf{u}}=(\mathbf{u}_1, \ldots, \mathbf{u}_J)^T$ and $\underline{\mathbf{v}}=(\mathbf{v}_1, \ldots, \mathbf{v}_J)^T$. Let also  $\Pi_W : \underline{W} \rightarrow V$ be the operator:
\begin{equation}\label{eq:aggregator_restrictor}
\Pi_W \mathbf{\underline{u}} = \sum_{j=1}^J \Pi_j \mathbf{u}_j, \ \ \forall \mathbf{\underline{u}} \in \underline{W}.
\end{equation}
We can then write
\[
\Pi_W=(\Pi_1, \ldots, \Pi_J) \ \text{and} \ \Pi'_W=(\Pi'_1, \ldots, \Pi'_J)^T.
\]
We assume that  for each $j$ there is an operator $A_j: V_j \rightarrow V'_j$ which is symmetric positive semi-definite, and we define $\underline{A}_W: \underline{W} \rightarrow \underline{W}'$ as follows:
\[
\underline{A}_W = \diag(A_1, A_2, \ldots, A_J).
\] 
We also assume that for each $j$  there is a SPD operator $D_j: V_j \rightarrow V'_j$, and define $\underline{D}: \underline{W} \rightarrow \underline{W}'$ as follows:
\[
\underline{D} = \diag(D_1, D_2, \ldots, D_J).
\]
We associate a coarse space with each $V_j$: $V_j^c \subset V_j$, and consider the corresponding orthogonal projection $Q_j: V_j \rightarrow V_j^c$ with respect to $(\cdot, \cdot)_{D_j}$. We define $\underline{Q}: \underline{W} \rightarrow \underline{W}'$ by $ \underline{Q}=\diag(Q_1, \ldots, Q_J)$.

Let us assume the following hold:
\begin{itemize}
	\item For all $\underline{\mathbf{w}} \in \underline{W}$:
	\begin{equation}\label{eq:assumption1}
	\|\Pi_W \underline{\mathbf{w}} \|_D^2 \leq C_{p,2} \| \underline{\mathbf{w}} \|_{\underline{D}}^2  
	\end{equation}
	for some positive constant $C_{p,2}$  independent of $\mathbf{w}$;
	\item For each $\mathbf{w} \in V$, there exists $\underline{\mathbf{w}} \in \underline{W}$ such that $ \mathbf{w}= \Pi_W \underline{\mathbf{w}}$ and
	\begin{equation}\label{eq:assumption2}
	\| \underline{\mathbf{w}} \|_{\underline{A}_W}^2 \leq C_{p,1} \| \mathbf{w} \|_A^2, 
	\end{equation}
	for some  positive constant $C_{p,1}$ independent of $\mathbf{w}$;
	\item For all $j$
	\begin{equation}\label{eq:assumption3}
	N(A_j) \subset V_j^c,
	\end{equation}
	where $N(A_j)$ is the kernel of $A_j$.
\end{itemize}
The above assumptions imply that if $\mathbf{w} \in N(A)$, then $\underline{\mathbf{w}} \in N(A_1) \times \ldots \times N(A_J)$.
We define the global coarse space $V_c$ by 
\begin{equation}\label{eq:coarsespace}
V_c= \sum_{j=1}^J \Pi_j V_j^c.
\end{equation}
Furthermore, for each coarse space $V_j^c$, we define:
\begin{equation}\label{eq:mujfactor}
\mu_j^{-1}(V_j^c) = \max_{\mathbf{v}_j \in V_j} \min_{\mathbf{v}_j^c \in V_j^c} \frac{\|\mathbf{v}_j - \mathbf{v}_j^c\|_{D_j}^2}{\|\mathbf{v}_j\|_{A_j}^2}.
\end{equation}
In the context of linear algebraic problems arising from finite elements methods, these are usually named the local Poincar\'e constants (see, e.g., \cite[Section~1.5]{WathenFEMBook}); finally we define 
\begin{equation}\label{eq:finalcrate}
\mu_c=\min_{\ 1 \leq j \leq J} \mu_j(V_j^c),
\end{equation}
which is finite thanks to assumption~\eqref{eq:assumption3}.

By TL-AMG convergence theory, if $D_j$ provides a convergent smoother, then $(1- \mu_j^{-1}(V_j^c))$ is the convergence rate for TL-AMG for $V_j$ with coarse space $V_j^c$ and the following theorem holds:
\begin{theorem}\label{thm:convergence_proof}
	If all the previous assumptions hold, then for each $\mathbf{v} \in V$ we have the error estimate:
	\[
	\min_{\mathbf{v}^c \in V_c} \|\mathbf{v} - \mathbf{v}_c\|_D^2 \leq C_{p,1} C_{p,2} \mu_c^{-1} \|\mathbf{v} \|_A^2.
	\] 
\end{theorem}
Then the TL-AMG with coarse space defined in \eqref{eq:coarsespace} converges with a rate:
\begin{equation}\label{eq:convergence_rate_bound}
\|E\|_A \leq 1 - \frac{\mu_c}{C_{p,1} C_{p,2} c^D}
\end{equation}
with $c^D$ depending on the convergent smoother, i.e.,
\begin{equation}\label{eq:norm_continuity_smoother}
c_D \| \mathbf{v} \|_D^2 \leq \|\mathbf{v} \|_{\overline{R}^{-1}}^2 \leq c^D \| \mathbf{v} \|_D^2.
\end{equation}
{ From the above result it is clear why the constant $\mu_c$ in~\eqref{eq:finalcrate} represents the convergence quality measure for the aggregates that we were looking for}. We will use it in Section~\ref{sec:ourconvergence},  to infer the convergence of the TL-AMG based on coarsening relying on weighted matching described in~\cite{DV2013,BootCMatch}, as well as to evaluate the quality of the aggregates. {Let us also underline that many of the convergence results for TL-AMG methods can be described by means of this set of tools; see, e.g.,~\cite[sections 12.4 and 13.1]{XuZikatanovReview} for the application to the classical AMG and aggregation-based~AMG.} 

\section{Generating aggregates from matching in weight\-ed graphs}
\label{sec:ourconvergence}

We now adopt the theory discussed in the previous section to analyze the
construction of the coarse space by means of the {\em coarsening based on
compatible weighted matching} as in~\cite{DV2013,BootCMatch}. We
note that, as described in the original papers, our aggregation approach is
driven by the idea to generate aggregates automatically, with no use of 
heuristics nor a priori information on the near kernel of the linear system;
however, after generating non-overlapped aggregates by applying maximum weight
matching, the setup of the prolongator operator is based on a projection of an
arbitrary vector (hopefully a sample of slow-convergent error components, see
Section~\ref{sec:selecting_the_weight} for discussion) on the aggregates, in a
way similar to the well-known approaches of AMG based on smoothed
aggregation~\cite{VMB1996}.

We look at the graph $G = (\mathcal{V},\mathcal{E})$ associated with the sparse matrix\footnote{For the sake of simplicity, we are using the same notation for representing linear operators and their corresponding matrices with the only change being the substitution of the adjoint operator with the transpose.} $A$, also known as the adjacency graph of $A$. This is the graph $G$ whose set of nodes $\mathcal{V}$ corresponds to the row/column indices $\mathcal{I} = \{1,\ldots,n\}$ of $A$, and whose set of edges $e_{i \mapsto j} = (i,j) \in \mathcal{E}$ is induced by the sparsity pattern of $A$. To this graph we associate an edge weight matrix $\hat{A}$ with the following entries:
\[(\hat{A})_{i,j} = \hat{a}_{i,j} = 1 - \frac{2 a_{i,j} w_i w_j}{a_{i,i} w_i^2 + a_{j,j}w_j^2},\]
where $a_{i,j}$ are the entries of $A$ and $\mathbf{w}=(w_i)_{i=1}^n$ is a given vector. 
For such a graph, a \emph{matching} $\mathcal{M}$ is a set of pairwise
non-adjacent edges, containing no loops, i.e., no two edges share a common
vertex. We call $\mathcal{M}$ a \emph{maximum product matching} if it maximizes
the product of the weights of the edges $e_{i \mapsto j}$ belonging to it, i.e.,
if it maximizes the product of the entries of $\hat{A}$ associated to the matched
indices. We stress that for sub-optimal matching algorithms, as discussed in
Section \ref{sec:matching_algorithms}, there may be nodes which are not endpoints
of any of the matched edges: we call such nodes \emph{unmatched}. By the above
procedure we are choosing as $V_1,\ldots,V_j$ the spaces defined by the
aggregates $\{\mathcal{A}_j\}_{j=1}^{J}$ for the row/column indices $\mathcal{I}$
denoting the matrix entries; equivalently, we are decomposing the index set as 
\begin{equation}\label{eq:index-set-decomposition}
\mathcal{I} = \bigcup_{i=1}^{J} \mathcal{A}_j, \; \mathcal{A}_i \cap \mathcal{A}_j = \emptyset \text{ if } i\neq j.
\end{equation}
More generally, to further reduce the dimension of the coarse space, we can perform subsequent pairwise matching steps, i.e., we can iterate $\ell$ times the
matching procedure, acting each time on the graph $G'$ obtained by collapsing
together the matched nodes from the previous step.

Let us consider the case in which a single step of pairwise aggregation is
performed. We can identify two types of aggregates $\mathcal{A}_j$: those
corresponding to pairs of matched nodes, for which $V_j = \mathbb{R}^2$, and those corresponding to the unmatched nodes, for which  $V_j = \mathbb{R}$.

The next step in the construction is the definition of the global prolongation matrix $P$ by means of the operators $\Pi_j : V_j \rightarrow V$, for $j=1,\ldots,J$, in~\eqref{eq:projector}. Let us denote by $n_p = \lvert \mathcal{M}\rvert$ the cardinality of the graph matching~$\mathcal{M}$, i.e., the number of matched nodes, and by $n_s$ the number of unmatched nodes. We identify for each edge $e_{j_1 \mapsto j_2} \in \mathcal{M}$ the vectors

\begin{equation}\label{eq:projector_defininig_vectors}
\mathbf{w}_{e_{j_1\mapsto j_2}} = \frac{1}{\sqrt{w_{j_1}^2 + w_{j_2}^2}} \begin{bmatrix}
w_{j_1} \\ w_{j_2}
\end{bmatrix}, \quad \mathbf{w}_{e_{j_1\mapsto j_2}}^\perp = \frac{1}{\sqrt{ \frac{w_{j_1}^2}{a_{j_2,j_2}^2} + \frac{w_{j_2}^2}{a_{j_1,j_1}^2}}} \begin{bmatrix}
- \frac{w_{j_1}}{a_{j_2,j_2}} \\ \frac{w_{j_2}}{a_{j_1,j_1}}
\end{bmatrix}.
\end{equation}
To build the local prolongator $\Pi_j$ we introduce the family of maps $\{\eta_j\}_{j=1}^J$~for
\begin{equation}
\begin{aligned}
& \eta_j:\;\{j_1,j_{n_j}\}\to \{1,2,\ldots,n\}\\
& \eta_j(j_p)=i \quad \Longleftrightarrow
\quad \mathcal{A}_j=\{j_1,j_{n_j}\}, \quad\mbox{and}\quad
i=j_p,
\end{aligned}
\end{equation}
where we assume that in the case of an unmatched node, i.e., when $n_j = 1$, then $\mathcal{A}_j = \{j_1\}$. Thus we have defined the correspondence relation between the indices in the local numbering on the aggregates and the numbering in the global space, that~is
\begin{equation}\label{eta-1}
\{j_1,j_{n_j}\} = \left\{\eta_j(j_1),\eta_j(j_{n_j})\right\}.
\end{equation}
Let now $\{\delta_i\}_{i=1}^{n}$ and $\{e_{j,j_p}\}_{j_p=1}^{n_j}$ be the basis of $V$ and $V_j$ respectively
\begin{equation*}
V=\operatorname{span}\{\delta_i\}_{i=1}^{n}, \quad  V_j=\operatorname{span}\{e_{j,j_p}\}_{p=1}^{n_j}.
\end{equation*}
We introduce the operator $\hat{\Pi}_j$ and its dual with respect to the Euclidean/$\ell^2$ inner product $\hat{\Pi}'_j$, respectively, as
\begin{equation*}
\forall\, s \in V_j, \;\;  s = \sum_{p=1}^{n_j} s_p e_{j,j_p}, \; \quad\text{then}\quad\hat{\Pi}_j s = \sum_{p=1}^{n_j}s_p\delta_{\eta_j(j_p)},
\end{equation*}
and
\begin{equation*}
\forall\, w \in V, \;\; \hat{\Pi}'_j w \in V_j, \qquad \hat{\Pi}'_j w = \sum_{p=1}^{n_j}w_{\eta_j(j_p)} e_{j,j_p},
\end{equation*}
that has been obtained by direct computation of its $\ell^2$ inner product. Finally, we define the $\Pi_j$ associated with the aggregates as 
\begin{equation}\label{eq:localprolongator}
\Pi_j = \hat{\Pi}_j \hat{\Pi}'_j, \quad j = 1,\ldots,J, \quad \Pi_j = \Pi_j', \;\; \Pi_j \Pi_k = 0, \text{ whenever } j \neq k.
\end{equation}
Then {$\underline{A}_W=\diag(A_1, A_2, \ldots, A_J)=\diag(\Pi'_1 A \Pi_1 , \Pi'_2 A \Pi_2 , \ldots, \Pi'_J A \Pi_J)$ is the block-diagonal operator corresponding to the restriction of $A$ to the unknowns belonging to the j-th aggregate,} and the corresponding columns of the projection matrix are given by
\begin{equation*}
\tilde{P} = [\mathbf{p}_1,\ldots,\mathbf{p}_{n_p}] \text{ for } \mathbf{p}_j = \Pi_j \mathbf{w}_{e_{i\mapsto j}}.
\end{equation*}
\begin{remark}\label{rmk:dorthognonality}
	The vectors in~\eqref{eq:projector_defininig_vectors} are by construction $D$--orthogonal with respect to the local matrix \[D_{e_{i\mapsto j}} = \diag([a_{i,i},a_{j,j}])\text{, i.e., } \mathbf{w}_{e_{i\mapsto j}}^T D_{e_{i\mapsto j}} \mathbf{w}_{e_{i\mapsto j}}^\perp = 0.\] 
\end{remark}
To complete the construction of the prolongation matrix, we also need to fix an ordering for the unmatched $n_s = n_c - n_p = J - n_p$ nodes. The local projector $\Pi_j$ is again the one in~\eqref{eq:localprolongator}, but we apply it to the scalars $\nicefrac{w_{k}}{\lvert w_k \rvert}$, $k=1,\ldots,n_s$, thus obtaining the remaining columns of the prolongation matrix
\begin{equation*}
W = [\mathbf{p}_{n_p+1},\ldots,\mathbf{p}_{n_p+n_s}] = [\mathbf{p}_{n_p+1},\ldots,\mathbf{p}_J] \text{ for } \mathbf{p}_{k} = \Pi_k \frac{\mathbf{w}_{k}}{\lvert\mathbf{w}_k\rvert}.
\end{equation*}
In an expanded form, the resulting prolongation matrix can then be expressed~as
\begin{equation}\label{eq:prolongator}
\begin{split}
P  &   \underbracket[.4pt]{\left.\left[\begin{array}{@{}c@{\quad}c}
	\underbracket[.4pt]{\left.\begin{array}{ccc}
		\mathbf{w}_{e_1} & \vphantom{\ddots}0 & 0 \\
		0 & \ddots & 0 \\
		0 & \vphantom{\ddots}0 & \mathbf{w}_{e_{n_p}}
		\end{array}\right]}_{n_p}{2n_p} & \text{\huge 0} \\
	\text{\huge 0} & \underbracket[.4pt]{\left.\begin{array}{ccc}
		\nicefrac{w_{1}}{\lvert w_1 \rvert} & \vphantom{\ddots}0 & 0 \\
		0 & \ddots & 0 \\
		0 & \vphantom{\ddots}0 & \nicefrac{w_{n_s}}{\lvert w_{n_s} \rvert}
		\end{array}\right]}_{n_s}n_s
	\end{array}\right]\right]}_{n_c = n_p + n_s = J}    \rotatebox[origin=lb]{90}{$n = 2n_p + n_s$} \\
= & \begin{bmatrix}
\tilde{P} & W
\end{bmatrix} = [\mathbf{p}_1,\ldots,\mathbf{p}_J],
\end{split}
\end{equation}
which also allows to express the global coarse space as the space generated by the columns of $P$, i.e., $V_c = \Span\{\mathbf{p}_1, \ldots, \mathbf{p}_J\}$. The matrix $P$ we have just built represents a piecewise constant interpolation operator. 

\subsection{Selecting the weight vector}
\label{sec:selecting_the_weight}
We can now use again the general theory for the convergence of a multigrid algorithm to discuss what is the optimal choice for the weight vector $\mathbf{w}$, and therefore identify the optimal prolongator operator $P$. To this aim we recall the following well known result~\cite{FVZ2005,XuZikatanovReview,BCKFH2018}.

\begin{theorem}\label{thm:optimal_prolongator}
	Let $\{ \lambda_j, \boldsymbol{\Phi}_j)_{j=1}^n$ be the eigenpairs of $\overline{T}= \overline{R} A$. Let us also assume that $\boldsymbol{\Phi}_j$ are orthogonal w.r.t. $(\cdot, \cdot)_{\overline{R}^{-1}}$. The convergence rate $\|E(P)\|_A$ is minimal for $P$ such that
	\begin{equation*}
	\Range(P)=\Range(P^{opt}), \ \ \text{where} \ \ P^{opt} = \{ \boldsymbol{\Phi}_1, \ldots, \boldsymbol{\Phi}_{n_c}\}.
	\end{equation*}
	In this case,
	\begin{equation*}
	\|E\|^2_A= 1 - \lambda_{n_c+1}
	\end{equation*}
\end{theorem}
Therefore, a sensible choice would be to include in the range of $P$ at least the first eigenvector $\boldsymbol{\Phi}_1$; this would be sufficient to enforce convergence, albeit possibly  with a poor convergence ratio.
\begin{proposition}\label{pro:at-least-is-convergent}
	Using the same notation of Theorem~\ref{thm:optimal_prolongator}, if the weight vector $\mathbf{w}$ used to define the prolongator matrix $P$ in~\eqref{eq:prolongator} is the $\boldsymbol{\Phi}_1$ eigenvector of $\overline{T}= \overline{R} A$ then the $A$--norm of the error propagation matrix $\|E\|^2_A$ is less or equal~than
	\begin{equation*}
	\|E\|^2_A \leq 1 - \lambda_{2}.
	\end{equation*}
\end{proposition}

\begin{proof}
	The range of the prolongation matrix $P$ in~\eqref{eq:prolongator} includes the original vector of the weights $\mathbf{w}$, i.e., there exists $\mathbf{h} \in \mathbb{R}^{n_c}$ such that $P \mathbf{h} = \mathbf{w}$. The conclusion follows immediately by a straightforward application of Theorem~\ref{thm:optimal_prolongator}.
\end{proof}

Unfortunately, this is not an optimal choice from a computational point of view;
if we did possess some \emph{a priori} information on the eigenvector, then using
this information could improve the quality of the aggregates, and thus the
convergence of the method. 

In the case where we do not possess information on the eigenvector(s), selecting
the appropriate vector $\mathbf{w}$ may not be an easy task. To obtain a good
candidate in a completely black--box manner we could exploit the smoother
$\overline{R}$ to select as a weight vector an $\varepsilon$--smooth algebraic
vector in the sense of the following~\cite{XuZikatanovReview}: %
\begin{definition}\label{def:algebraic_smooth_vectors}
	Let $R: V \rightarrow V$ be a smoothing operator such that its symmetrization~$\overline{R}$ is positive definite. Given $\epsilon \in (0,1)$, we say that the vector $v$ is algebraically $\epsilon$-smooth with respect to $A$~if
	\begin{equation*}
	\|\mathbf{v}\|^2_A \leq \epsilon \|\mathbf{v}\|^2_{\overline{R}^{-1}}.
	\end{equation*}
\end{definition}
Such a vector can be obtained by performing a few iterations of the smoother on either a random choice or on the initial theoretical guess. 

The last possible adaptive refinement that we are going to consider is the application of a bootstrap iteration exploiting the multigrid hierarchy itself as in~\cite{BootCMatch}. A whole hierarchy $B_0$ associated with an initial guess $\mathbf{w}_0$, again either a random or user-defined guess, is built in the first step of the bootstrap procedure. Then the hierarchy is used to refine the choice of vectors $\mathbf{w}$ by means of the iteration~\eqref{eq:gitm} for the homogeneous linear system, i.e.,
\begin{equation}\label{eq:bootstrapiteration}
\text{ Given }\mathbf{w}_0 \text{ compute} \begin{cases}
\mathbf{w}^{(0)} = \mathbf{w}_{r-1}, & r=1,\ldots,k-1,\\
\displaystyle \mathbf{w}^{(j)} = \prod_{p=0}^{r-1}(I - B_p^{-1}A)\mathbf{w}^{(j-1)}, & j = 1,\ldots,m,\\
\mathbf{w}_{r+1} = \mathbf{w}^{(m)}. 
\end{cases}
\end{equation}
To build the multigrid hierarchies $B_p$ for the bootstrap iteration~\eqref{eq:bootstrapiteration} we exploit now the vectors $\mathbf{w}_r$ available at each $r$th step. 

We stress that, from an  operational point of view, this means that if one knows at least one  $\varepsilon$-smooth vector $\mathbf{w}$ to be used as $\mathbf{w}_0$, then it is possible to use it to launch the bootstrap iteration~\eqref{eq:bootstrapiteration} and obtain hierarchies $B_0,B_1,\ldots,B_{r-1}$, each satisfying the convergence result in Theorem~\ref{thm:convergence_proof}, and generating, when accumulated all--together, an algorithm with improved convergence rate. Moreover, if the bootstrap iteration is launched with a random vector then the TL--AMG algorithm with the bootstrap procedure can still obtain an acceptable convergence rate (see~\cite{BootCMatch,DV2019}).

\subsection{Selecting the matching algorithm}
\label{sec:matching_algorithms}

One of the main costs in the construction of the multigrid hierarchy is represented by the computation of the maximum product matching needed to identify the aggregates. It is useful to distinguish here between two different approaches. The first approach is to  compute an \emph{exact} matching, i.e., a matching that achieves exactly the optimum value for the  product. The  second approach computes a matching whose product value is not optimal, but is guaranteed to be greater or equal to $\nicefrac{1}{2}$ of the maximum; this is called a \emph{$\frac{1}{2}$-approximate} maximum product matching. Relaxing the requirement to obtain the exact optimum allows the achievement of both a reduction of the construction time, as well as the possibility to perform the building phase in a parallel context with a limited amount of data exchange. 
For the details regarding these computational complexity aspects, we refer to the discussion in~\cite{BootCMatch}; here we focus on the quality of the aggregates obtained by using the different matching algorithms.

For the class of exact algorithms, we employ the algorithm in~\cite{matchingduffhsl} that is  implemented in the \texttt{HSL\_MC64} routine~\cite{hslcode}, while for the approximate class we refer to the  {$\frac 12$}--approximation algorithm in~\cite{preis}, a parallel distributed-memory version of which is employed in~\cite{DDF2021}, the auction type algorithm from~\cite{auction}, and the suitor algorithm in~\cite{suitor}, which is what we applied in a parallel Graphics Processing Unit (GPU) setting (see~\cite{BDP2019}).

\subsection{Computing the \texorpdfstring{$\mu_c$}{muc} constant}\label{sec:computingmucconstant}
First we focus on the task of computing exactly the $\mu_c$ constant in Theorem~\ref{thm:convergence_proof}. Thus we first need to prove that the Assumptions in~\eqref{eq:assumption1},~\eqref{eq:assumption2} and~\eqref{eq:assumption3} hold for the construction discussed in Section~\ref{sec:ourconvergence}.

\begin{lemma}\label{lem:assumptionlemma}
	Let the two--grid multigrid hierarchy be constructed with the prolongator $P$ in~\eqref{eq:prolongator}. 
	Then assumptions~\eqref{eq:assumption1}, and \eqref{eq:assumption2} hold true with $C_{p,1}= 1$, and $C_{p,2}=1$. Moreover, if $A$ is SPD, assumption~\eqref{eq:assumption3} holds since $N(A_j) = \{\mathbf{0}\}$ for every $j$. 
\end{lemma}

\begin{proof}
	To prove~\eqref{eq:assumption1} we observe that by~\eqref{eq:aggregator_restrictor} and~\eqref{eq:localprolongator} we have that for all $\mathbf{\underline{v}} \in W$
	\begin{equation*}
	\|\Pi_W \mathbf{\underline{v}} \|_D^2 = \sum_{j=1}^{J}\|\Pi_j \mathbf{\underline{v}} \|_{D_j}^2 = \sum_{j=1}^{J}\left\| \begin{bmatrix}\underline{v}_{j_1}\\\underline{v}_{j_2}\end{bmatrix} \right\|_{D_j}^2 = \|\mathbf{\underline{v}} \|_{\underline{D}}^2, \Rightarrow \, C_{p,2}=1.
	\end{equation*}
	To prove~\eqref{eq:assumption1} we use the ``local-to-global'' maps in~\eqref{eta-1} to have the index correspondence between the aggregates and the global matrix. Then, noting that $\Pi_j^2=\Pi_j$, $\Pi_k\Pi_j=0$ for $k\neq j$, and that $\Pi_j=\Pi_j'$, for $j=1,\ldots,J$, by a direct computation, we find that
	\begin{align*}
	\| \underline{\mathbf{w}} \|^2_{\underline{A}_W} = & \langle \underline{A}_W  \underline{\mathbf{w}}, \underline{\mathbf{w}} \rangle_{\ell^2} =   \left\langle  \sum_{j=1}^J \Pi_j A \Pi_j \mathbf{w}, \sum_{k=1}^J \Pi_k \mathbf{w} \right\rangle_{\ell^2} \\
	= & \sum_{k=1}^J\sum_{j=1}^J \langle   \Pi_k \Pi_j A \Pi_j \mathbf{w}, \mathbf{w} \rangle_{\ell^2} = \sum_{j=1}^J \langle \Pi_j A \Pi_j \mathbf{w},\mathbf{w} \rangle_{\ell^2} \\ = & 
	\sum_{j=1}^J \langle A \Pi_j \mathbf{w}_j,\Pi_j \mathbf{w}_j\rangle_{\ell^2}  
	=  \|\mathbf{w}\|^2_A. & 
	\end{align*}
	The kernel of the projected matrices $A_j$ is reduced to the zero vector since the projector has orthogonal columns, and thus the projected matrices on $\underline{W}$ are~SPD.
\end{proof}

The above assumptions practically depend on the fact that independently from the number of aggregation sweeps we collect together, we are decomposing the index set $\mathcal{I}$ as a direct sum of non-overlapping indices as in~\eqref{eq:index-set-decomposition}.

This means that we can compute the global constant $\mu_c$ in~\eqref{eq:finalcrate} \emph{a posteriori} by solving the generalized eigenvalue problem 
\begin{equation}\label{eq:theeigenvalueproblem}
D(I - Q) \mathbf{x} = \mu_c^{-1} A \mathbf{x}, 
\end{equation}
where $Q$ has been built from the $D_j$--orthogonal projectors $Q_j : V_j \rightarrow V_j^c$, which in our case have the following representation matrices:
\begin{equation*}
Q_j = \begin{cases}
\mathbf{w}_j (\mathbf{w}_j^T D_j \mathbf{w}_j)^{-1} \mathbf{w}_j^T D_j, & j = 1,\ldots,n_p \\
1 , & j=n_p+1,\ldots,n_p+n_s=J
\end{cases}
\end{equation*}
and in aggregate form as the $D$--orthogonal projector represented by:
\begin{equation}\label{eq:thedorthoprojector}
Q = P(P^T D P)^{-1} P^TD  = \operatorname{diag}(Q_1,\ldots,Q_J).
\end{equation}

\subsection{Estimating the \texorpdfstring{$\mu_c$}{muc} constant}

The general theory for an aggregation-based multigrid, as formalized in~\cite{XuZikatanovReview} and specialized in the previous Section~\ref{sec:computingmucconstant} for our method, was originally applied in~\cite{N2010,NN2011} for the case of disjoint aggregates with piecewise constant prolongators having unit coefficients; refer also to the \emph{bibliographical notes} in \cite[Section~8.5]{XuZikatanovReview}. An additional tool provided by the discussion in~\cite{NN2011} is the possibility of carrying out a purely \emph{local} analysis by looking only at the restriction on the aggregates of the operators $\underline{A}_W$, and $\underline{D}$ under stricter hypothesis on the matrix $A$ of the system and on possible aggregates.

Specifically, to adopt the general strategy introduced in~\cite{NN2011}, we identify these operators as the restriction of the operator $A$ to the aggregates obtained through the matching algorithm,~i.e.,
\begin{equation}\label{eq:restricted_operators}
\underline{A}_W = (A_1,\ldots,A_J), \quad A_k = A\rvert_{V_k}, \quad \underline{D} = (\underline{D}_1,\ldots,\underline{D}_J), \quad D_k = D\rvert_{V_k}.
\end{equation}
We can then write the complete matrix $A$ as the sum of the block diagonal matrix $\underline{A}_W$ and a remainder $A_R$ containing all the parts we have discarded. Under the stricter hypothesis on $A$ discussed in \cite{NN2011} it is possible to find symmetric and non-negative definite $\underline{A}_W$ and $A_R$. This allows us to apply~\cite[Theorem 3.4]{NN2011} and obtain the `local` bound to the global $\mu_c$ constant in~Theorem~\ref{thm:convergence_proof}. We simply restate the result here in the notation from~\cite{XuZikatanovReview} and the construction from Section~\ref{sec:ourconvergence}.

\begin{theorem}[Restatement of {\cite[Theorem 3.4]{NN2011}}]\label{thm:our_convergence_result}
	Let $\underline{A}_W = (A_1,\ldots,A_J)$ and $\underline{D} = (\underline{D}_1,\ldots,\underline{D}_J)$ satisfy the splitting condition $A=\underline{A}_W + A_R$, with $\underline{A}_W$ and $A_R$ both symmetric and non-negative definite, that is, every $\{A_j\}_{j=1}^{J}$ is non-zero symmetric non-negative definite and $\underline{D}$ symmetric positive definite. Let $\mathbf{p}$ be one of the columns of $P$ in~\eqref{eq:prolongator}, i.e., $\mathbf{p} = \mathbf{w}_{e_{i\rightarrow j}}$ for the indices $(i,j)$ relative to the given aggregate.
	
	Then $\mu_c$ is defined as in~\eqref{eq:finalcrate}, and the $\mu_j^{-1}(V_j^c)$ are such that
	\begin{equation*}
	\lambda_2^{-1}(D_j^{-1}A_j) \leq \mu_j^{-1}(V_j^c) \leq \lambda_1^{-1}(D_j^{-1}A_j).
	\end{equation*}
	Moreover, if either $(\mathbf{w}_{e_{i\rightarrow j}},\lambda_1(D_j^{-1}A_j))$, or $(\mathbf{w}_{e_{i\rightarrow j}}^\perp,\lambda_2(D_j^{-1}A_j))$ are eigencouples of the matrix $D_j^{-1}A_j$, then  
	\begin{equation*}
	\mu_j^{-1}(V_j^c) = \lambda_2^{-1}(D_j^{-1}A_j).
	\end{equation*}
\end{theorem}
We stress that while in general it is always possible to compute the quantity $\mu_c$ in~\eqref{eq:finalcrate} by solving the eigenvalue problem in~\eqref{eq:theeigenvalueproblem}, and thus estimate the overall quality of the matching procedure, application of Theorem~\ref{thm:our_convergence_result} to obtain the bound by using only local information requires the stricter hypotheses on the splitting of $A$. 

\section{Numerical experiments}\label{sec:numerical_experiments}
To highlight the results of Theorem~\ref{thm:our_convergence_result} we consider the case study of the 2D Laplace equation with variable coefficients on the unit square $\Omega = [0,1]^2$, dicretized with 5--point finite differences,  i.e. the equation
\begin{equation}\label{eq:thediffusionproblem}
\begin{cases}
-\nabla \cdot( a(x,y) \nabla u(x,y)) = f(x,y), & (x,y) \in \Omega,\\
u(x,y) = 0, & (x,y) \in \partial \Omega,
\end{cases}
\end{equation}
{ and discretized by Lagrangian P1 elements on an unstructured triangular grid.}
We focus on a 2D example so that we can graphically represent the different aggregates. We concentrate first on the computation of the bounds discussed in Theorem~\ref{thm:our_convergence_result} and on the analysis of the different bounds obtained for the different choices of the matching algorithm in Section~\ref{sec:matching_algorithms} while keeping fixed the choice of the weight vector $\mathbf{w}$. Then, in the second part of the numerical examples, we devote our attention to the analysis of the quality of the aggregates for different choices of the weight vectors $\mathbf{w}$, while considering also the different refinement strategies discussed in Section~\ref{sec:selecting_the_weight}.

The version of the BootCMatch algorithm~\cite{BootCMatch} we use here for the tests is available on the repository \url{https://github.com/bootcmatch/BootCMatch}. All the plots and the eigenvalues/eigenvectors computations are then performed in Matlab v. 9.6.0.1072779 (R2019a) on the matrices exported in Matrix Market format.

  \subsection{Computing the \texorpdfstring{$\mu_c$}{muc} constants}\label{sec:computing_the_muc_constant}
To confirm the applicability of the theory developed in Section~\ref{sec:ourconvergence} we compute both the ``true'' $\mu_c$ constants by solving the generalized eigenvalue problem with the $D$--orthogonal projector $Q$ in~\eqref{eq:thedorthoprojector}, and the estimate obtained by means of Theorem~\ref{thm:our_convergence_result}{, when the splitting for the matrices $\underline{A}_{W}$ is available,} for three different prototypical problems obtained from different choices of the diffusion coefficient in~\eqref{eq:thediffusionproblem}. For each of these cases we consider the various matching algorithms discussed in Section~\ref{sec:matching_algorithms} and the application of $\ell = 1, 2$ steps of pairwise matching, i.e., we consider aggregates made by at most two or four fine variables. In all cases, we consider the weight vector $\mathbf{w} = (1,1,\ldots,1)^T$, which is suggested by the structure of the matrix. We stress that all the results obtained in the following subsections can be read alongside the numerical experiments in~\cite{BootCMatch} since they complement and further explains the convergence behavior of the method discussed there. To present a wider array of tests, we have given other examples in Appendix~\ref{sec:appendix_computing_muc}.

\subsubsection{The constant coefficient diffusion}\label{sec:constant_coefficient_diffusion} The first case is the Laplacian with homogeneous coefficients, i.e., $a(x,y) = 1$, on a uniform $n\times n$ grid. This gives rise to the matrix
\begin{equation*}
A_{n^2} = I_n \otimes T_n + T_n \otimes I_n, \quad T_n = \operatorname{tridiag}(-1,2,-1),
\end{equation*}
scaled in such a way that its coefficients are independent from the dimension $n^2$ of the problem. We first  visualize the different aggregates generated by the various matching algorithms in Figure~\ref{fig:homogeneous_poisson_aggregates}.
\begin{figure}[htbp]
	\centering
	\subfloat[\texttt{HSL\_MC64} -- $\ell=1$]{\includegraphics[width=0.25\columnwidth]{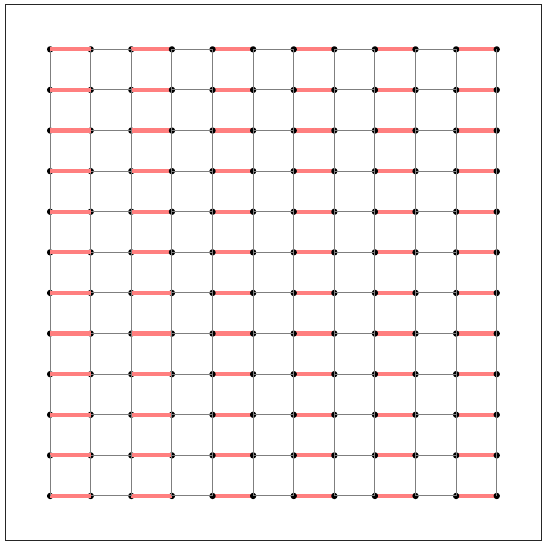}}\hfil
	\subfloat[\texttt{HSL\_MC64} -- $\ell=2$]{\includegraphics[width=0.25\columnwidth]{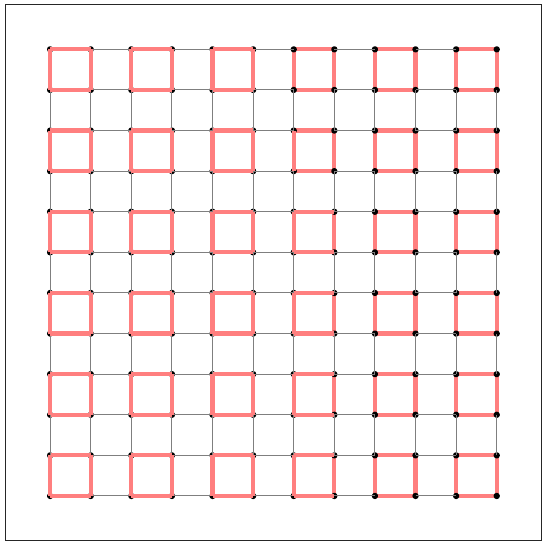}}\hfil
	\subfloat[\texttt{PREIS} -- $\ell=1$]{\includegraphics[width=0.25\columnwidth]{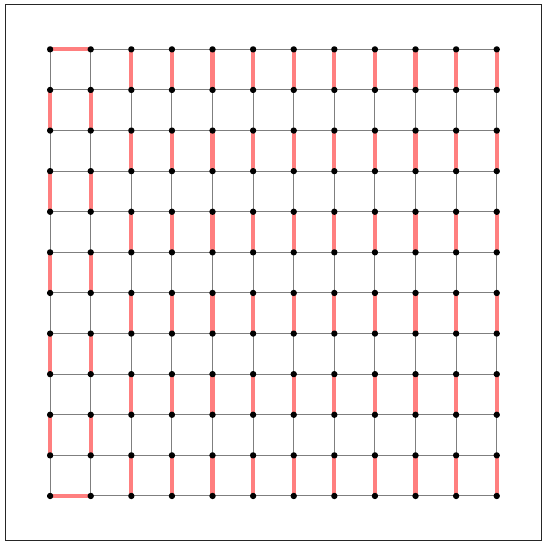}}\hfil
	\subfloat[\texttt{PREIS} -- $\ell=2$]{\includegraphics[width=0.25\columnwidth]{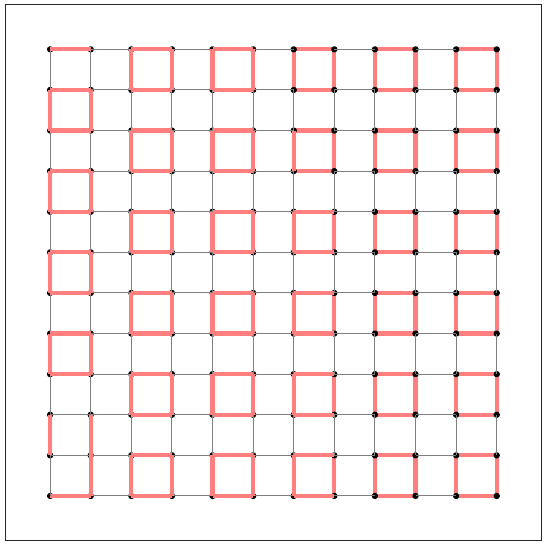}}
	
	\subfloat[\texttt{AUCTION} -- $\ell=1$]{\includegraphics[width=0.25\columnwidth]{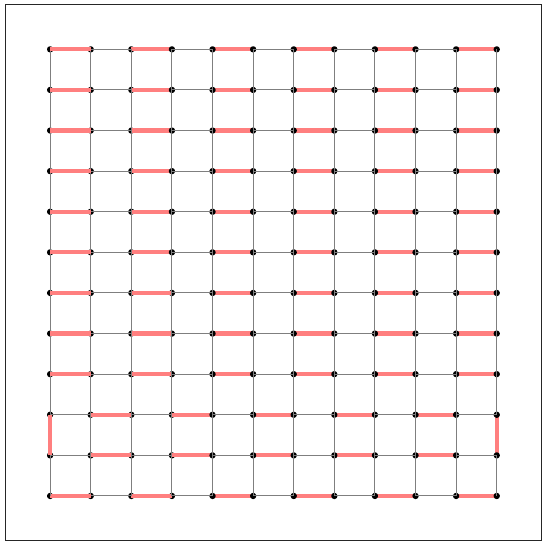}}\hfil
	\subfloat[\texttt{AUCTION} -- $\ell=2$]{\includegraphics[width=0.25\columnwidth]{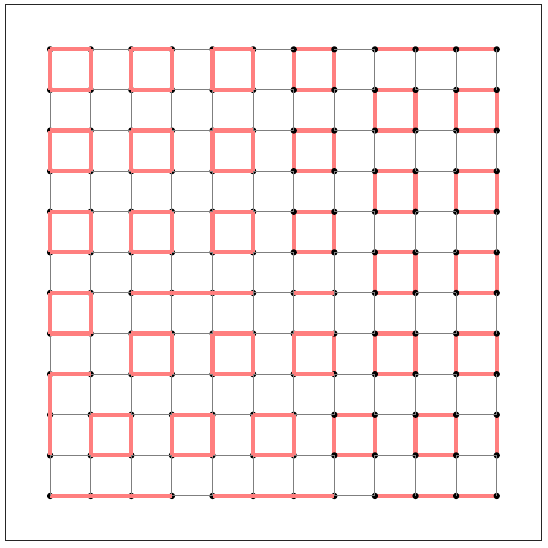}}\hfil
	\subfloat[\texttt{SUITOR} -- $\ell=1$]{\includegraphics[width=0.25\columnwidth]{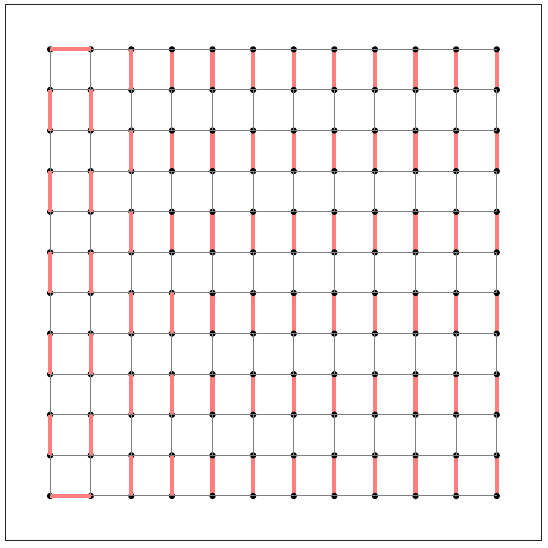}}\hfil
	\subfloat[\texttt{SUITOR} -- $\ell=2$]{\includegraphics[width=0.25\columnwidth]{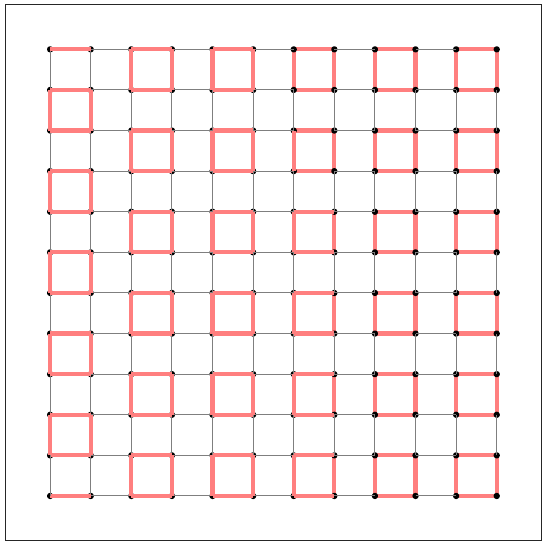}}
	\caption{Constant coefficient diffusion problem. Aggregates obtained with the weight vector $\mathbf{w}=(1,1,\ldots,1)^T$, and the different matching algorithms for $\ell=1,2$ pairwise matching steps.}
	\label{fig:homogeneous_poisson_aggregates}
\end{figure}
In this case the aggregation based on the maximum product matching \texttt{HSL\_MC64} produces the same aggregates that can be obtained by using the standard C$\backslash$F--splitting. Moreover, by~\eqref{eq:prolongator} it is straightforward to observe that $P$ is a scalar multiple of the one obtained by choosing $\mathbf{w}_{e_i}$ equal to the vector of all ones; hence, the methods produce exactly the same $Q$ of the classical aggregation, and therefore the same bounds obtained for it in~\cite[Theorem~3.4]{NN2011}. The aggregates also match the quality of the aggregates in~\cite{MatchingKimXuZikatanov}, in which the matching strategy for the identification of the aggregates is applied directly to $A$ and coupled with the prolongator $P$ whose nonzero entries are all $1$; see the results in Table~\ref{tab:hom_coeff_w_const}.
\begin{table}[htbp]
	\centering
	\caption{Constant coefficient diffusion problem. Comparison of the bound in Theorem~\ref{thm:our_convergence_result} with true value of $\mu_c$ in~\eqref{eq:finalcrate}.
		Aggregates obtained with the weight vector $\mathbf{w}=(1,1,\ldots,1)^T$, and the different matching algorithms for $\ell=1,2$ pairwise matching steps.}
	\label{tab:hom_coeff_w_const}
	\subfloat[\texttt{HSL\_MC64} -- exact matching]{
		\begin{tabular}{lcccc}
			
			& \multicolumn{2}{c}{$\ell = 1$} & \multicolumn{2}{c}{$\ell = 2$} \\
			\cmidrule(l{2pt}r{2pt}){2-3}\cmidrule(l{2pt}r{2pt}){4-5}
			n & bound & $\mu_c^{-1}$ & bound & $\mu_c^{-1}$ \\
			\midrule
			12 & 2.000 & 1.940 & 2.000 & 1.959 \\
			24 & 2.000 & 1.984 & 2.000 & 1.989 \\
			48 & 2.000 & 1.996 & 2.000 & 1.997 \\
			96 & 2.000 & 1.999 & 2.000 & 1.999 \\
			
	\end{tabular}}\hfil
	\subfloat[\texttt{PREIS} -- $\frac{1}{2}$--approximate matching]{
		\begin{tabular}{lcccc}
			
			& \multicolumn{2}{c}{$\ell = 1$} & \multicolumn{2}{c}{$\ell = 2$} \\
			\cmidrule(l{2pt}r{2pt}){2-3}\cmidrule(l{2pt}r{2pt}){4-5}
			n & bound & $\mu_c^{-1}$ & bound & $\mu_c^{-1}$ \\
			\midrule
			12 & 2.000 & 1.923& 2.062 & 2.046 \\
			24 & 2.000 & 1.982& 2.062 & 2.052 \\
			48 & 2.000 & 1.996& 2.062 & 2.052 \\
			96 & 2.000 & 1.999& 2.062 & 2.052 \\
			
	\end{tabular}}
	
	\subfloat[\texttt{AUCTION} -- $\frac{1}{2}$--approximate matching]{
		\begin{tabular}{lcccc}
			
			& \multicolumn{2}{c}{$\ell = 1$} & \multicolumn{2}{c}{$\ell = 2$} \\
			\cmidrule(l{2pt}r{2pt}){2-3}\cmidrule(l{2pt}r{2pt}){4-5}
			n & bound & $\mu_c^{-1}$ & bound & $\mu_c^{-1}$ \\
			\midrule
			12 & 2.000 & 1.908 & 2.667 & 2.544 \\
			24 & 2.000 & 1.980 & 2.894 & 2.964 \\
			48 & 2.000 & 1.995 & 2.667 & 2.166 \\ 
			96 & 2.000 & 1.999 & 2.667 & 2.173 \\
			
	\end{tabular}}\hfil
	\subfloat[\texttt{SUITOR} -- $\frac{1}{2}$--approximate matching]{
		\begin{tabular}{lcccc}
			
			& \multicolumn{2}{c}{$\ell = 1$} & \multicolumn{2}{c}{$\ell = 2$} \\
			\cmidrule(l{2pt}r{2pt}){2-3}\cmidrule(l{2pt}r{2pt}){4-5}
			n & bound & $\mu_c^{-1}$ & bound & $\mu_c^{-1}$ \\
			\midrule
			12 & 2.000 & 1.923 & 2.000 & 1.954 \\
			24 & 2.000 & 1.982 & 2.000 & 1.988 \\
			48 & 2.000 & 1.996 & 2.000 & 1.997 \\
			96 & 2.000 & 1.999 & 2.000 & 1.999 \\
			
	\end{tabular}}
\end{table}

Concerning the usage of alternative matching methods, we see that the \texttt{HSL\_MC64} and the \texttt{SUITOR} algorithms do produce the same $\mu_c$ constants and bounds, even if \texttt{SUITOR} is only guaranteed to reach a value of the objective function  one half away from the optimal one. In general, we can observe that in the cases $\ell=1$ the same constants are reached for different aggregates. This suggests that reaching the maximum weight is not mandatory and that different configurations can yield the same results in terms of the overall quality of the aggregates. {To achieve the upper bound from Theorem~\ref{thm:our_convergence_result}, we use the auxiliary splitting obtained by decreasing the diagonal blocks on the various aggregates by a correction of the form $\pm \delta_j I$ where each $\delta_j$ is computed heuristically to enforce the hypotheses. In these cases, for all the matching algorithms when we employ a single sweep, we use $\delta_j = 1/3 \min(A_j \mathbf{1})$, that is $1/3$ of the minimum row sum of the projection of $A$ on the aggregate. When two sweeps are employed, we use instead $\delta_j = \min(A_j \mathbf{1})$ for all the matching but the Auction case in which we employ $\delta_j = 1/2\min(A_j \mathbf{1})$. We stress that it is difficult to prescribe a formula to achieve the splitting and the local bound without looking into the matrices obtained from the matching procedure, since in general, this may not exist; see, e.g., the next example in which we encounter such a case for one of the matching algorithms.}

\subsubsection{Diffusion with axial anisotropies} As the second test case we consider having a simple spatial anisotropy oriented with the $y$--grid lines, i.e.,
\begin{equation*}
A_{n^2} = \varepsilon(I_n \otimes T_n) + T_n \otimes I_n, \quad T_n = \operatorname{tridiag}(-1,2,-1), \quad \varepsilon = 100,
\end{equation*}
in which we are again using a scaling that makes the matrix coefficients independent of the problem size. Intuitively, in this case,  we would  expect the aggregates to be oriented with the anisotropy, i.e., along  the $y$--axis. If we look at the aggregates we obtain in Figure~\ref{fig:yanisotropy_poisson_aggregates} we observe that the matching algorithms produce aggregates corresponding to our intuition, with the  exception of the \texttt{PREIS} algorithm that for $\ell=2$ produces some aggregates that do not seem feasible.
\begin{figure}[htbp]
	\centering
	\subfloat[\texttt{HSL\_MC64} -- $\ell=1$]{\includegraphics[width=0.25\columnwidth]{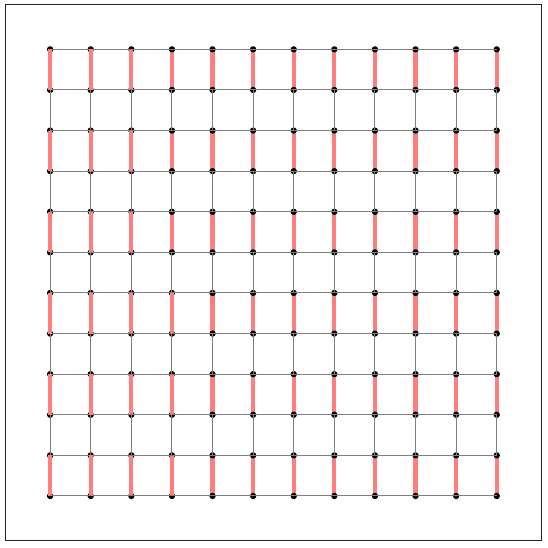}}\hfil
	\subfloat[\texttt{HSL\_MC64} -- $\ell=2$]{\includegraphics[width=0.25\columnwidth]{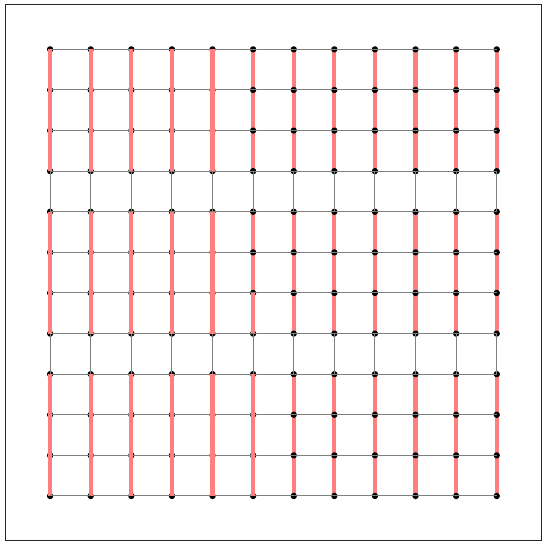}}\hfil
	\subfloat[\texttt{PREIS} -- $\ell=1$]{\includegraphics[width=0.25\columnwidth]{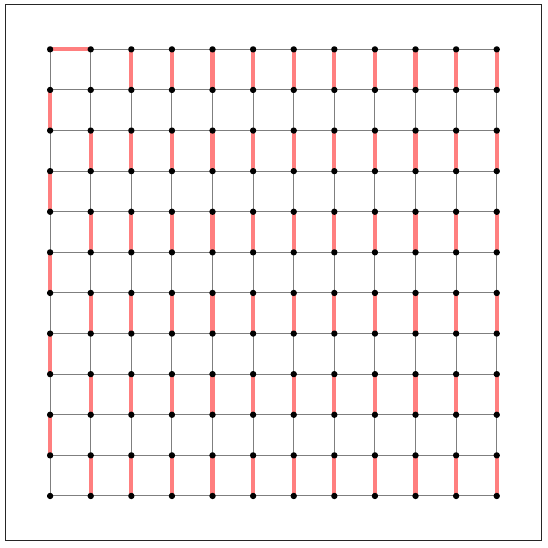}}\hfil
	\subfloat[\texttt{PREIS} -- $\ell=2$]{\includegraphics[width=0.25\columnwidth]{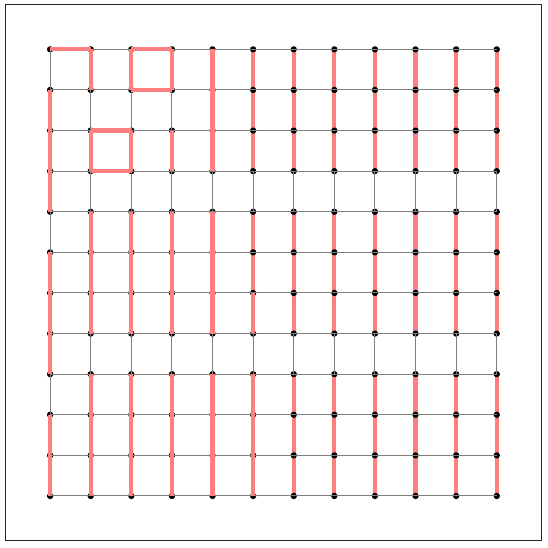}}
	
	\subfloat[\texttt{AUCTION} -- $\ell=1$]{\includegraphics[width=0.25\columnwidth]{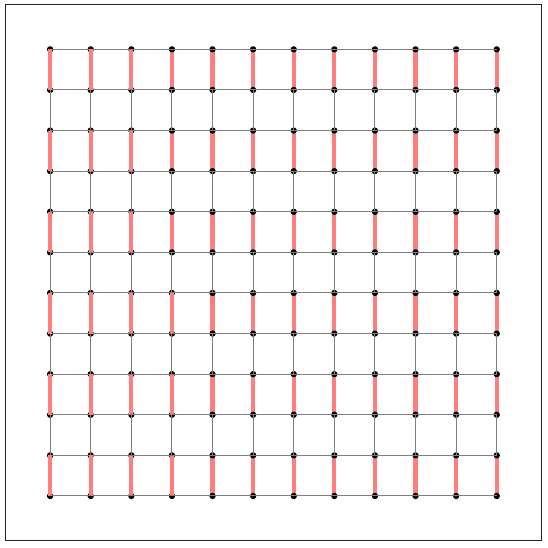}}\hfil
	\subfloat[\texttt{AUCTION} -- $\ell=2$]{\includegraphics[width=0.25\columnwidth]{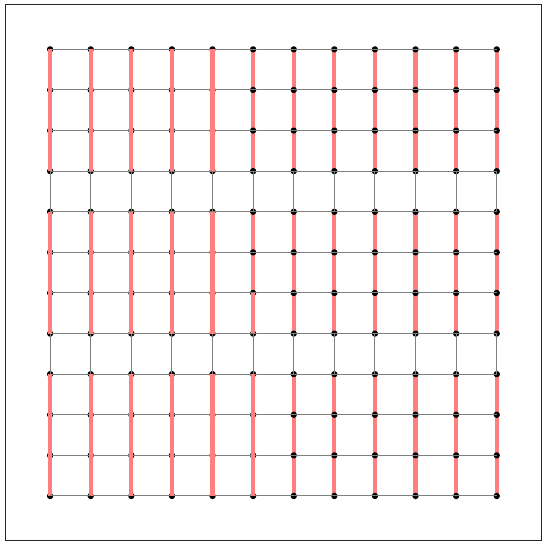}}\hfil
	\subfloat[\texttt{SUITOR} -- $\ell=1$]{\includegraphics[width=0.25\columnwidth]{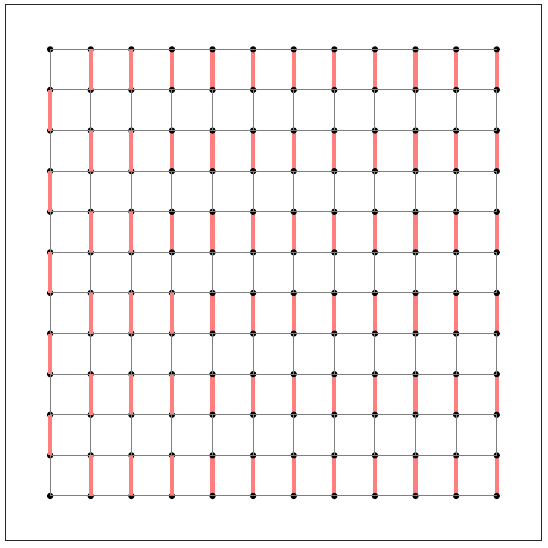}}\hfil
	\subfloat[\texttt{SUITOR} -- $\ell=2$]{\includegraphics[width=0.25\columnwidth]{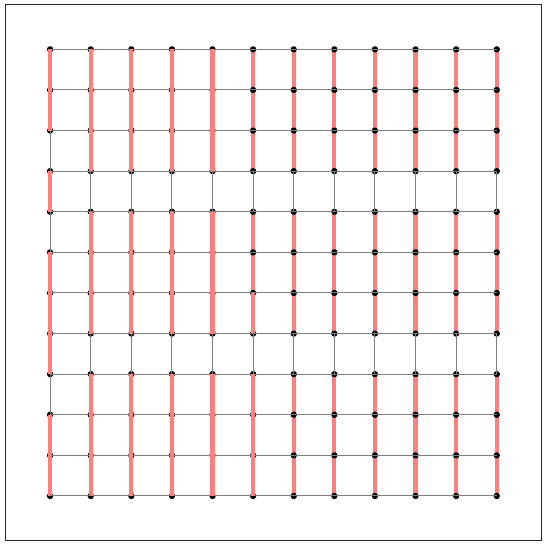}}
	\caption{Diffusion problem with $y$--axis oriented anisotropy $\varepsilon = 100$. Aggregates obtained with the weight vector $\mathbf{w}=(1,1,\ldots,1)^T$, and the different matching algorithms for $\ell=1,2$ pairwise matching steps.}
	\label{fig:yanisotropy_poisson_aggregates}
\end{figure}
Indeed, if we look also at the constants $\mu_c$, and their estimates reported in Table~\ref{tab:yanisotropy_poisson_aggregates} we observe that, excluding the case of the \texttt{PREIS} algorithm, the $\mu_c$ constant behaves consistently. {The failure in obtaining a bound in the case of the \texttt{PREIS} algorithm is due to the inability of finding a suitable splitting for the aggregates generated by this matching. Indeed, the existence of such splitting is a stricter hypothesis, and cannot be guaranteed in general. We refer back to the discussion in~\cite{NN2011} where the original strategy for obtaining the local bound was devised.} 
\begin{table}[htbp]
	\centering
	\caption{Diffusion problem with $y$--axis oriented anisotropy $\varepsilon = 100$. Comparison of the bound in Theorem~\ref{thm:our_convergence_result} with true value of $\mu_c$ in~\eqref{eq:finalcrate} for $\ell=1,2$ pairwise aggregation steps, while using the various matching algorithm with weight vector $\mathbf{w}=(1,1,\ldots,1)^T$. The $\dagger$ represents a case {in which we could not find the splitting needed to apply Theorem~\ref{thm:our_convergence_result}}.}
	\label{tab:yanisotropy_poisson_aggregates}
	\subfloat[\texttt{HSL\_MC64} -- exact matching]{
		\begin{tabular}{lcccc}
			
			& \multicolumn{2}{c}{$\ell = 1$} & \multicolumn{2}{c}{$\ell = 2$} \\
			\cmidrule(l{2pt}r{2pt}){2-3}\cmidrule(l{2pt}r{2pt}){4-5}
			n & bound & $\mu_c^{-1}$ & bound & $\mu_c^{-1}$ \\
			\midrule
			12 & 1.980 & 1.010 & 5.025 & 3.443 \\
			24 & 1.980 & 1.010 & 5.025 & 3.447 \\
			48 & 1.980 & 1.010 & 5.025 & 3.448 \\
			96 & 1.980 & 1.010 & 5.025 & 3.448 \\
			
	\end{tabular}}\hfil
	\subfloat[\texttt{PREIS} -- $\frac{1}{2}$--approximate matching]{
		\begin{tabular}{lcccc}
			
			& \multicolumn{2}{c}{$\ell = 1$} & \multicolumn{2}{c}{$\ell = 2$} \\
			\cmidrule(l{2pt}r{2pt}){2-3}\cmidrule(l{2pt}r{2pt}){4-5}
			n & bound & $\mu_c^{-1}$ & bound & $\mu_c^{-1}$ \\
			\midrule
			12 & 1.765 & 1.741 & $\dagger$ & 8.580 \\
			24 & 1.765 & 1.745 & $\dagger$ & 8.725 \\
			48 & 1.765 & 1.745 & $\dagger$ & 8.730 \\
			96 & 1.765 & 1.745 & $\dagger$ & 8.730 \\
			
	\end{tabular}}
	
	\subfloat[\texttt{AUCTION} -- $\frac{1}{2}$--approximate matching]{
		\begin{tabular}{lcccc}
			
			& \multicolumn{2}{c}{$\ell = 1$} & \multicolumn{2}{c}{$\ell = 2$} \\
			\cmidrule(l{2pt}r{2pt}){2-3}\cmidrule(l{2pt}r{2pt}){4-5}
			n & bound & $\mu_c^{-1}$ & bound & $\mu_c^{-1}$ \\
			\midrule
			12 & 1.980 & 1.010 & 5.025 & 3.443 \\
			24 & 1.980 & 1.010 & 5.025 & 3.447 \\
			48 & 1.980 & 1.010 & 5.025 & 3.448 \\
			96 & 1.980 & 1.010 & 5.025 & 3.448 \\
			
	\end{tabular}}\hfil
	\subfloat[\texttt{SUITOR} -- $\frac{1}{2}$--approximate matching]{
		\begin{tabular}{lcccc}
			
			& \multicolumn{2}{c}{$\ell = 1$} & \multicolumn{2}{c}{$\ell = 2$} \\
			\cmidrule(l{2pt}r{2pt}){2-3}\cmidrule(l{2pt}r{2pt}){4-5}
			n & bound & $\mu_c^{-1}$ & bound & $\mu_c^{-1}$ \\
			\midrule
			12 & 1.111 & 1.010 & 3.448 & 3.442 \\
			24 & 1.111 & 1.010 & 3.448 & 3.447 \\
			48 & 1.111 & 1.010 & 3.448 & 3.448 \\
			96 & 1.111 & 1.010 & 3.448 & 3.448 \\
			
	\end{tabular}}
\end{table}
It is interesting to compare the value of the constant for $\ell=1$ step of matching for this case with the one obtained for the case with constant coefficients in Table~\ref{tab:hom_coeff_w_const}:  observe in particular that the strong directionality of the diffusion makes the pairwise aggregates much more effective. On the other hand, we observe also that switching to larger aggregates leads to a worse quality of the aggregates  than in the case of an isotropic problem. 

\subsubsection{Diffusion on an unstructured mesh} {As a final test case, we consider again the Poisson problem with a constant diffusion coefficient but on an unstructured triangular mesh obtained via a Delaunay-based algorithm for which we report the subsequent refinements in Figure~\ref{fig:unstructuredmeshes}.}
\begin{figure}[htbp]
	\centering
	\includegraphics[width=\columnwidth]{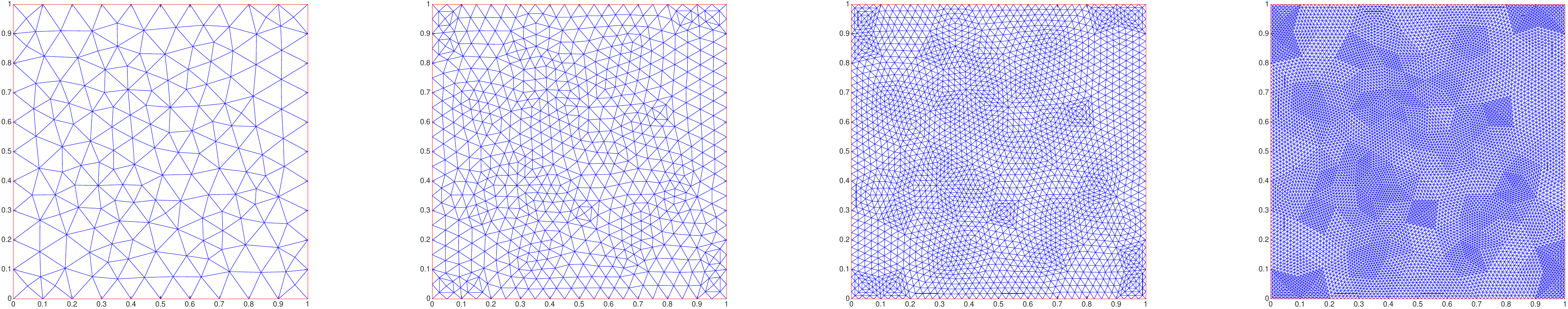}
	\caption{{Unstructured meshes for the Poisson problem, four levels of refinement using a Delaunay-based algorithm.}}
	\label{fig:unstructuredmeshes}
\end{figure}

\begin{figure}[htbp]
	\centering
	\subfloat[\texttt{HSL\_MC64} -- $\ell=1$]{\includegraphics[width=0.25\columnwidth]{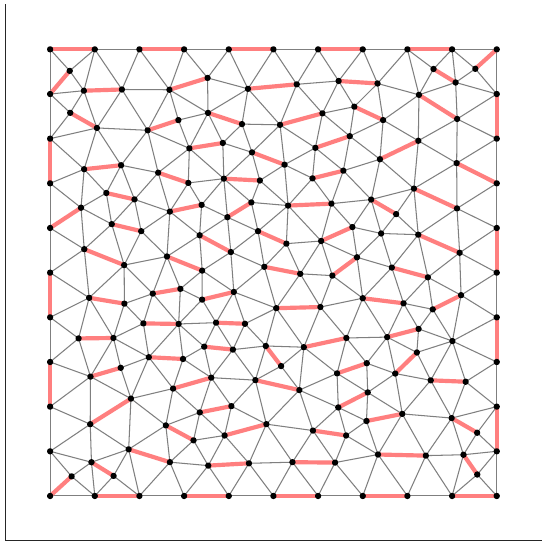}}\hfil
	\subfloat[\texttt{PREIS} -- $\ell=1$]{\includegraphics[width=0.25\columnwidth]{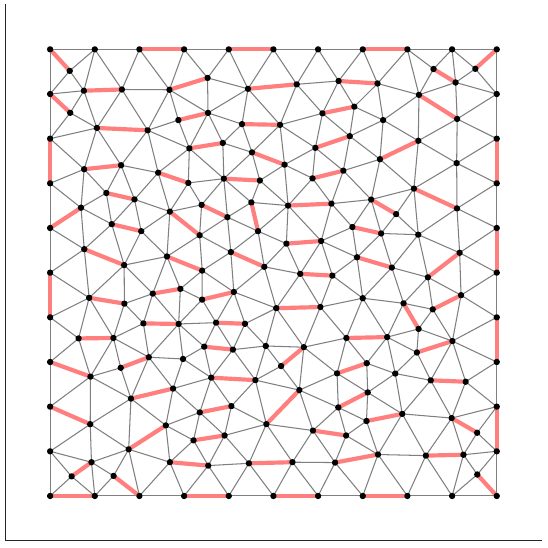}}\hfil
	\subfloat[\texttt{AUCTION} -- $\ell=1$]{\includegraphics[width=0.25\columnwidth]{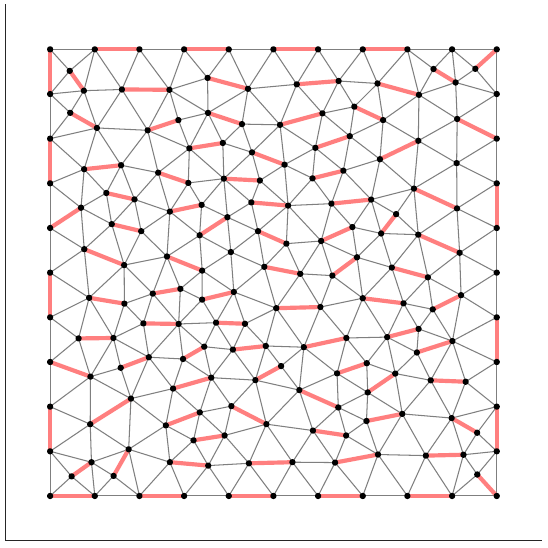}}\hfil
	\subfloat[\texttt{SUITOR} -- $\ell=1$]{\includegraphics[width=0.25\columnwidth]{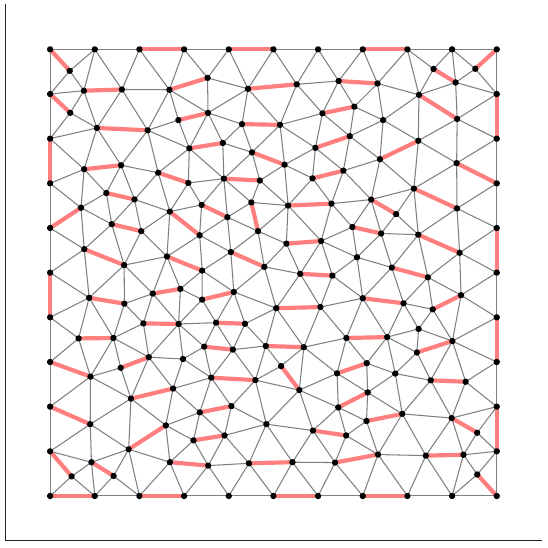}}
	\caption{{Diffusion problem with constant coefficients on an unstructured grid. Aggregates obtained with the weight vector $\mathbf{w}=(1,1,\ldots,1)^T$, and the different matching algorithms for $\ell=1$ pairwise matching steps.}}
	\label{fig:unstructured_aggregates}
\end{figure}
{The aggregates obtained for this test problem are depicted in Figure~\ref{fig:unstructured_aggregates}, whereas the  constants and bounds for $\ell=1$ step of matching are shown in Table~\ref{tab:unstructured_aggregates}. Again for this case we could not find an appropriate splitting to produce the local bound of Theorem~\ref{thm:our_convergence_result} when $\ell=2$ steps of pairwise matching were used.}
\begin{table}[htbp]
	\centering
	\caption{{Diffusion problem with constant coefficients on an unstructured grid. Comparison of the bound in Theorem~\ref{thm:our_convergence_result} with true value of $\mu_c$ in~\eqref{eq:finalcrate}. Aggregates obtained with the weight vector $\mathbf{w}=(1,1,\ldots,1)^T$, and the different matching algorithms for $\ell=1$ pairwise matching steps.}}
	\label{tab:unstructured_aggregates}
	\subfloat[\texttt{HSL\_MC64} -- exact matching]{
		\begin{tabular}{lcccc}

			dofs & bound & $\mu_c^{-1}$ \\
			\midrule
			185 & 3.000 & 1.613 \\
			697 & 3.000 & 1.562 \\
			2705 & 3.000 & 1.639 \\
			10657 & 3.000 & 1.897 \\
			
	\end{tabular}}\hfil
	\subfloat[\texttt{PREIS} -- $\frac{1}{2}$--approximate matching]{
		\begin{tabular}{lcccc}

			dofs & bound & $\mu_c^{-1}$ \\
			\midrule
			185 & 2.396 & 1.830 \\
			697 & 2.306 & 1.667 \\
			2705 & 2.258 & 2.157 \\
			10657 & 2.249 & 2.001 \\
	\end{tabular}}
	
	\subfloat[\texttt{AUCTION} -- $\frac{1}{2}$--~approximate matching]{
		\begin{tabular}{lcccc}

			dofs & bound & $\mu_c^{-1}$ \\
			\midrule
			185 & 3.000 & 1.583 \\
			697 & 3.000 & 1.596 \\
			2705 & 2.103 & 1.794 \\
			10657 & 2.106 & 1.759 \\
	\end{tabular}}\hfil
	\subfloat[\texttt{SUITOR} -- $\frac{1}{2}$--approximate matching]{
		\begin{tabular}{lcccc}

			dofs & bound & $\mu_c^{-1}$ \\
			\midrule
			185 & 2.695 & 1.686 \\
			697 & 2.484 & 1.645 \\
			2705 & 2.258 & 1.690 \\
			10657 & 2.249 & 1.893 \\
	\end{tabular}}
\end{table}
{If we compare the results in Table~\ref{tab:unstructured_aggregates} with the ones in Table~\ref{tab:hom_coeff_w_const}, then we observe that the quality of the aggregates, in this case, is analogous to the structured homogeneous case.  We also  observe that, again, the \texttt{AUCTION} algorithm manages to obtain aggregates with better quality than the ones obtained by all other algorithms, including the ones obtained by the exact matching algorithm. This is in agreement with the computational results discussed in~\cite{BootCMatch}.}

\subsection{Selecting the weight vector}\label{sec:selecting_weight_vectors}
We consider here the same test problems of the previous section, in which all the aggregates were computed by using the weight vector $\mathbf{w}=(1,1,\ldots,1)^T$, and compare them with the possible different choices for the weight vector discussed in Section~\ref{sec:selecting_the_weight}. In every case we compare the aggregates obtained by using as weight vector $\mathbf{w}$ either:
\begin{enumerate}
	\item a random initial guess, refined by some smoother iterations,
	\item the vector $\mathbf{w}=(1,1,\ldots,1)^T$, refined by some smoother iterations,
	\item the eigenvector associated with the smallest eigenvalue.
\end{enumerate}
Information on using the bootstrap procedure is contained in the Appendix Section~\ref{sec:thebootstrapprocedureresults}.

\subsubsection*{Random weight} We start considering the choice of an initial random weight vector $\mathbf{w}$ for all the test problems in Section~\ref{sec:selecting_the_weight}, and consider using as smoother for its refinement the $\ell_1$--Jacobi method~\cite{BFKY2011}; each refinement step, in this case, has a cost that is dominated by a diagonal scaling. We test the procedure for all the matching algorithms discussed in Section~\ref{sec:matching_algorithms}, but we visualize the attained aggregates only for \texttt{SUITOR}. From what we have seen in the previous section, the \texttt{SUITOR} matching algorithm consistently gives good results for all the problems, and is, from a computational point of view, the best candidate when looking for the parallel applicability of the AMG algorithms~\cite{BootCMatch}.
\begin{figure}[b]
	\centering
	\subfloat[Constant coefficient diffusion problem]{\includegraphics[width=\columnwidth]{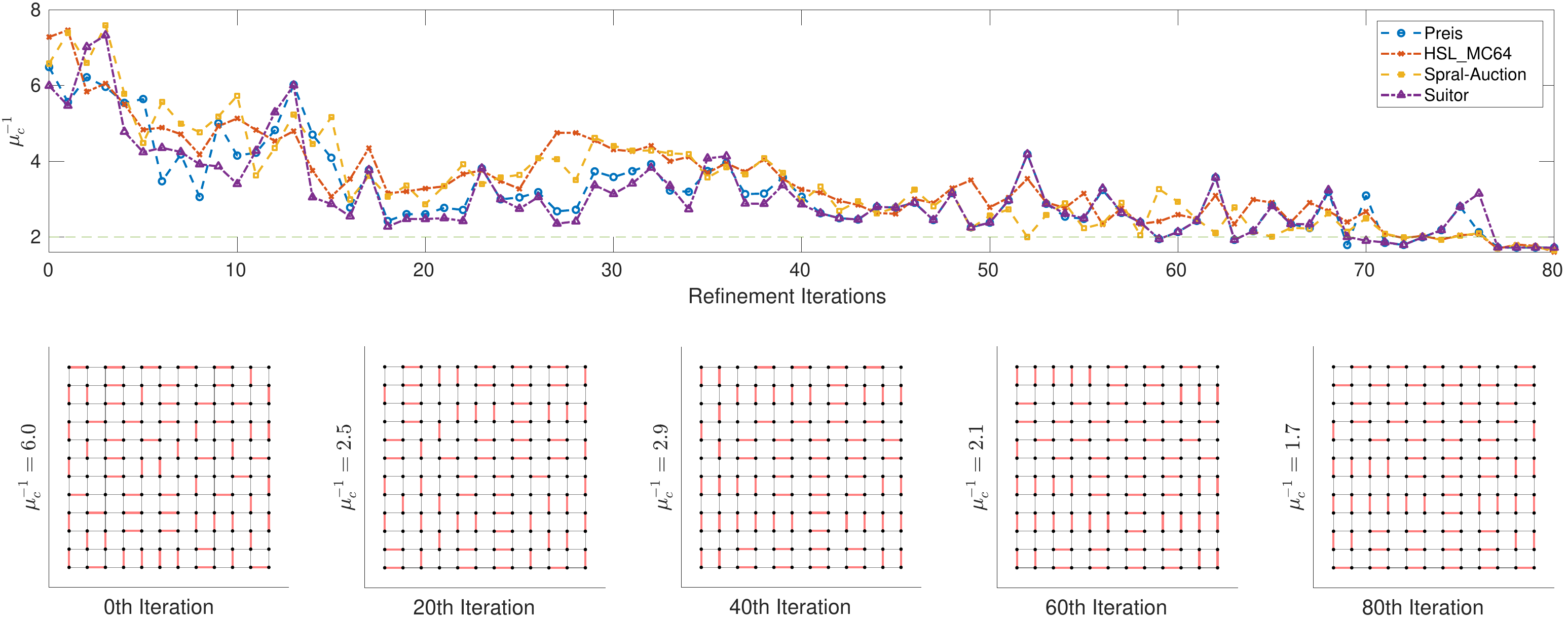}}
	
\end{figure}
\begin{figure}[t]
    \ContinuedFloat
    \centering
    \subfloat[Diffusion problem with $y$--axis oriented anisotropy $\varepsilon = 100$]{\includegraphics[width=\columnwidth]{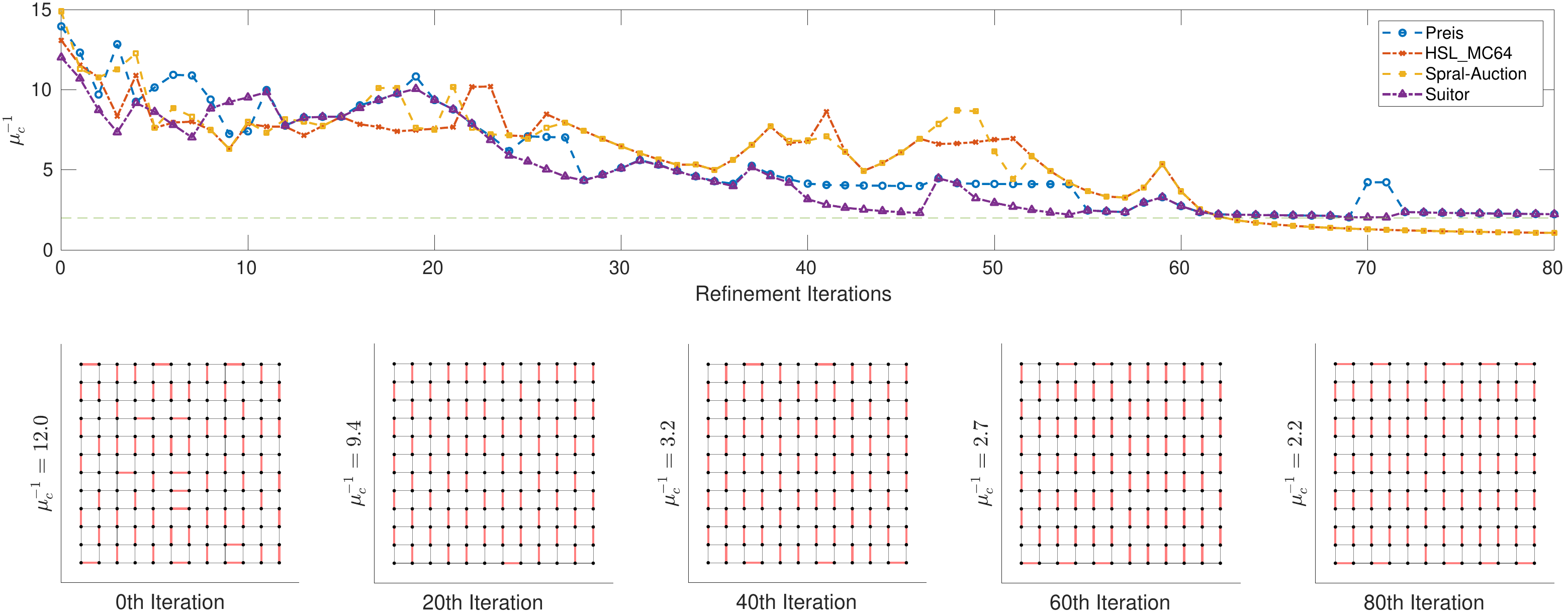}}

	\subfloat[{Constant coefficient diffusion problem on an unstructured grid}\label{fig:refinement_from_random_unstructured}]{\includegraphics[width=\columnwidth]{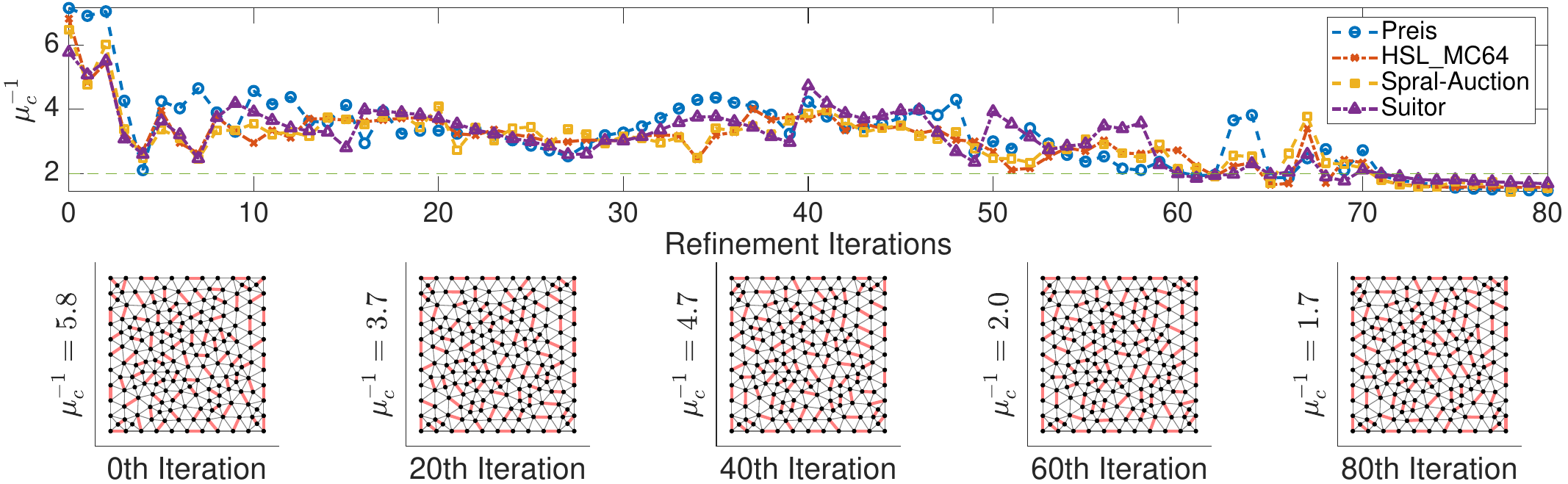}}
	
	\caption{Refinement of the weight vector starting from a random guess, and using the $\ell_1$--Jacobi smoother. We report a graph containing the $\mu_c^{-1}$ constant up to 80 refinement steps for a single sweep of pairwise aggregation. The depicted aggregates are the ones obtained with the \texttt{SUITOR} algorithm.}
	\label{fig:refinement_from_random}
\end{figure}
In Figure~\ref{fig:refinement_from_random} we report the results obtained; as we can observe, a random initial guess without any refinement is a very poor choice, and we need several refinement steps to obtain constants $\mu_c$ that are comparable with the ones we have seen in Section~\ref{sec:computing_the_muc_constant}. However, we can still go below the results obtained with the theoretical guess given by the constant weight vector $\mathbf{w}=(1,1,\ldots,1)^T$, at the cost of performing many refinement iterations. Note also that the aggregates for which these results are obtained would have been difficult to guess. 

We consider for this case also a Poisson problem with an axially rotated anisotropy of angle $\theta$ and modulus $\varepsilon$ on the same unstructured grid from Figure~\ref{fig:unstructuredmeshes}, that is, we consider the discretization of
	\begin{equation}\label{eq:rotatedpoisson}
	\begin{cases}
	- \nabla \cdot ( \mathbf{A} \nabla u ) = f, & (x,y) \in \Omega,\\
	u = 0, & (x,y) \in \partial\Omega,
	\end{cases} \qquad \mathbf{A} \in \mathbb{R}^{2\times 2}.
	\end{equation}
	Results for this test case are given in Figure~\ref{fig:aniso-poisson-refinement-from-random}.
\begin{figure}[htbp]
	\centering
	\subfloat[$\theta = \pi/6$, $\varepsilon = 1e-2$]{\includegraphics[width=0.95\columnwidth]{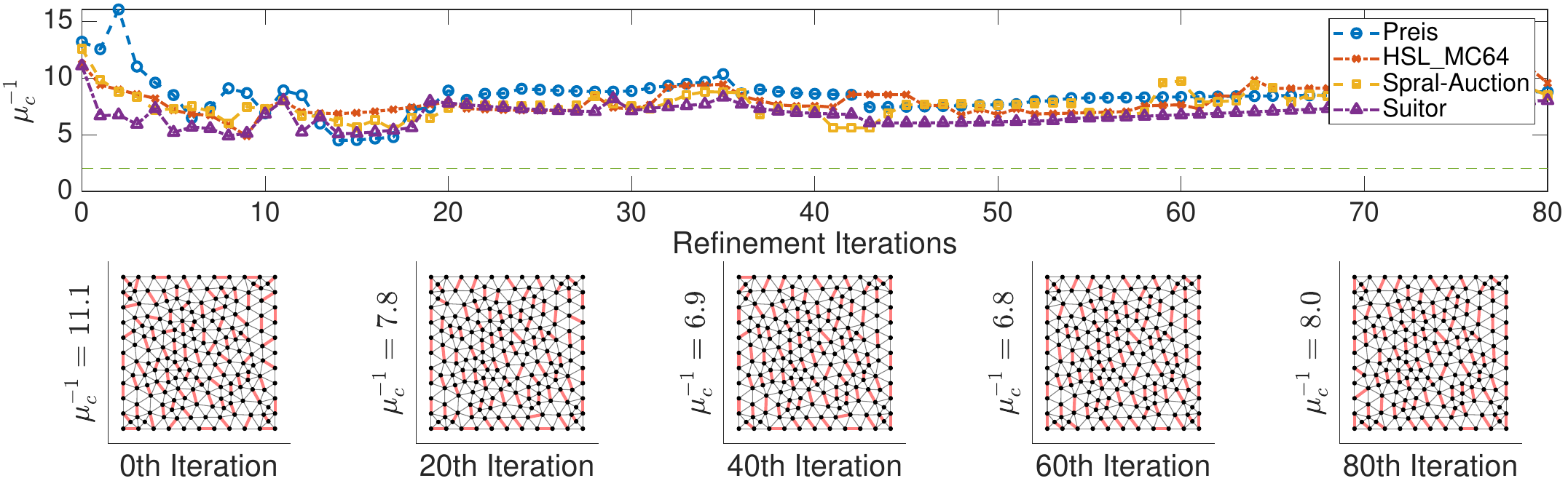}}
	
	\subfloat[$\theta = \pi/6$, $\varepsilon = 1e-3$]{\includegraphics[width=0.95\columnwidth]{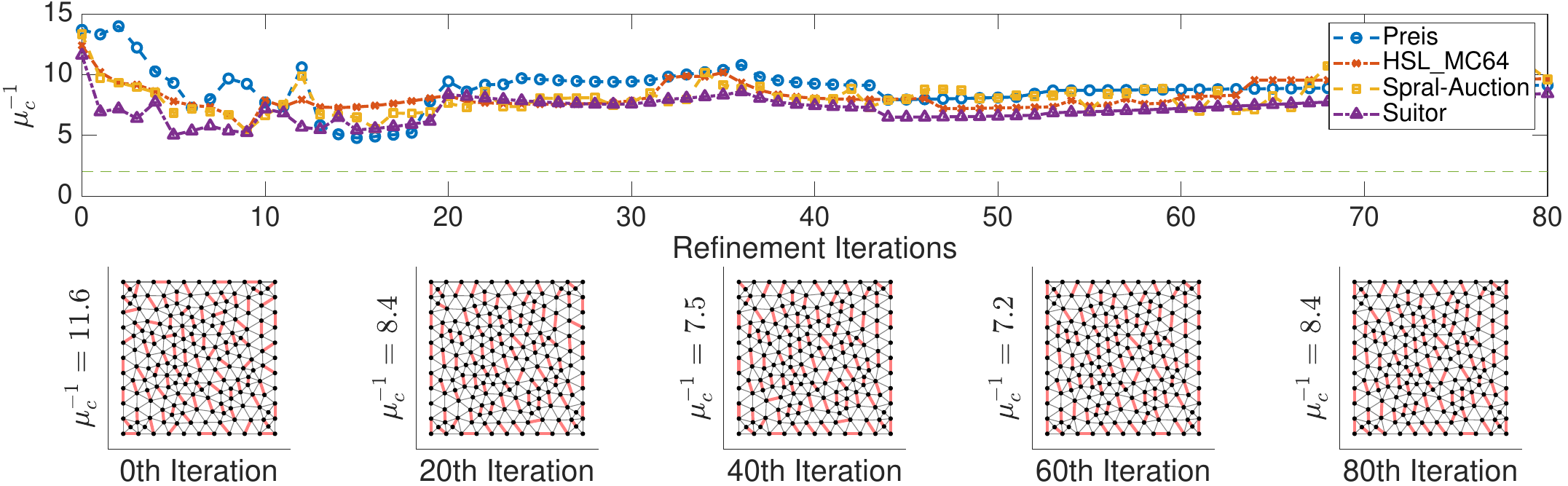}}
	
	\subfloat[$\theta = \pi/3$, $\varepsilon = 1e-2$]{{\includegraphics[width=0.95\columnwidth]{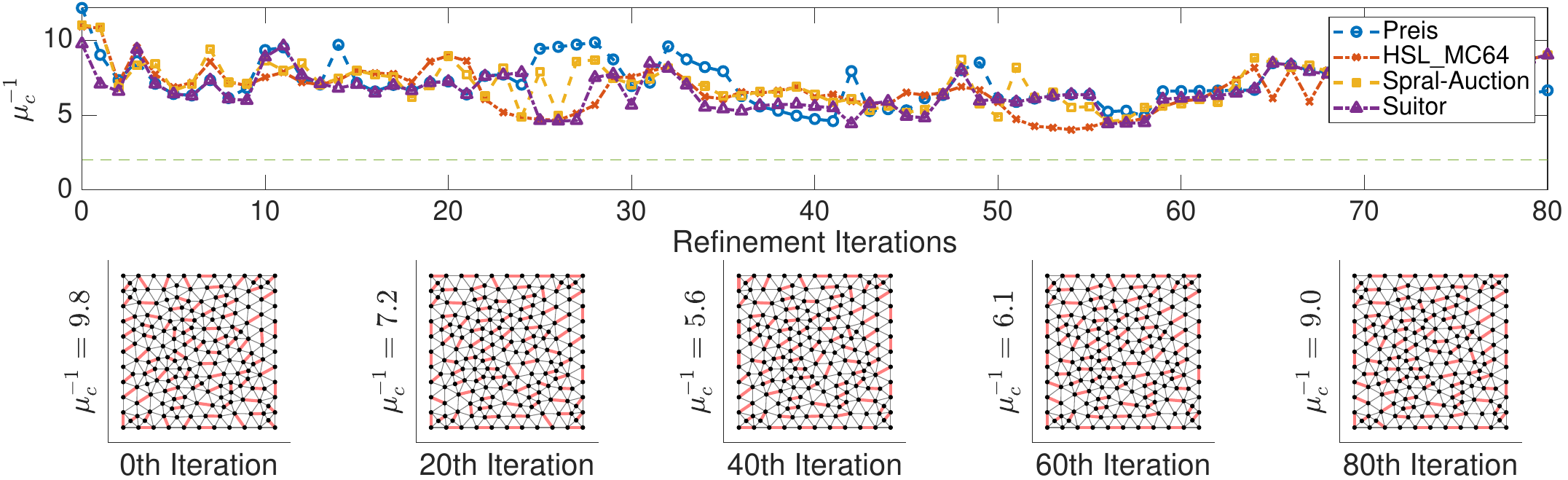}}}
	
	\subfloat[$\theta = \pi/3$, $\varepsilon = 1e-3$]{{\includegraphics[width=0.95\columnwidth]{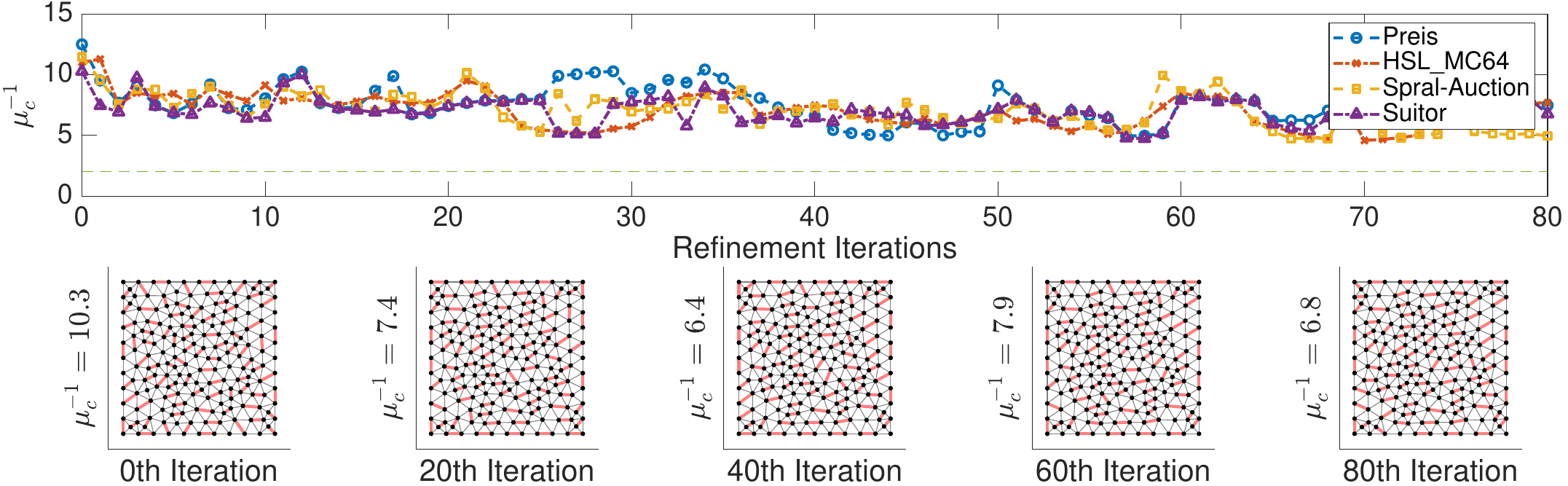}}}
	
	\caption{Poisson problem on an unstructured grid with rotated anisotropy of angle $\theta$, and modulus $\varepsilon$. Refinement of the weight vector starting from a random guess, and using the $\ell_1$--Jacobi smoother. We report a graph containing the $\mu_c^{-1}$ constant up to 80 refinement steps for a single sweep of pairwise aggregation. The depicted aggregates are the ones obtained with the \texttt{SUITOR} algorithm.}
	\label{fig:aniso-poisson-refinement-from-random}
\end{figure}

{If we compare these results with the one in Figure~\ref{fig:refinement_from_random_unstructured}, we observe that there is a moderate increase in the convergence constant for all combinations of rotation angle and modulus. Moreover, we can observe that over-refinement of the weight vector does not improve the overall quality of the aggregation procedure.}

\subsubsection*{Refined uniform weight} As we have seen from the previous set of examples, a sufficient number of refinement steps on a random weight vector $\mathbf{w}$ already improves the quality of the aggregates obtained through the matching algorithms. Therefore, we expect to obtain a similar result when we start from a more reasonable guess for the weight vector. We consider the same experimental setting and only  change the initial guess from a random $\mathbf{w}$ to the uniform vector $\mathbf{w} = (1,1,\ldots,1)^T$.
\begin{figure}[b]
	\centering
	\subfloat[Constant coefficient diffusion problem]{\includegraphics[width=\columnwidth]{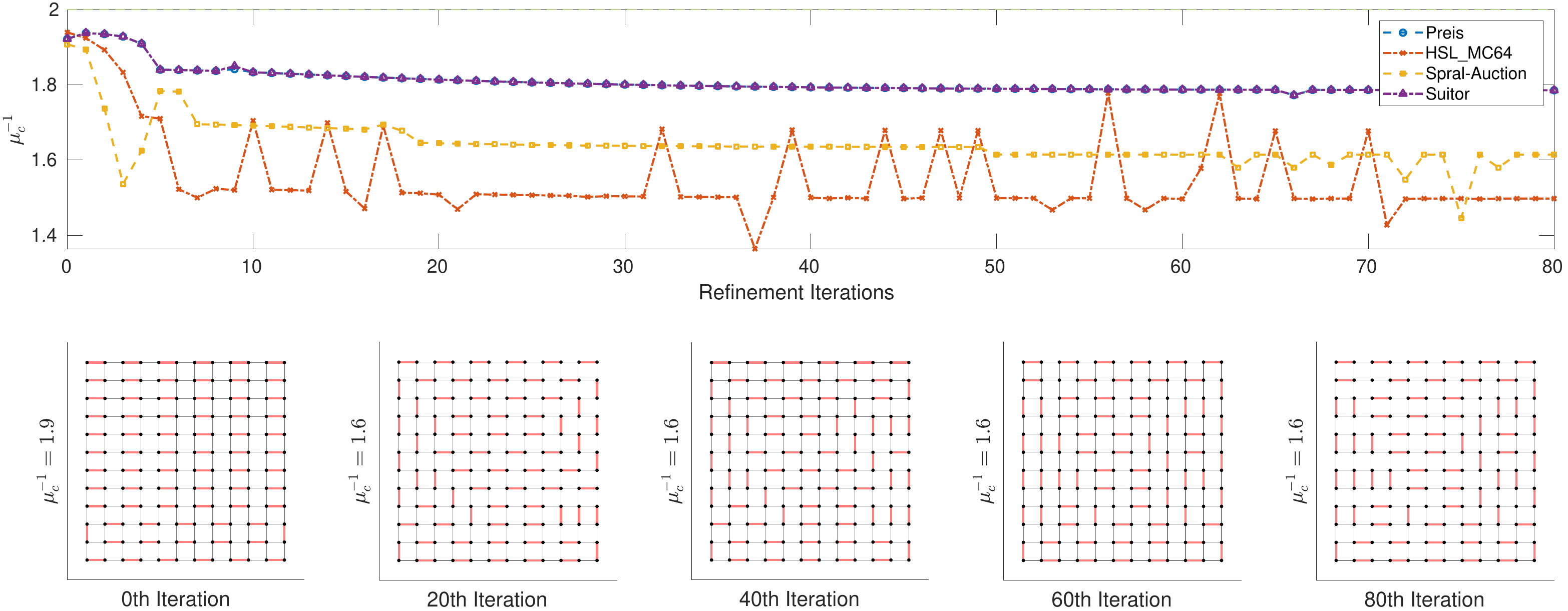}}
	
	\subfloat[Diffusion problem with $y$--axis oriented anisotropy $\varepsilon = 100$]{\includegraphics[width=\columnwidth]{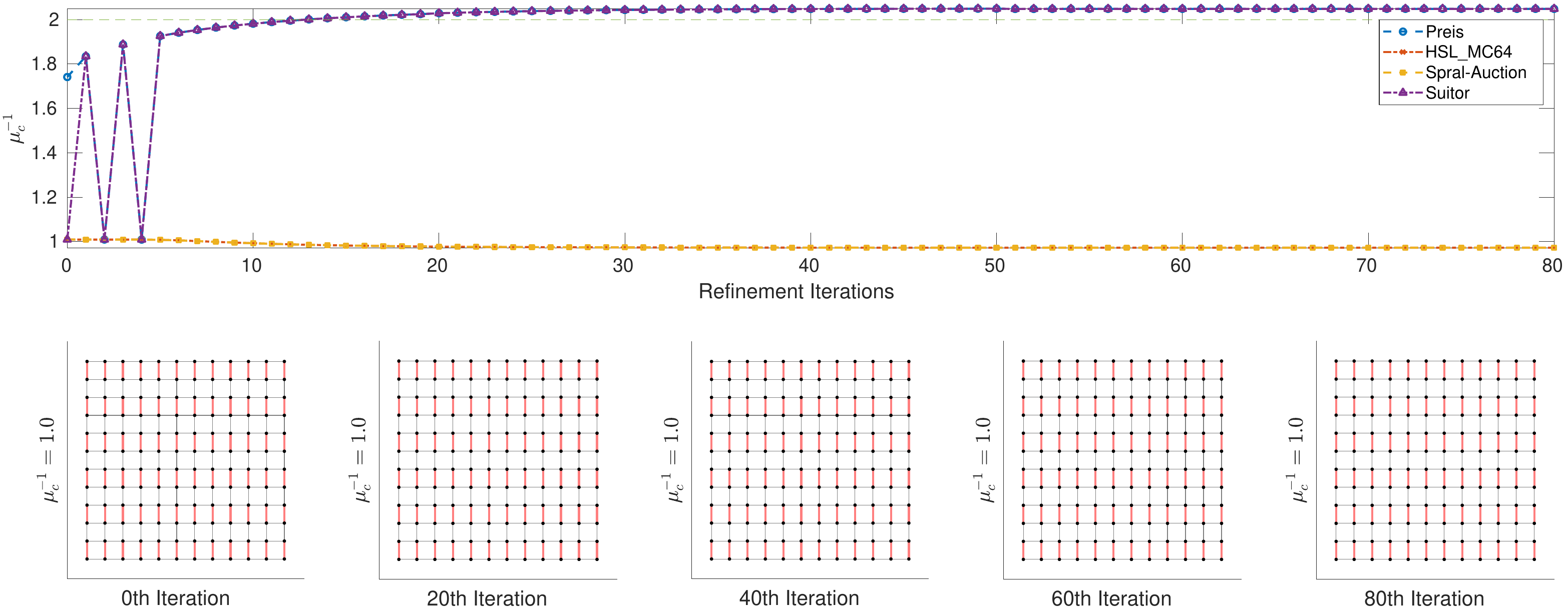}}
\end{figure}
\begin{figure}[t]
    \centering
    \ContinuedFloat

	\subfloat[{Constant coefficient diffusion problem on an unstructured grid}\label{fig:refinement_from_unit_unstructured}]{\includegraphics[width=\columnwidth]{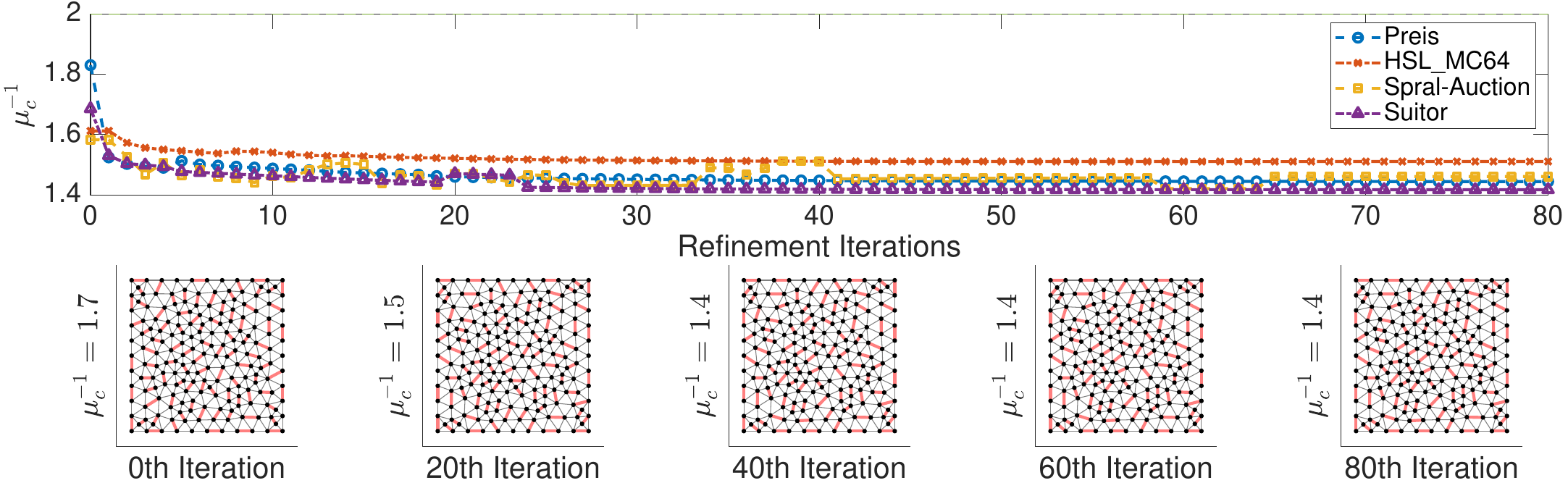}}
	
	\caption{Refinement of the weight vector starting from the all one guess, and using the $\ell_1$--Jacobi smoother. We report a graph containing the $\mu_c^{-1}$ constant up to 80 refinement steps for a single sweep of pairwise aggregation. The depicted aggregates are the ones obtained with the \texttt{AUCTION} algorithm.}
	\label{fig:refinement_from_unit}
\end{figure}
For this case, we plot in Figure~\ref{fig:refinement_from_unit}  the aggregates obtained with the \texttt{AUCTION} algorithm, which attains the best constants. What is interesting to notice in this case is that very few iterations of the smoother coupled with the \texttt{AUCTION} algorithm generate aggregates that are better than the ones obtained by the complete matching algorithm \texttt{HSL\_MC64}. The cases in which directionality in the coefficient is present end up in reproducing the expected aggregates with very few iterations.

As for the previous case, we consider again the Poisson problem on an unstructured mesh with rotated anisotropy from~\eqref{eq:rotatedpoisson}. Again, if we compare the results for this case in Figure~\ref{fig:aniso-poisson-refinement-from-allone} with the ones in Figure~\ref{fig:refinement_from_unit_unstructured} we observe that there is a decrease in the performance of the aggregation procedure. Nevertheless, a small number of refinement iterations brings the quality of the aggregates near to the one of the homogeneous case.
\begin{figure}[p]
	\centering
	\subfloat[$\theta = \pi/6$, $\varepsilon = 1e-2$]{\includegraphics[width=0.95\columnwidth]{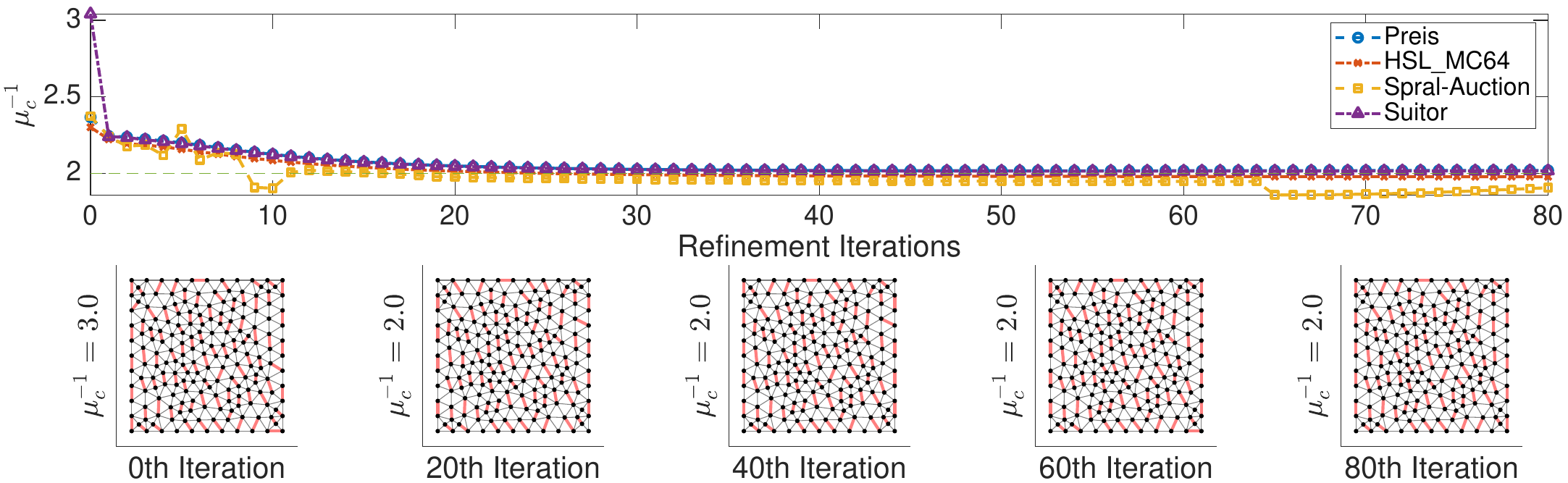}}
	
	\subfloat[$\theta = \pi/6$, $\varepsilon = 1e-3$]{\includegraphics[width=0.95\columnwidth]{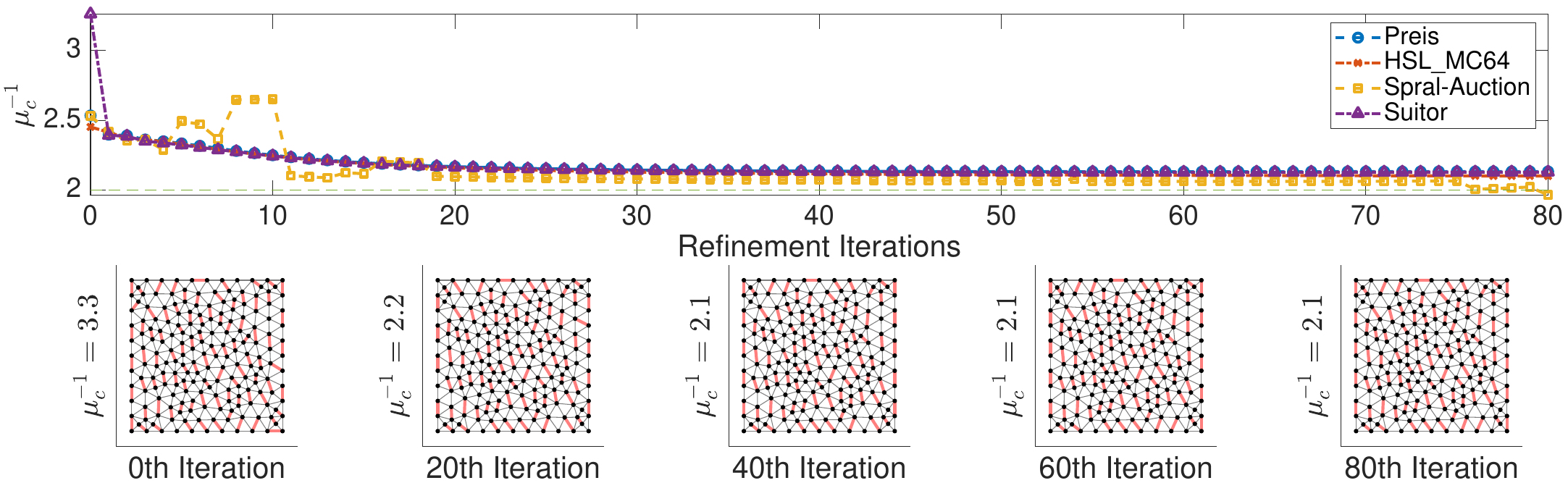}}

	\subfloat[$\theta = \pi/3$, $\varepsilon = 1e-2$]{{\includegraphics[width=0.95\columnwidth]{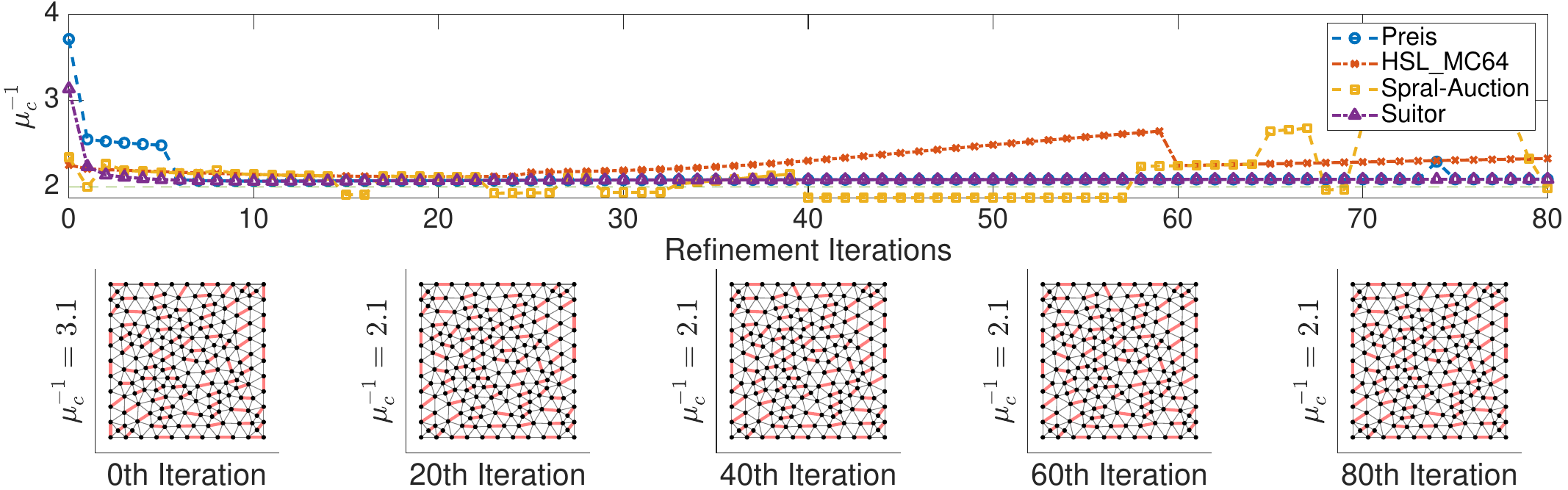}}}
	
	\subfloat[$\theta = \pi/3$, $\varepsilon = 1e-3$]{{\includegraphics[width=0.95\columnwidth]{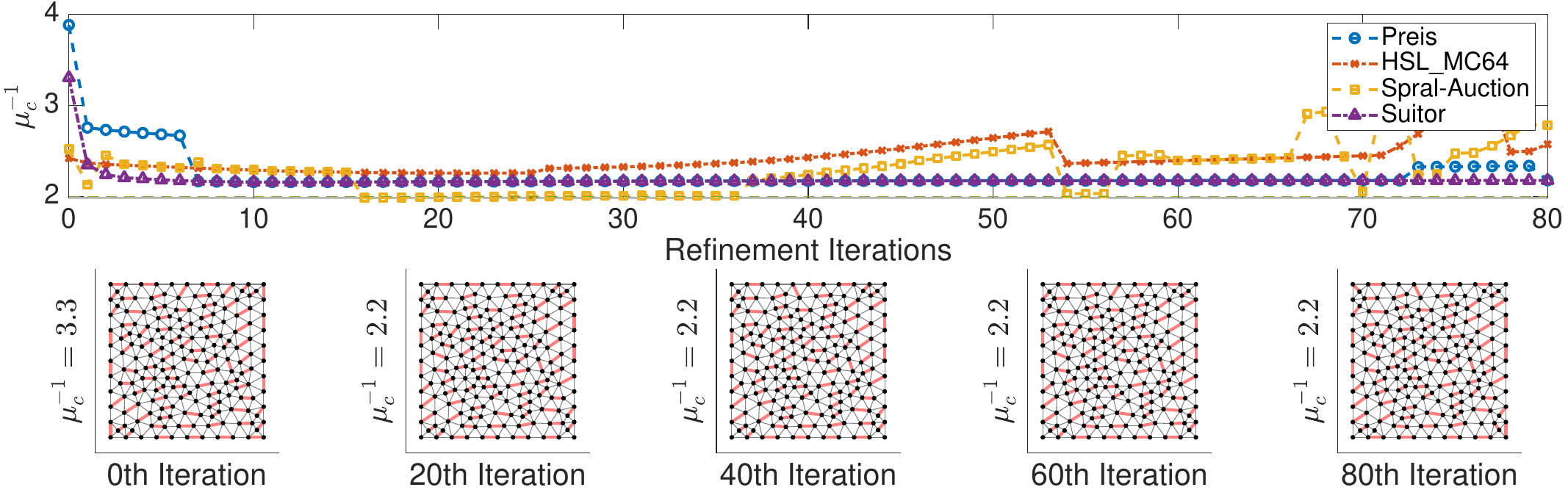}}}
	
	\caption{Poisson problem on an unstructured grid with rotated anisotropy of angle $\theta$, and modulus $\varepsilon$. Refinement of the weight vector starting from all one guess, and using the $\ell_1$--Jacobi smoother. We report a graph containing the $\mu_c^{-1}$ constant up to 80 refinement steps for a single sweep of pairwise aggregation. The depicted aggregates are the ones obtained with the \texttt{AUCTION} algorithm.}
	\label{fig:aniso-poisson-refinement-from-allone}
\end{figure}

\subsubsection*{The eigenvector weight} To complete our analysis  we consider the aggregates generated by using as weight vector $\mathbf{w}$ the eigenvector associated with the smallest eigenvalue as in Proposition~\ref{pro:at-least-is-convergent}. Since this is a theoretical test, we consider only the application of the full matching algorithm \texttt{HSL\_MC64}. We report the constants $\mu_c$ obtained by this choice in Table~\ref{tab:eigenvector_choice}.
\begin{table}[htb]
	\centering
	\caption{Constants $\mu_c^{-1}$ obtained by using as weight vector $\mathbf{w}$ the eigenvector relative to the smallest eigenvalue as suggested by Proposition~\ref{pro:at-least-is-convergent}.}
	\label{tab:eigenvector_choice}
	\begin{tabular}{ccccc}
		
		& \multicolumn{2}{c}{Homogeneous} & \multicolumn{2}{c}{$y$--axis} \\
		\cmidrule(l{2pt}r{2pt}){1-3}\cmidrule(l{2pt}r{2pt}){4-5}
		n  & $\ell = 1$ & $\ell = 2$ & $\ell = 1$ & $\ell = 2$ \\
		\cmidrule(l{2pt}r{2pt}){1-3}\cmidrule(l{2pt}r{2pt}){4-5}
		12 & 1.476 & 2.336 & 0.973 & 2.699 \\
		24 & 1.737 & 3.826 & 1.001 & 3.249  \\
		48 & 1.809 & 4.274 & 1.008 & 3.401  \\
		96 & 1.808 & 4.854 & 1.009 & 3.437 \\
	\end{tabular}
	\begin{tabular}{ccc}
		\multicolumn{3}{c}{Homogeneous unstructured} \\
		\cmidrule(l{2pt}r{2pt}){1-3}
		dofs & $\ell = 1$ & $\ell = 2$ \\
		\cmidrule(l{2pt}r{2pt}){1-3}
		185 & 1.5076& 2.1977\\
		697 & 1.5184& 2.7255\\
		2705 & 1.6349& 3.1663\\
		10657 & 1.7281& 4.0177\\
	\end{tabular}
	
\end{table}
If we compare them with the results in Tables~\ref{tab:hom_coeff_w_const},~\ref{tab:yanisotropy_poisson_aggregates}, and~\ref{tab:jumping_poisson_aggregates} we observe two different behaviors. In the case of the simpler homogeneous problem selecting the eigenvector makes for worse $\mu_c$ constants when $\ell = 2$ steps of pairwise aggregations are used with respect to the case in which the vector $\mathbf{w} = (1,1,\ldots,1)^T$ is used in Table~\ref{tab:hom_coeff_w_const}. If we look at the aggregates obtained by this choice in Figure~\ref{fig:eigenvector_choice_hom_coeff} and compare them with the one in Figure~\ref{fig:homogeneous_poisson_aggregates}, we see that the new aggregates are very far from the \emph{box} aggregates obtained in that case, this causes that for certain aggregates we get an $M$--matrix $A_k$ 
\begin{equation*}
A_k = \begin{bmatrix}
4 & & -1 & \\
& 4 & -1 & -1 \\
-1 & -1 & 4 &  \\
& -1 &  & 4
\end{bmatrix}, \qquad D_k = \begin{bmatrix}
4 \\ & 4 \\ & & 4 \\ & & & 4
\end{bmatrix},
\end{equation*}
whose scaled version $D_k^{-1} A_k$ is not a matrix with constant row sum. Therefore  the associated $\mathbf{w}_{e_k}$ is not an eigenvector, i.e., we get a $\mu_k$ constant that is intermediate between $\lambda_1$ and $\lambda_2$, as discussed in Theorem~\ref{thm:our_convergence_result}. On the other hand, the constant vector choice always provides an irreducible and diagonally dominant $M$--matrix $D_k^{-1} A_k$, hence the vector $\mathbf{w}_{e_k} = (1,1,1,1)^T$ is the unique eigenvector associated with the smallest eigenvalue, thus we obtain a better constant. 
\begin{figure}[htbp]
	\centering
	\subfloat[Constant coefficient diffusion problem\label{fig:eigenvector_choice_hom_coeff}]{
		\includegraphics[width=0.21\columnwidth]{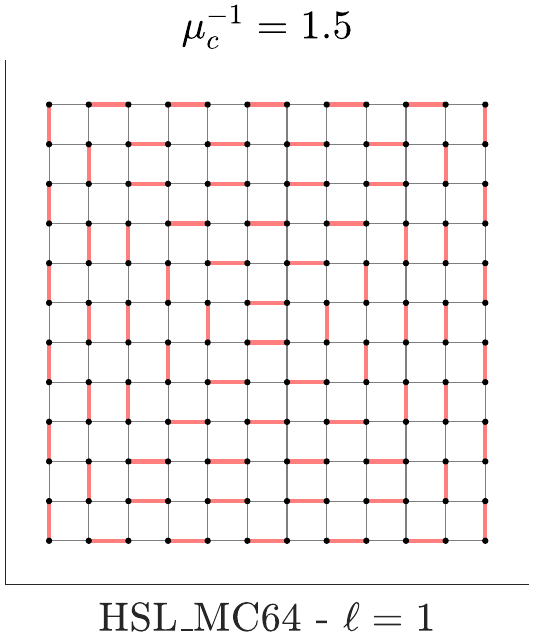}
		\includegraphics[width=0.21\columnwidth]{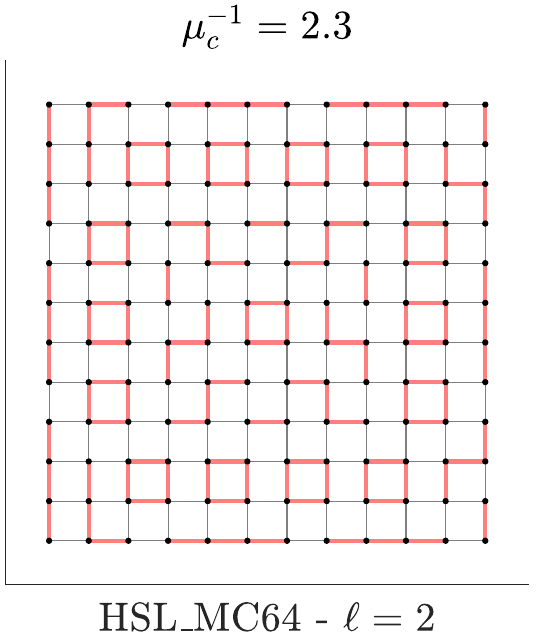}
	}\hfil
	\subfloat[Diffusion problem with $y$--axis oriented anisotropy $\varepsilon = 100$]{
		\includegraphics[width=0.21\columnwidth]{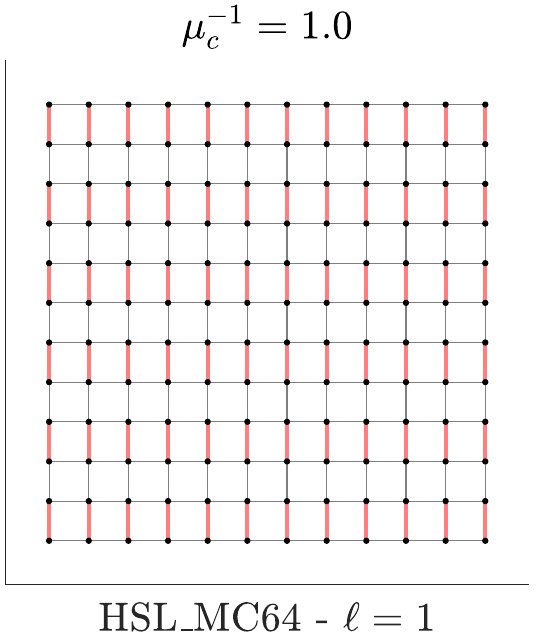}
		\includegraphics[width=0.21\columnwidth]{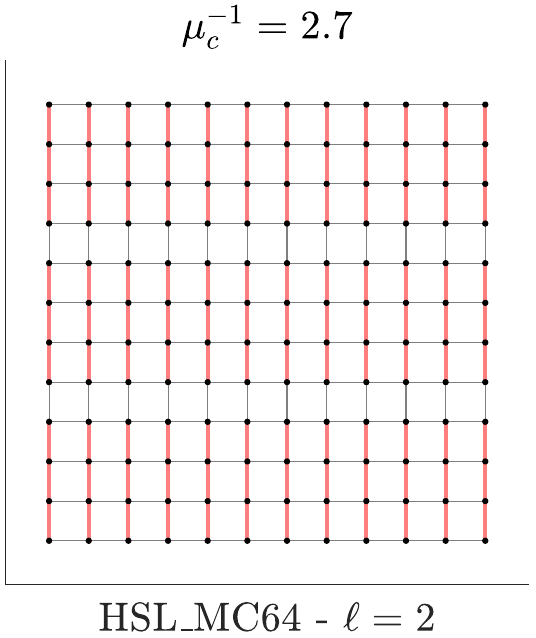}
	}
	
	\caption{Aggregates obtained by using as weight vector $\mathbf{w}$ the eigenvector associated with the smallest eigenvalue as suggested by Proposition~\ref{pro:at-least-is-convergent}}
	\label{fig:eigenvector_choice}
\end{figure}
Focusing now on the other cases in Table~\ref{tab:eigenvector_choice}, whose aggregates are also depicted in Figure~\ref{fig:eigenvector_choice}, we obtain nearly the same results with the exception of the piecewise regular coefficients in which we are able to improve the attained constants -- observe also that they are near the one obtained with the \texttt{SUITOR} algorithm and the $\mathbf{w} = (1,\ldots,1)^T$ vector, even if the aggregates are very different.

What we can conclude from testing the usage of the eigenvector associated with the smallest eigenvalue is that, although guaranteeing the convergence due to Proposition~\ref{pro:at-least-is-convergent}, it can generate sub-optimal aggregates. On the other hand, either selecting a vector knowing the structure of the matrices $\{A_k\}_k$, as in the constant coefficient case with the $\mathbf{w}=(1,1,\ldots,1)^T$ vector or refining a choice by means of the smoothing procedure, can yield better results as we have seen. 

\section{Quality of the aggregates and the compatible relaxation principle}
\label{sec:compatiblerel}

As already mentioned in Section~\ref{sec:intro}, the need to measure the quality of a coarse space and to set up a general procedure for coarsening of
the widest range of linear systems led to the nice principle of {\em compatible relaxation}. After its introduction in~\cite{BrandtCompRelax}, it has been widely analyzed and related to the general theories for AMG convergence in many papers, starting from~\cite{FalgoutVassilevskyMeasure}. This principle has been applied as a guideline to define the coarsening method described in this paper, as emphasized in the original papers~\cite{DV2013,BootCMatch}. In the following, we show that the results obtained by the quality measure discussed in this paper are in good agreement with a quality measure based on the convergence rate of a compatible relaxation, showing the coherence of the convergence theories. Main advantage in using the constant $\mu_c$ in~\eqref{eq:convergence_rate_bound} is that it does not depend on a selected smoother and often gives more accurate information on the quality of the coarse space, as also shown in some of our experiments. Furthermore, we observe that the setup of a compatible relaxation scheme requires to build in an explicit way the complementary space to the coarse space, as explained in the following.

To introduce the measure based on compatible relaxation, we need to define the following $2\times 2$--block factorization
\begin{equation}\label{eq:blockfactorization}
\begin{bmatrix}
P_f^T \\P^T
\end{bmatrix} A \begin{bmatrix}
P_f & P
\end{bmatrix} = \begin{bmatrix}
A_{ff} & A_{fc} \\
A_{cf} & A_{cc}
\end{bmatrix}, \qquad \text{ for } P^T D P_f = 0,
\end{equation}
where $\Range(P_f)$ is the space in which the smoother should be effective; this  can be used to obtain a decomposition of the whole $\mathbb{R}^n$ since for all $\mathbf{e}\in\mathbb{R}^n$ we have $\mathbf{e} = P_f \mathbf{e}_f + P \mathbf{e}_c$. Exploiting the observation in Remark~\ref{rmk:dorthognonality}, we can express the matrix $P_f$ through the block factorization~\eqref{eq:blockfactorization} in a straightforward way as
\begin{equation*}
P_f = \begin{bmatrix}
\tilde{P}_f \\
0
\end{bmatrix} \in \mathbb{R}^{n \times n_p}, \text{ where } \tilde{P}_f = [\mathbf{p}_1^f,\ldots,\mathbf{p}_{n_p}^f] \text{ for }  \mathbf{p}_j^f = \Pi_j \mathbf{w}_{e_{i\mapsto j}}^\perp.
\end{equation*}
By this construction, each relaxation scheme that is well defined for the block $A_{ff}$ is then a \emph{compatible relaxation}, i.e., a scheme that
keeps the values of the coarse variables intact, and therefore makes the smoothing and coarse correction operators work each on the appropriate subspaces.

To validate numerically this claim we then look at the convergence radius $\rho(\cdot)$ of the iterative method induced by the restriction of the $\ell_1$--Jacobi global smoother on the matrix $A_{ff}$ in~\eqref{eq:blockfactorization}, i.e., we look at
\begin{equation}\label{eq:cr_convratio}
\rho_f = \rho( I - M_{ff}^{-1} A_{ff}) <  1, \quad M_{ff} = P_f^T M P_f, \qquad A_{ff} = P_f^T A P_f,
\end{equation}
where $M$ is the iteration matrix of the $\ell_1$--Jacobi global smoother for $A$. In Table~\ref{tab:convergence_ratio} we report the value of $\rho_f$ for each combination of test problem and matching algorithm, while setting the weight vector $\mathbf{w} = (1,1,\ldots,1)^T$, and the number of matching steps to $\ell = 1$. 
\begin{table}[htbp]
	\centering
	\caption{Convergence ratio $\rho_f$ of the compatible relaxation scheme~\eqref{eq:cr_convratio} for all the test problems. The coarse space is built from a single step of all the matching algorithms from Section~\ref{sec:matching_algorithms} with weight vector choice $\mathbf{w} = (1,1,\ldots,1)^T$, and no refinement iterations.}
	\label{tab:convergence_ratio} 
	\subfloat[Constant coefficient diffusion problem]{
		\begin{tabular}{ccccc}
			
			$n$   & \texttt{HSL\_MC64} &\texttt{PREIS} & \texttt{AUCTION} & \texttt{SUITOR} \\
			\cmidrule(l{2pt}r{2pt}){1-1}\cmidrule(l{2pt}r{2pt}){2-2}\cmidrule(l{2pt}r{2pt}){3-3}\cmidrule(l{2pt}r{2pt}){4-4}\cmidrule(l{2pt}r{2pt}){5-5} 
			12 & 0.766 & 0.794 & 0.755 & 0.794 \\
			24 & 0.816 & 0.824 & 0.805 & 0.824 \\
			48 & 0.826 & 0.831 & 0.829 & 0.832 \\
			96 & 0.832 & 0.833 & 0.832 & 0.833 \\
			
	\end{tabular}}\hfil
	\subfloat[Diffusion problem with $y$--axis oriented anisotropy $\varepsilon = 100$]{
		\begin{tabular}{ccccc}
			
			$n$   & \texttt{HSL\_MC64} &\texttt{PREIS} & \texttt{AUCTION} & \texttt{SUITOR} \\
			\cmidrule(l{2pt}r{2pt}){1-1}\cmidrule(l{2pt}r{2pt}){2-2}\cmidrule(l{2pt}r{2pt}){3-3}\cmidrule(l{2pt}r{2pt}){4-4}\cmidrule(l{2pt}r{2pt}){5-5} 
			12 & 0.969 & 0.963 & 0.978 & 0.963 \\
			24 & 0.986 & 0.986 & 0.989 & 0.986 \\
			48 & 0.993 & 0.777 & 0.994 & 0.871 \\
			96 & 0.996 & 0.994 & 0.996 & 0.994 \\
			
	\end{tabular}}
	
	\subfloat[][{Rotated anisotropy $\theta = \pi/6$ and $\varepsilon=100$ on an}\\ {unstructured grid}]{
		\begin{tabular}{ccccc}
			dofs   & \texttt{HSL\_MC64} &\texttt{PREIS} & \texttt{AUCTION} & \texttt{SUITOR} \\
			\cmidrule(l{2pt}r{2pt}){1-1}\cmidrule(l{2pt}r{2pt}){2-2}\cmidrule(l{2pt}r{2pt}){3-3}\cmidrule(l{2pt}r{2pt}){4-4}\cmidrule(l{2pt}r{2pt}){5-5} 	
			185 & 0.803 & 0.802 & 0.796 & 0.799 \\
			697 & 0.831 & 0.846 & 0.811 & 0.843 \\
			2705 & 0.851 & 0.863 & 0.862 & 0.854 \\
			10657 & 0.882 & 0.917 & 0.873 & 0.869 \\
	\end{tabular}}

	\subfloat[{Constant coefficient problem on an unstructured grid}]{
		\begin{tabular}{ccccc}
			
			dofs   & \texttt{HSL\_MC64} &\texttt{PREIS} & \texttt{AUCTION} & \texttt{SUITOR} \\
			\cmidrule(l{2pt}r{2pt}){1-1}\cmidrule(l{2pt}r{2pt}){2-2}\cmidrule(l{2pt}r{2pt}){3-3}\cmidrule(l{2pt}r{2pt}){4-4}\cmidrule(l{2pt}r{2pt}){5-5} 
			185 & 0.723 & 0.756 & 0.729 & 0.725 \\
			697 & 0.735 & 0.750 & 0.754 & 0.743 \\
			2705 & 0.746 & 0.788 & 0.770 & 0.768 \\
			10657 & 0.785 & 0.794 & 0.775 & 0.800 \\
	\end{tabular}}

\end{table}

If we compare the  constants obtained here with the ones in the columns for $\ell=1$ in the Tables~\ref{tab:hom_coeff_w_const}, \ref{tab:yanisotropy_poisson_aggregates}
and~\ref{tab:unstructured_aggregates}, we observe that the value of the $\rho_f$ constants behaves consistently with quality measure $\mu_c$ within the same experiment, while it is harder to use it to compare among the aggregates for different test cases. {This is specifically true  for the case of the unstructured mesh, where, even if the quality of the aggregates seem to be degraded with respect to the corresponding finite difference case, the convergence ratio of the compatible relaxation is only mildly affected.}

\section{Conclusions}
\label{sec:concl}

This paper has presented some theoretical results which complement the available computational evidence on the convergence properties of the {\em coarsening based on compatible weighted matching}. This is a purely algebraic and automatic procedure, exploiting unsmoothed aggregation for coarsening of general SPD matrices in AMG, introduced in~\cite{DV2013,BootCMatch}. We have shown that the necessary conditions for convergence of AMG, as stated in~\cite{XuZikatanovReview}, are satisfied. Furthermore, we {used} the theory 
to have a quality measure of aggregates which we used as {\em a posteriori} guideline to analyze the effectiveness of different edge weights and maximum weight matching algorithms exploited in the coarsening procedure. We have applied the theory to different test cases arising from scalar elliptic PDEs, and we have shown that the good quality of the coarsening procedure is preserved in the case of using sub-optimal algorithms for computing maximum weight matching and that it appears also insensitive to anisotropy and discontinuities in the coefficients of the considered test cases.

\backmatter

\section*{Declarations}

\subsection*{Funding}

The research leading to these results received funding from Horizon 2020 Project  ``Energy oriented Centre of Excellence: toward exascale for energy'' (EoCoE--II), Project ID: 824158. The first three authors are members of the INdAM--GNCS research group.

\subsection*{Conflict of interest}

All authors certify that they have no affiliations with or involvement in any organization or entity with any financial interest or non-financial interest in the subject matter or materials discussed in this manuscript.

\subsection*{Data availability}

The datasets generated during and/or analysed during the current study are available in the GitHub repository, \url{https://github.com/bootcmatch/BootCMatch}.

\begin{appendices}

\section{Additional experiments}\label{secA1:additional_experiments}

To make the main text easier to read, we have collected here some additional experiments on different test cases. The construction of the section mirrors that of Section~\ref{sec:numerical_experiments}. We further discuss here the use of the \emph{bootstrap procedure} for the selection of the weight vector $\mathbf{w}$ in Section~\ref{sec:thebootstrapprocedureresults}.

\subsection{Computing the \texorpdfstring{$\mu_c$}{muc} constant}\label{sec:appendix_computing_muc}

We collect here other test cases for which we have computed the $\mu_c$ constant within the same framework of Section~\ref{sec:computing_the_muc_constant}. That is, we consider again the Poisson problem~\eqref{eq:thediffusionproblem} but with different coefficients. Specifically, Section~\ref{sec:diffusion_with_jumping_coefficients} discusses the case of a discontinuous diffusion coefficient, while Section~\ref{sec:random_coefficients_appendix} is about the case of a diffusion coefficient sampled from a random distribution.

\subsubsection{Diffusion with jumps in the coefficients}\label{sec:diffusion_with_jumping_coefficients}
{We consider now} the case in which the diffusion coefficient is only piece-wise regular, exhibiting jumps between two values in the $\Omega = [0,1]^2$ domain. For these kind of problems there is some numerical evidence showcasing the efficiency of aggregation-based AMG methods~\cite{N2010,NN2011}, yet it is interesting to evaluate how a fully algebraic and unsmoothed aggregation procedure behaves. For our test, we consider the case
\begin{equation*}
a(x,y) = \begin{cases}
3, & x > \nicefrac{1}{2} \text{ and } y > \nicefrac{1}{2},\\
1, & \text{otherwise}.
\end{cases}
\end{equation*}
We report the aggregates obtained for this test problem in Figure~\ref{fig:jumping_poisson_aggregates}.
\begin{figure}[htbp]
	\centering
	\subfloat[\texttt{HSL\_MC64} -- $\ell=1$]{\includegraphics[width=0.25\columnwidth]{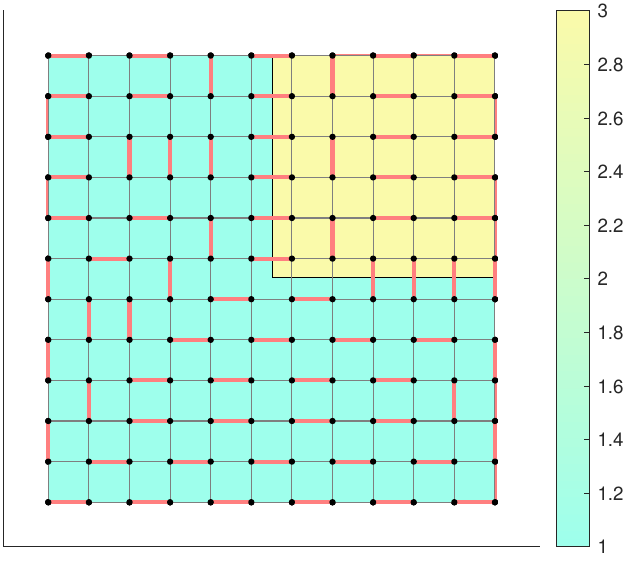}}
	\subfloat[\texttt{HSL\_MC64} -- $\ell=2$]{\includegraphics[width=0.25\columnwidth]{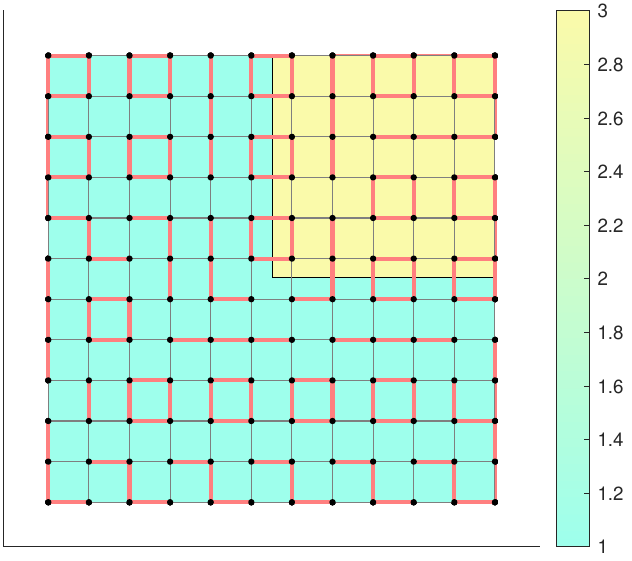}}
	\subfloat[\texttt{PREIS} -- $\ell=1$]{\includegraphics[width=0.25\columnwidth]{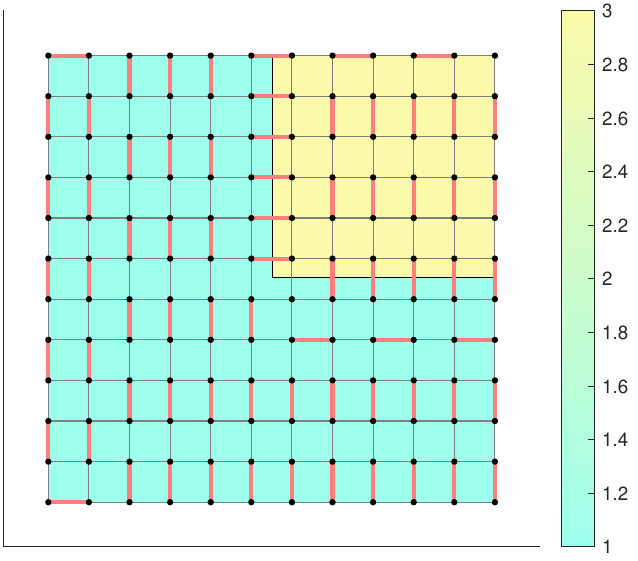}}
	\subfloat[\texttt{PREIS} -- $\ell=2$]{\includegraphics[width=0.25\columnwidth]{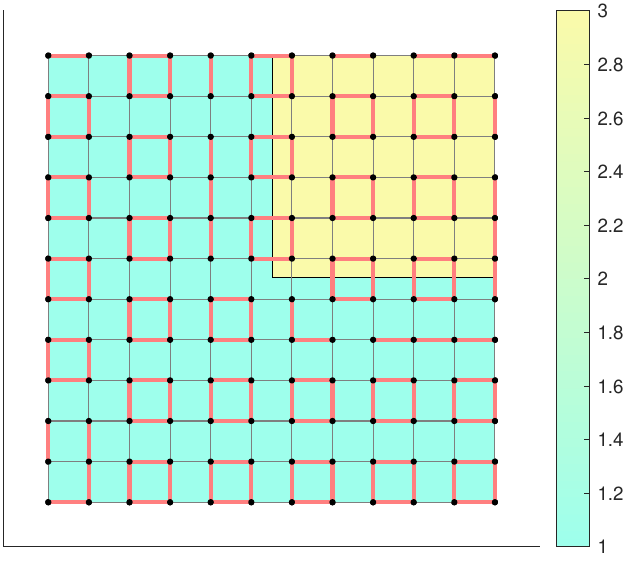}}
	
	\subfloat[\texttt{AUCTION} -- $\ell=1$]{\includegraphics[width=0.25\columnwidth]{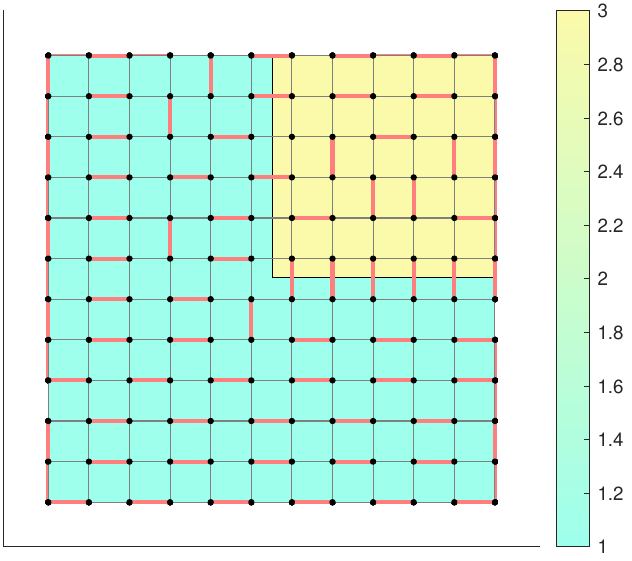}}
	\subfloat[\texttt{AUCTION} -- $\ell=2$]{\includegraphics[width=0.25\columnwidth]{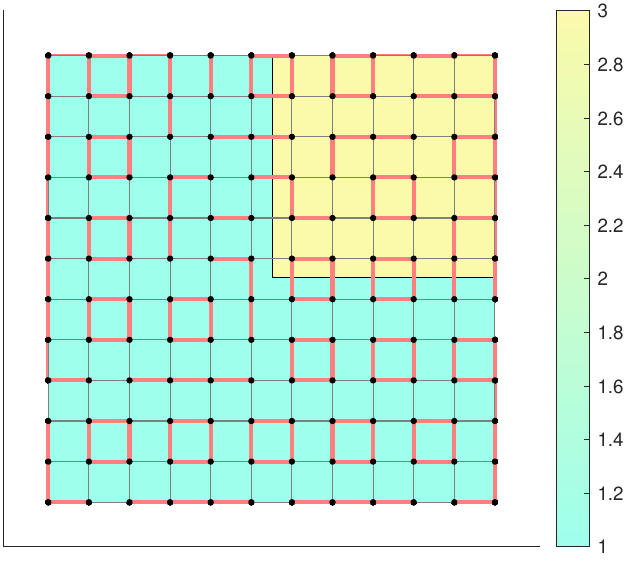}}
	\subfloat[\texttt{SUITOR} -- $\ell=1$]{\includegraphics[width=0.25\columnwidth]{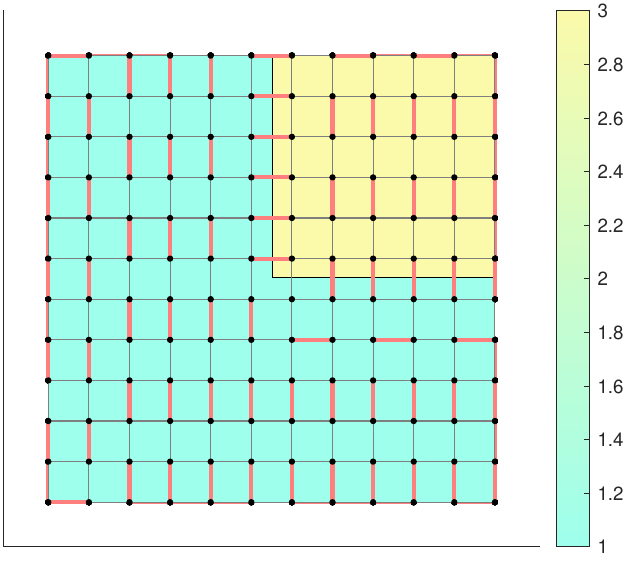}}
	\subfloat[\texttt{SUITOR} -- $\ell=2$]{\includegraphics[width=0.25\columnwidth]{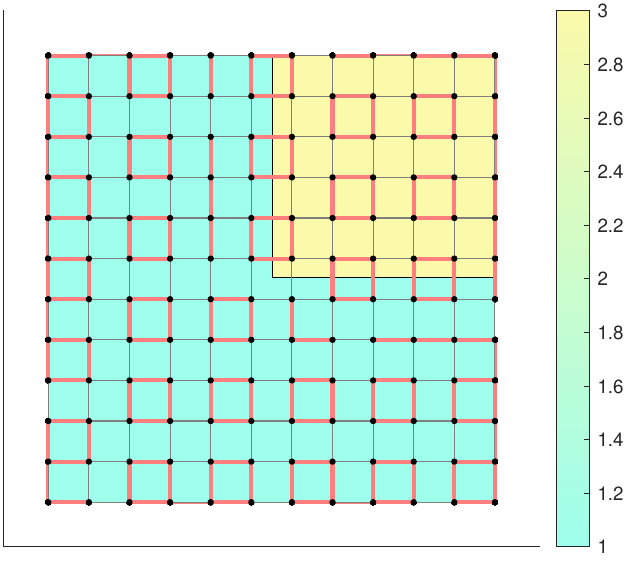}}
	\caption{Diffusion problem with piece-wise regular coefficients jumping between two values. Aggregates obtained with the weight vector $\mathbf{w}=(1,1,\ldots,1)^T$, and the different matching algorithms for $\ell=1,2$ pairwise matching steps. The colored regions represents the value of the diffusion coefficient $a(x,y)$ on the domain $\Omega$.}
	\label{fig:jumping_poisson_aggregates}
\end{figure}
To evaluate the quality of the attained aggregates we can look again at the $\mu_c$ constant given in Table~\ref{tab:jumping_poisson_aggregates}.
\begin{table}[htbp]
	\centering
	\caption{Diffusion problem with piece-wise regular coefficients jumping between two values. Comparison of the bound in Theorem~\ref{thm:our_convergence_result} with true value of $\mu_c$ in~\eqref{eq:finalcrate}. Aggregates obtained with the weight vector $\mathbf{w}=(1,1,\ldots,1)^T$, and the different matching algorithms for $\ell=1,2$ pairwise matching steps.}
	\label{tab:jumping_poisson_aggregates}
	\subfloat[\texttt{HSL\_MC64} -- exact matching]{
		\begin{tabular}{lcccc}
			
			& \multicolumn{2}{c}{$\ell = 1$} & \multicolumn{2}{c}{$\ell = 2$} \\
			\cmidrule(l{2pt}r{2pt}){2-3}\cmidrule(l{2pt}r{2pt}){4-5}
			n & bound & $\mu_c^{-1}$ & bound & $\mu_c^{-1}$ \\
			\midrule
			12 & 2.205 & 1.730 & 3.385 & 2.964 \\
			24 & 2.205 & 1.934 & 3.385 & 3.762 \\
			48 & 2.000 & 1.978 & 5.766 & 4.713 \\
			96 & 2.205 & 1.996 & 5.399 & 5.399 \\
			
	\end{tabular}}\hfil
	\subfloat[\texttt{PREIS} -- $\frac{1}{2}$--approximate matching]{
		\begin{tabular}{lcccc}
			
			& \multicolumn{2}{c}{$\ell = 1$} & \multicolumn{2}{c}{$\ell = 2$} \\
			\cmidrule(l{2pt}r{2pt}){2-3}\cmidrule(l{2pt}r{2pt}){4-5}
			n & bound & $\mu_c^{-1}$ & bound & $\mu_c^{-1}$ \\
			\midrule
			12 & 2.000 & 1.885 & 2.985 & 2.428 \\
			24 & 2.000 & 1.980 & 3.360 & 2.714 \\ 
			48 & 2.477 & 1.996 & 4.135 & 3.092 \\
			96 & 2.000 & 1.999 & 3.360 & 2.772 \\
			
	\end{tabular}}
	
	\subfloat[\texttt{AUCTION} -- $\frac{1}{2}$--approximate matching]{
		\begin{tabular}{lcccc}
			
			& \multicolumn{2}{c}{$\ell = 1$} & \multicolumn{2}{c}{$\ell = 2$} \\
			\cmidrule(l{2pt}r{2pt}){2-3}\cmidrule(l{2pt}r{2pt}){4-5}
			n & bound & $\mu_c^{-1}$ & bound & $\mu_c^{-1}$ \\
			\midrule
			12 & 2.000 & 1.761 & 3.745 & 2.684 \\
			24 & 2.000 & 1.933 & 4.820 & 3.122 \\
			48 & 2.000 & 1.981 & 5.606 & 3.402 \\
			96 & 2.000 & 1.995 & 4.000 & 3.846 \\
			
	\end{tabular}}\hfil
	\subfloat[\texttt{SUITOR} -- $\frac{1}{2}$--approximate matching]{
		\begin{tabular}{lcccc}
			
			& \multicolumn{2}{c}{$\ell = 1$} & \multicolumn{2}{c}{$\ell = 2$} \\
			\cmidrule(l{2pt}r{2pt}){2-3}\cmidrule(l{2pt}r{2pt}){4-5}
			n & bound & $\mu_c^{-1}$ & bound & $\mu_c^{-1}$ \\
			\midrule
			12 & 2.000 & 1.885 & 2.985 & 2.424 \\
			24 & 2.000 & 1.980 & 3.360 & 2.714 \\
			48 & 2.477 & 1.996 & 3.854 & 1.996 \\
			96 & 2.000 & 1.999 & 3.360 & 2.772 \\
			
	\end{tabular}}

\end{table}
We observe that in this case the $\frac{1}{2}$--approximate algorithms deliver aggregates with better quality with respect to the optimal matching algorithm. These results should be compared with the ones in Section~\ref{sec:constant_coefficient_diffusion} with respect to which we observe the somewhat expected deterioration of the quality of the aggregates.

\subsubsection{Diffusion with random coefficients}\label{sec:random_coefficients_appendix} We consider now a less regular case in which the diffusion coefficient has the form $a(x,y) = 0.1 + \eta(x,y)$, for $\eta(x,y)$ a sampling of a uniform random distribution on the $[0,1]$ interval. As for the previous cases, we evaluate the quality of the attained aggregates by looking at both the $\mu_c$ constants and the relative bounds in Table~\ref{tab:random_poisson_aggregates}. For this case, the heuristic search of the decomposition needed to apply Theorem~\ref{thm:our_convergence_result} for the case in which two steps of pairwise matching were used did not succeed. Therefore, only the ``\emph{true}'' constant could be computed.
\begin{table}[htpb]
	\centering
	\caption{{Diffusion problem with random diffusion coefficient. Comparison of the bound in Theorem~\ref{thm:our_convergence_result} with true value of $\mu_c$ in~\eqref{eq:finalcrate}. Aggregates obtained with the weight vector $\mathbf{w}=(1,1,\ldots,1)^T$, and the different matching algorithms for $\ell=1$ pairwise matching steps for a single instance.}}
	\label{tab:random_poisson_aggregates}
	\subfloat[\texttt{HSL\_MC64} -- exact matching]{
		{\begin{tabular}{lcccc}
				n & bound & $\mu_c^{-1}$ \\
				\midrule
				12 & 2.606 & 1.564 \\
				24 & 2.368 & 1.506 \\
				48 & 2.190 & 1.667 \\
				96 & 2.081 & 1.716 \\
	\end{tabular}}}\hfil
	\subfloat[\texttt{PREIS} -- $\frac{1}{2}$--~approximate matching]{
		{\begin{tabular}{lcccc}
				n & bound & $\mu_c^{-1}$ \\
				\midrule
				12 & 2.694 & 1.562 \\
				24 & 2.698 & 1.726 \\
				48 & 2.474 & 1.970 \\
				96 & 2.424 & 2.062 \\
	\end{tabular}}}\hfil
	
	\subfloat[\texttt{AUCTION} -- $\frac{1}{2}$--~approximate matching]{
		{\begin{tabular}{lcccc}
				n & bound & $\mu_c^{-1}$ \\
				\midrule
				12 & 2.752 & 1.450 \\
				24 & 2.247 & 1.731 \\
				48 & 2.179 & 1.629 \\
				96 & 2.134 & 1.693 \\
	\end{tabular}}}\hfil
	\subfloat[\texttt{SUITOR} -- $\frac{1}{2}$--~approximate matching]{
		{\begin{tabular}{lcccc}
				n & bound & $\mu_c^{-1}$ \\
				\midrule
				12 & 2.524 & 1.477 \\
				24 & 2.318 & 1.794 \\
				48 & 2.079 & 1.654 \\
				96 & 2.181 & 1.839 \\
	\end{tabular}}}

\end{table}
{Thus, to have an overview of the mean behavior over several rounds, we consider the box-plots in Figures~\ref{fig:1d_aggregate} and~\ref{fig:2d_aggregate}. Each box represents the $\mu_c^{-1}$ constant over one hundred different samplings of the coefficient function for the discrete Laplacian matrix.}
\begin{figure}[htbp]
	\centering
	\subfloat[$\ell=1$\label{fig:1d_aggregate}]{\includegraphics[width=0.48\columnwidth]{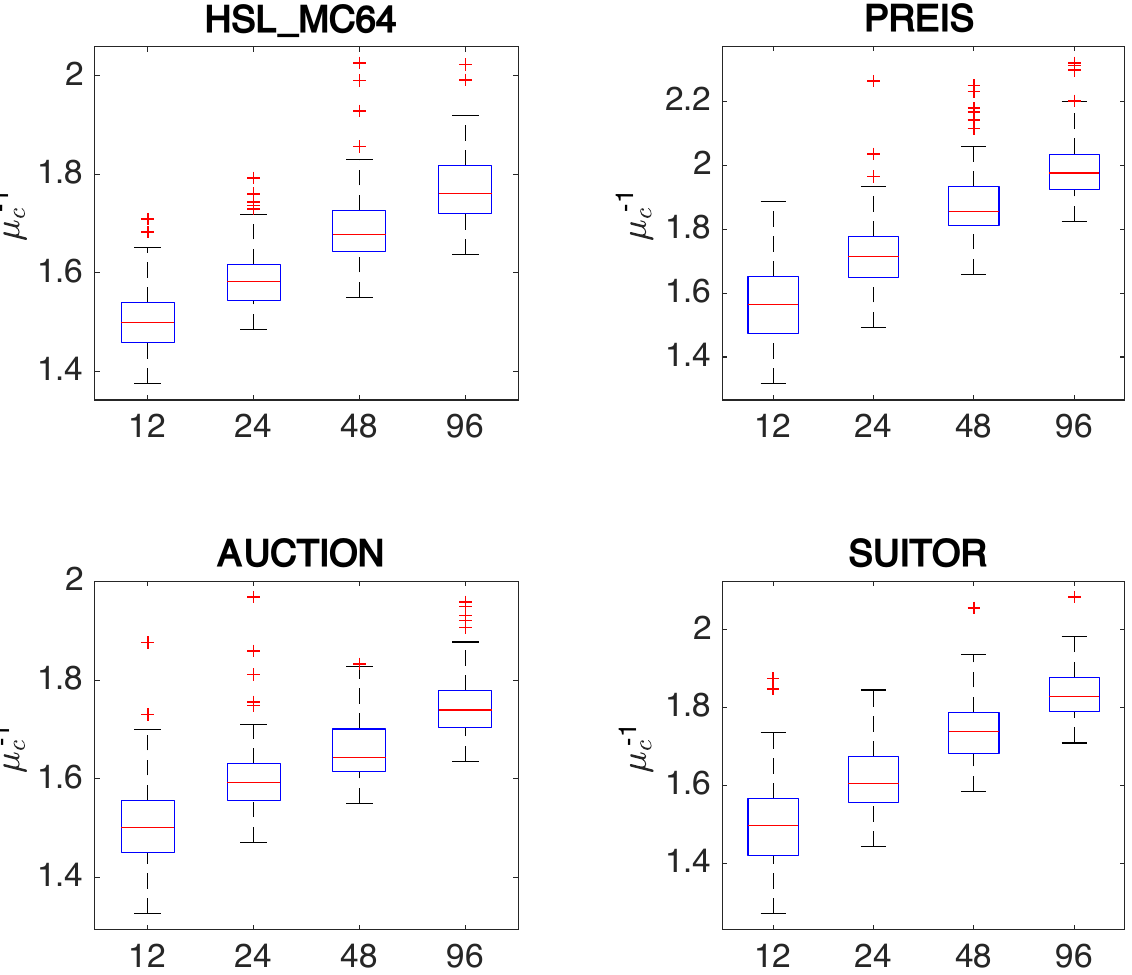}}\hfil
	\subfloat[$\ell=2$\label{fig:2d_aggregate}]{\includegraphics[width=0.48\columnwidth]{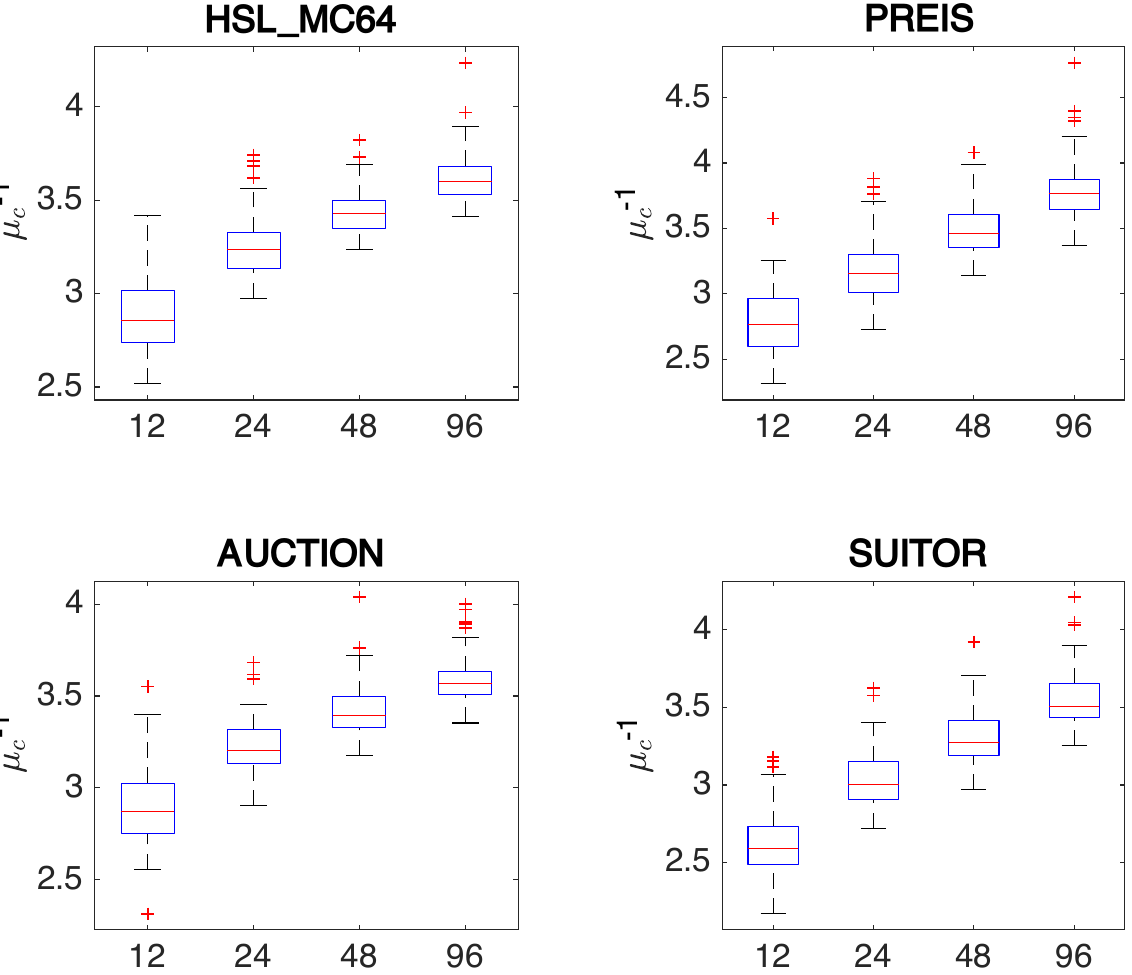}}
	\caption{{Diffusion problem with random diffusion coefficient. Box-plots for the $\mu_c$ constants over one-hundred random samples.}}
\end{figure}
{For all the instances the values of the quality constant for the aggregates behave consistently for both $\ell=1$ and $\ell=2$ steps of matching. Moreover, the outliers are limited in number and value. If we compare these results with the one for the cases with the anisotropies in Table~\ref{tab:yanisotropy_poisson_aggregates} we observe also that the overall effect is indeed comparable.}

\subsection{Selecting the weight vector}

This appendix expands on the experiments discussed in Section~\ref{sec:selecting_weight_vectors}. We focus again on the problem of computing a good weight vector $\mathbf{w}$ for the aggregation procedure.

\subsubsection*{Random weight}

For the two auxiliary test cases discussed in Appendix~\ref{sec:appendix_computing_muc} we consider also the problem of finding a good $\mathbf{w}$ vector. 
From the results in Figure~\ref{fig:refinement_from_random_appendix} we observe again that a random initial guess without any refinement is an inferior choice, and we need several refinement steps to obtain constants $\mu_c$ that are comparable with the ones we have seen in Section~\ref{sec:computing_the_muc_constant}. Furthermore, in the jumping coefficient's case, the choice is so poor that even iterating on it does not deliver a good enough result.
\begin{figure}[htbp]
	\subfloat[Diffusion problem with piece-wise regular coefficients jumping between two values]{\includegraphics[width=\columnwidth]{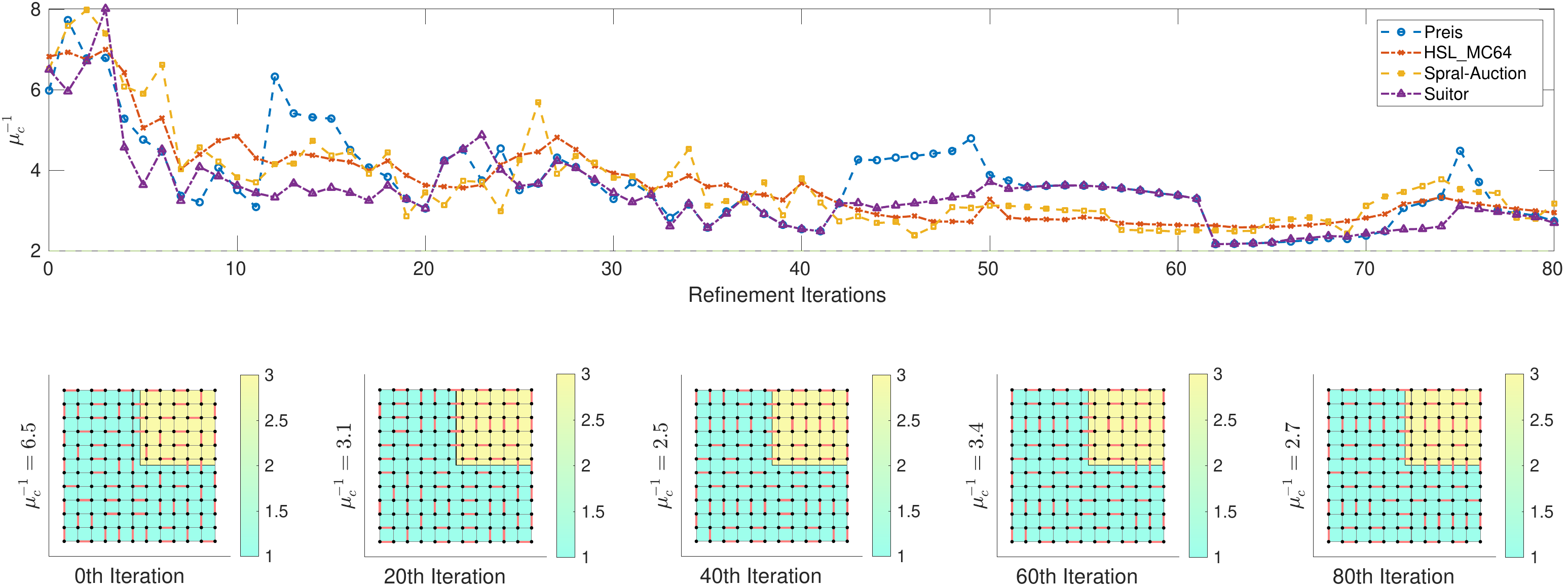}}
	
	\subfloat[Diffusion problem with random diffusion coefficient]{\includegraphics[width=\columnwidth]{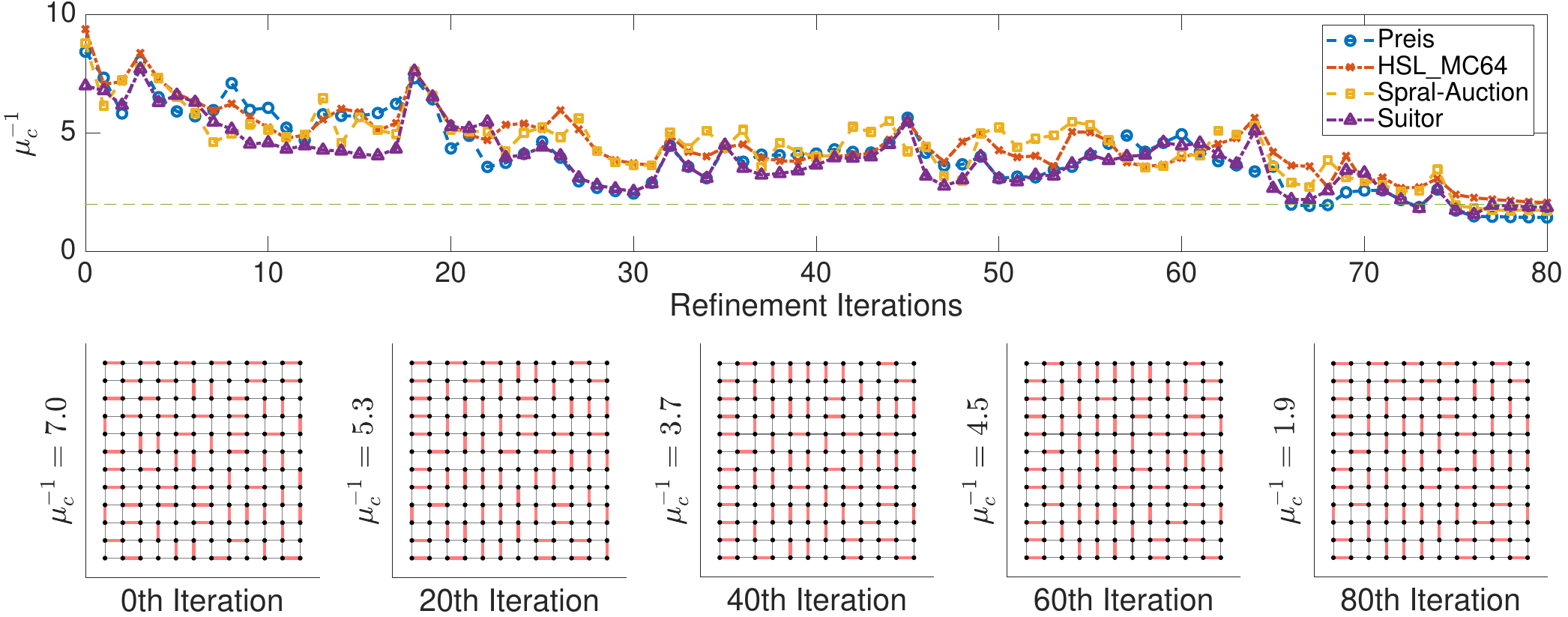}}

	\caption{Refinement of the weight vector starting from a random guess, and using the $\ell_1$--Jacobi smoother. We report a graph containing the $\mu_c^{-1}$ constant up to 80 refinement steps for a single sweep of pairwise aggregation. The depicted aggregates are the ones obtained with the \texttt{SUITOR} algorithm.}
	\label{fig:refinement_from_random_appendix}
\end{figure}
To make a comparison Figure~\ref{fig:refinement_from_random_appendix} should be looked alongside Figure~\ref{fig:refinement_from_random}. With respect to those results we actually need more than 80 iterations to bring the $\mu_c^{-1}$ constant below the value $2.0$.

\subsubsection*{Refined uniform weight}

We collect here the results for the refinements obtained starting with the $\mathbf{w}=(1,1,\ldots,1)^T$ vector for the two test cases discussed in the appendix. The results are given in Figure~\ref{fig:refinement_from_unit_appendix} and should be looked at alongside the ones in Figure~\ref{fig:refinement_from_unit}.
\begin{figure}[htbp]
    \centering

	\subfloat[Diffusion problem with piece-wise regular coefficients jumping between two values]{\includegraphics[width=\columnwidth]{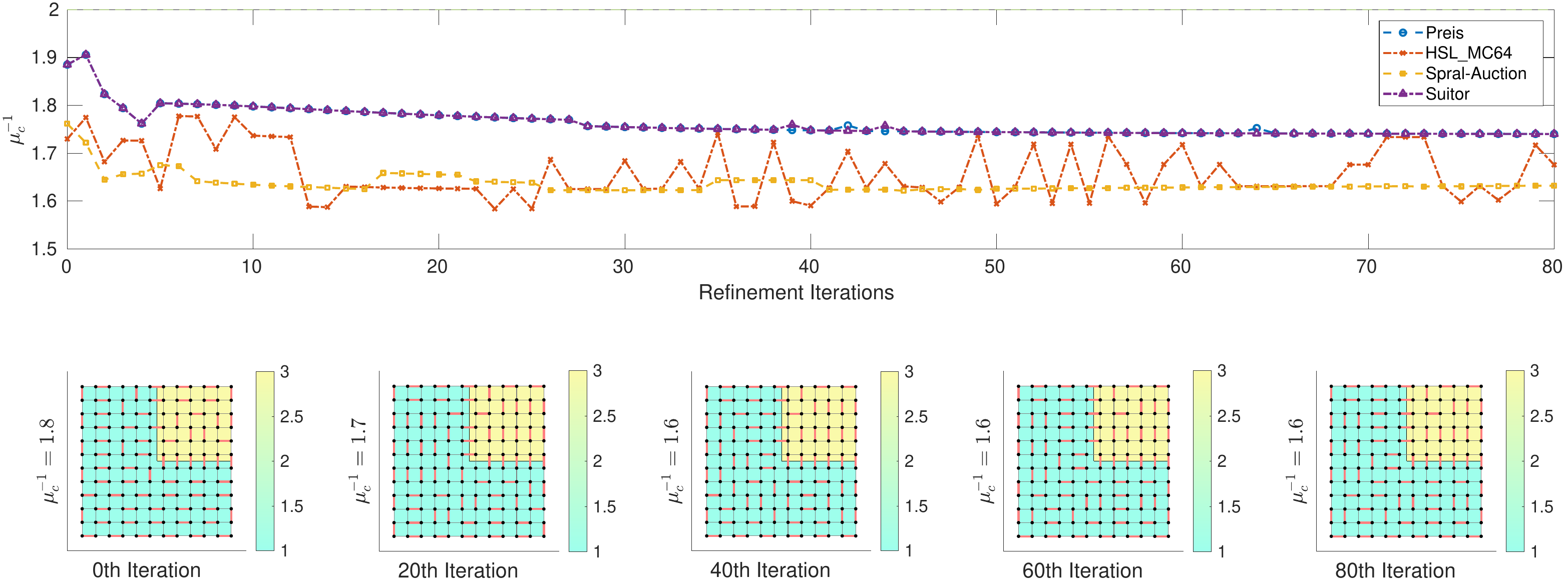}}

	\subfloat[Diffusion problem with random diffusion coefficient]{\includegraphics[width=\columnwidth]{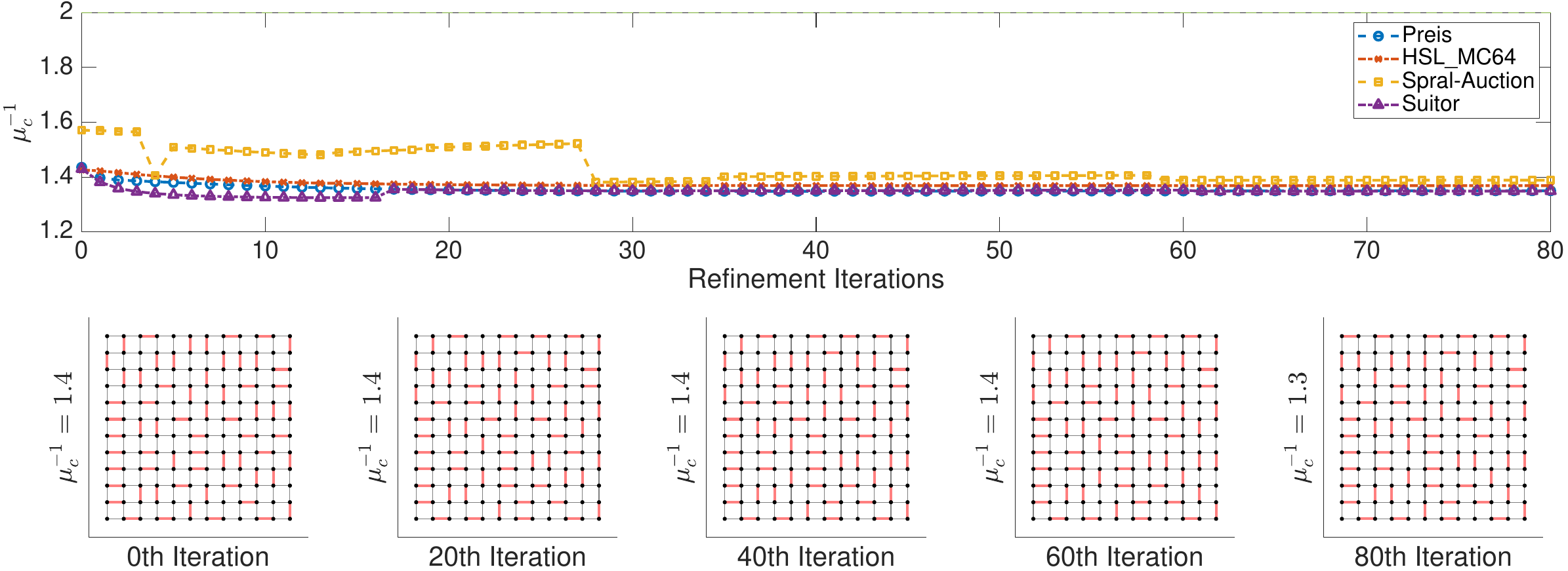}}

	\caption{Refinement of the weight vector starting from the all one guess, and using the $\ell_1$--Jacobi smoother. We report a graph containing the $\mu_c^{-1}$ constant up to 80 refinement steps for a single sweep of pairwise aggregation. The depicted aggregates are the ones obtained with the \texttt{AUCTION} algorithm.}
	\label{fig:refinement_from_unit_appendix}
\end{figure}
Also in this case the behavior is analogous, and already a few iterations of the select smoother are able to reduce the $\mu_c^{-1}$ constant below the value $2.0$ as it was happening for the other cases. We stress also that even on a regular grid, the displacement of these aggregates for the coefficient described here is nontrivial.   

\subsubsection*{The eigenvector weight}

We collect here the results obtained using the eigenvector relative to the smallest eigenvalue as weight vector $\mathbf{w}$. We complete with Table~\ref{tab:eigenvector_choice_appendix} the information we gave in Table~\ref{tab:eigenvector_choice} with the additional cases considered here, i.e., the diffusion with discontinuity in the coefficients and the random diffusion coefficients.
\begin{table}[htb]
	\centering
	\caption{Constants $\mu_c^{-1}$ obtained by using as weight vector $\mathbf{w}$ the eigenvector relative to the smallest eigenvalue as suggested by Proposition~\ref{pro:at-least-is-convergent}.}
	\label{tab:eigenvector_choice_appendix}
	\begin{tabular}{ccccccccc}
		
		& \multicolumn{2}{c}{jumping coeff.s} & \multicolumn{2}{c}{random coeff.s} \\
		\cmidrule(l{2pt}r{2pt}){1-3}\cmidrule(l{2pt}r{2pt}){4-5}
		n  & $\ell = 1$ & $\ell = 2$ & $\ell = 1$ & $\ell = 2$ \\
		\cmidrule(l{2pt}r{2pt}){1-3}\cmidrule(l{2pt}r{2pt}){4-5}
		12  & 1.604 & 2.227 &  1.6612 & 2.2900 \\
		24  & 1.881 & 2.582 & 1.5209 & 2.4982 \\
		48  & 1.974 & 2.971 & 1.5537 & 2.9712 \\
		96  & 1.994 & 2.965 & 1.4373 & 2.3346\\
	\end{tabular}	
\end{table}
We can depict also in this case the aggregates obtained by this method in Figure~\ref{fig:eigenvector_choice_appendix} and that can be analyzed alongside the ones in Figure~\ref{fig:eigenvector_choice}. 
\begin{figure}[htbp]
\centering
	\subfloat[Diffusion problem with piece-wise regular coefficients jumping between two values]{
		\includegraphics[width=0.23\columnwidth]{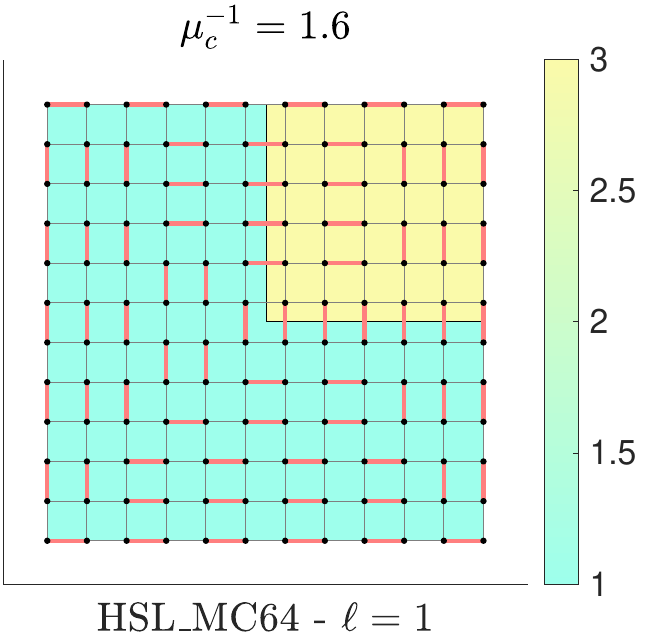}
		\includegraphics[width=0.23\columnwidth]{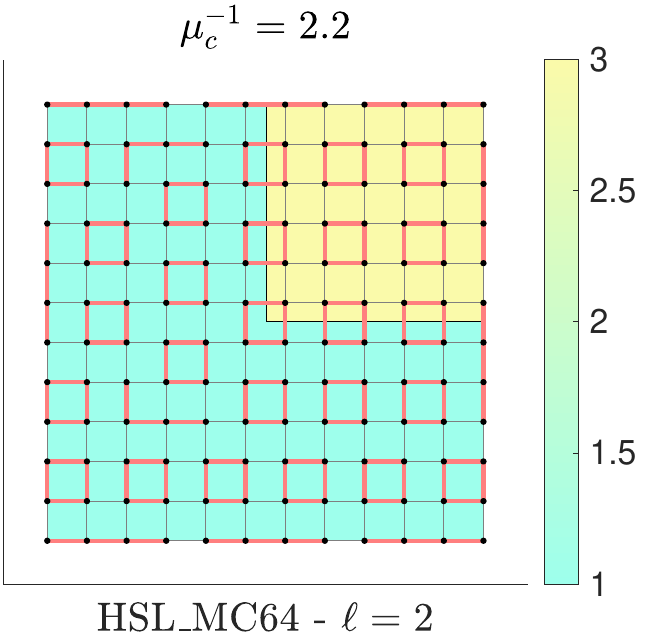}
	}\hfil
	\subfloat[{Diffusion problem with random diffusion coefficient}]{
		\includegraphics[width=0.2\columnwidth]{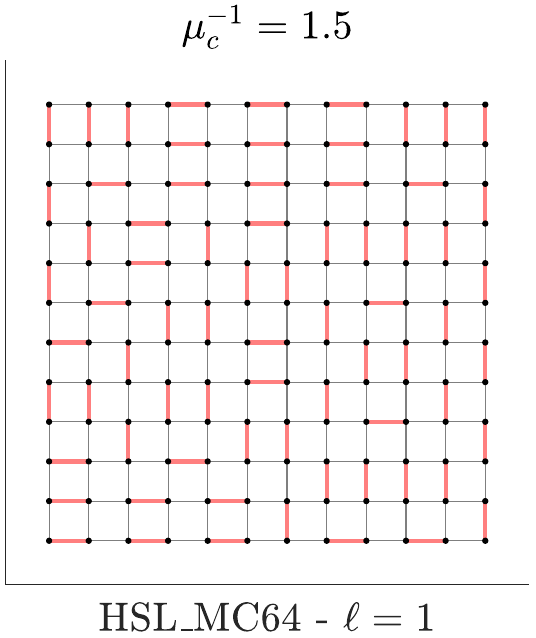}
		\includegraphics[width=0.2\columnwidth]{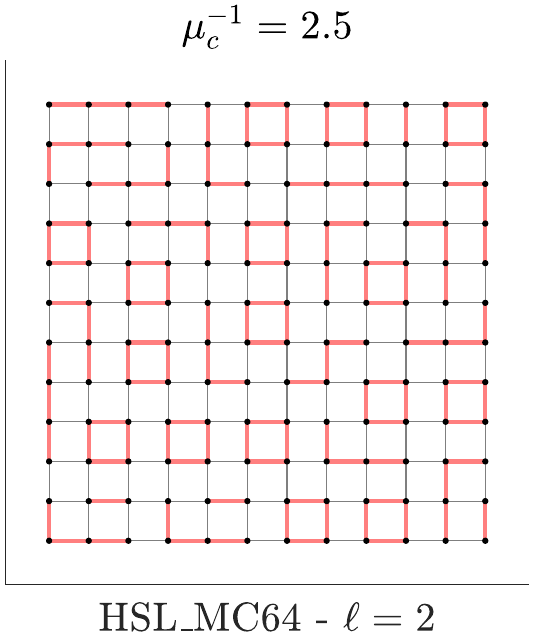}
	}
	\caption{Aggregates obtained by using as weight vector $\mathbf{w}$ the eigenvector associated with the smallest eigenvalue as suggested by Proposition~\ref{pro:at-least-is-convergent}}
	\label{fig:eigenvector_choice_appendix}
\end{figure}
If we focus particularly on the case with $\ell=2$ sweeps of matching we can observe that again the algorithm has found on its own nontrivial aggregates to try to automatically accommodate the irregularities in the coefficients of the matrix.

\subsubsection{The bootstrap procedure}\label{sec:thebootstrapprocedureresults}

We have observed in the previous section that refining the weight vector by means of a standard stationary method may require a certain number of iterations, i.e., a certain number of matrix-vector products. While in some cases this may be feasible, e.g., if we plan to  reuse the same multigrid hierarchy for (many) different solutions, in other cases we could decide to exploit this larger setup time to achieve more than the simple refinement of the weight vector $\mathbf{w}$.

We have recalled in~\eqref{eq:bootstrapiteration} how we can use the multigrid hierarchy itself to refine the choice of the weight $\mathbf{w}$. The secondary effect of having built this composite solver, as discussed in~\cite[Section~5]{BootCMatch}, is then the possibility of using it as a preconditioner for a Krylov subspace solver. Therefore, the setup cost is compensated by the trade-off between the necessity of obtaining a better weight vector $\mathbf{w}$, and having an efficient preconditioner with a user's defined convergence rate.

As an application, we consider here only the case of the diffusion with jumping coefficients from Section~\ref{sec:diffusion_with_jumping_coefficients}. Specifically, we consider the case in which we initialize the bootstrapping procedure with the vector $\mathbf{w}$ obtained after five-step of the stationary $\ell_1$-Jacobi method applied to the uniform, all one vector. We perform $\ell=2$ steps of pairwise aggregation with the \texttt{HSL\_MC64} algorithm, and 4 bootstrap iterations where each AMG operator is applied as a V-cycle and 1 iteration of $\ell_1$-Jacobi is applied for pre/post smoothing. 
\begin{figure}[htb]
	\centering
	\includegraphics[width=\columnwidth]{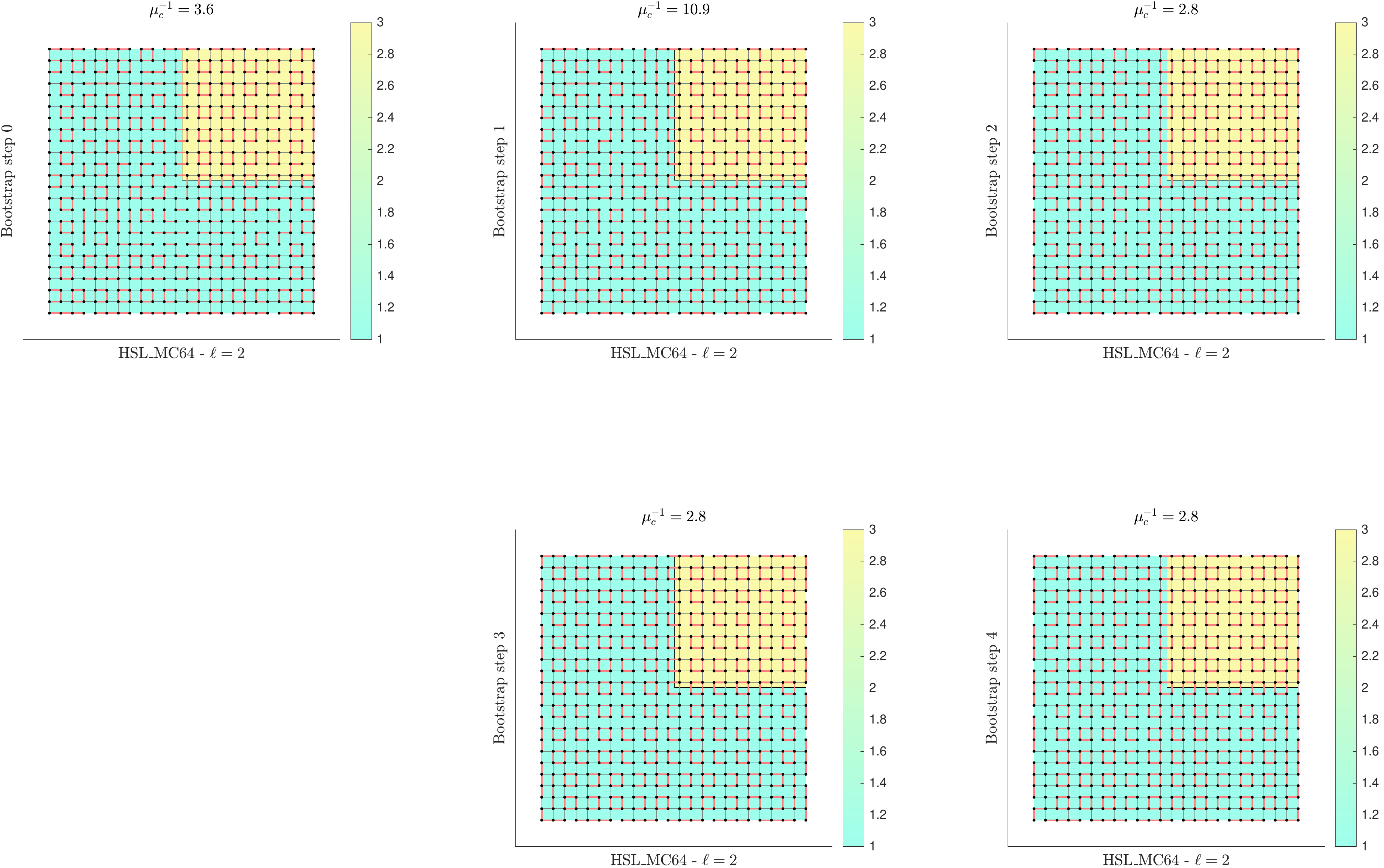}
	\caption{Diffusion problem with piece-wise regular coefficients jumping between two values. Aggregates obtained by using the weight vector $\mathbf{w}$ from 4 consecutive steps of bootstrap iteration. The procedure is initialized with the vector $\mathbf{w}_0$ obtained after five step of the stationary $\ell_1$-Jacobi method applied to the all one vector.}
	\label{fig:bootstrap}
\end{figure}
We reported the obtained aggregates, together with the $\mu_c$ constants, in Figure~\ref{fig:bootstrap}. We observe that consistent with what was happening when using the simple stationary iterative method, the reduction of the $\mu_c$ constant is not monotone; see again Figures~\ref{fig:refinement_from_random}, and~\ref{fig:refinement_from_unit}. Nevertheless, we get aggregates with better quality if compared with the results in Table~\ref{tab:jumping_poisson_aggregates}. Moreover, after the initial oscillation in quality the value of the $\mu_c$ constant, as expected, seems to stabilize; however, even if the overall constant is the same, the aggregates obtained are different. It is worth noting that, when using the bootstrap procedure to produce a preconditioner,  the fact that the aggregates stop improving after a certain number of steps does not necessarily imply that the convergence ratio of the product hierarchy also stops improving. Indeed, the convergence ratio is guaranteed to become better for each newly added component; see the analysis in~\cite[Section~5]{BootCMatch}.

\subsection{Compatible relaxation}

We complete the data from Table~\ref{tab:convergence_ratio} using the compatible relaxation principle discussed in Section~\ref{sec:compatiblerel} with the test cases discussed in this Appendix. Namely, the case with random coefficients and the case of diffusion with a discontinuity jumping between two values. From the comparison of the results collected in Table~\ref{tab:convergence_ratio_appendix} with the ones in Table~\ref{tab:convergence_ratio} we observe again an analogous behavior. 
\begin{table}[htbp]
	\centering
	\caption{Convergence ratio $\rho_f$ of the compatible relaxation scheme~\eqref{eq:cr_convratio} for all the test problems. The coarse space is built from a single step of all the matching algorithms from Section~\ref{sec:matching_algorithms} with weight vector choice $\mathbf{w} = (1,1,\ldots,1)^T$, and no refinement iterations.}
	\label{tab:convergence_ratio_appendix} 
	\subfloat[Diffusion problem with piece-wise regular coefficients jumping between two values]{
		\begin{tabular}{ccccc}
			
			$n$   & \texttt{HSL\_MC64} &\texttt{PREIS} & \texttt{AUCTION} & \texttt{SUITOR} \\
			\cmidrule(l{2pt}r{2pt}){1-1}\cmidrule(l{2pt}r{2pt}){2-2}\cmidrule(l{2pt}r{2pt}){3-3}\cmidrule(l{2pt}r{2pt}){4-4}\cmidrule(l{2pt}r{2pt}){5-5} 
			12 & 0.765 & 0.793 & 0.750 & 0.793 \\
			24 & 0.816 & 0.824 & 0.805 & 0.824 \\
			48 & 0.826 & 0.831 & 0.829 & 0.831 \\
			96 & 0.832 & 0.833 & 0.832 & 0.833 \\
	\end{tabular}}
	
		\subfloat[{Diffusion problem with random diffusion coefficient}]{
		\begin{tabular}{ccccc}
			
			$n$   & \texttt{HSL\_MC64} &\texttt{PREIS} & \texttt{AUCTION} & \texttt{SUITOR} \\
			\cmidrule(l{2pt}r{2pt}){1-1}\cmidrule(l{2pt}r{2pt}){2-2}\cmidrule(l{2pt}r{2pt}){3-3}\cmidrule(l{2pt}r{2pt}){4-4}\cmidrule(l{2pt}r{2pt}){5-5} 
			12 & 0.779 & 0.784 & 0.780 & 0.808 \\
			24 & 0.828 & 0.789 & 0.811 & 0.797 \\
			48 & 0.820 & 0.819 & 0.833 & 0.824 \\
			96 & 0.848 & 0.841 & 0.854 & 0.828 \\
	\end{tabular}}
\end{table}
Furthermore, we stress again that the ``collective'' number given by the compatible relaxation principle blends together the behavior of the smoothing and coarsening strategy. Thus the difference we have observed in the quality of the aggregation procedure are harder to interpret by using only this measure.

\end{appendices}

\bibliography{multigridconvergence}%

\end{document}